\documentclass[11pt]{article}
\usepackage{amssymb}
\usepackage{amsbsy}
\usepackage[latin1]{inputenc}
\usepackage{amsthm}
\usepackage[dvips]{graphicx}
\usepackage{graphicx} 
\usepackage{subfigure}
\usepackage{pst-eucl}
\usepackage[latin1]{inputenc}
\usepackage[english]{babel}
\usepackage{amsmath,amssymb,graphics,mathrsfs}
\usepackage{amsmath,amssymb,latexsym}
\usepackage{graphicx,color}
\usepackage[T1]{fontenc}
\usepackage[active]{srcltx}
\usepackage{multicol}
\usepackage[latin1]{inputenc}
\usepackage{pst-all}
\usepackage{enumerate}
\usepackage{pstricks}
\usepackage{pstricks-add}
\usepackage{setspace}
\usepackage{soul}
\usepackage{cancel}
\usepackage{nonfloat}
\usepackage[margin=10pt,font=footnotesize,labelfont=bf,labelsep=endash]{caption}
\usepackage[left=4cm,top=3cm,right=2.4cm,bottom=3.2cm]{geometry}
\usepackage[colorlinks=true,citecolor=red,linkcolor=blue,urlcolor=RubineRed,pdfpagetransition=Blinds,pdftoolbar=false,pdfmenubar=false]{hyperref}

\newcommand{\R}{\hbox{\rm I \kern -5pt R}}     % simbolo de Reales
\newcommand{\p} {\hbox{\rm I \kern -5pt P}}
\def\x  {\boldsymbol x} %{\textbf x}
\def\n  {\boldsymbol n} %{\textbf n}
\def\W        {{\textbf{W}}}

\def\H   {\boldsymbol H} % {\textbf H}%    {{\textbf{H}}}
\def\L    {\boldsymbol L} % {\textbf L}%  {{\textbf{ L}}}

\newtheorem{prop}{Proposition}[section]
\newtheorem{defi}[prop]{Definition}%[chapter]
\newtheorem{tma}[prop]{Theorem}%[chapter]
\newtheorem{cor}[prop]{Corollary}%[chapter]
\newtheorem{obs}[prop]{Remark}%[chapter]
\newtheorem{lem}[prop]{Lemma}%[chapter]
\allowdisplaybreaks

\begin{document}

\title{Study of a chemo-repulsion model with quadratic production. Part I: Analysis of the continuous problem and time-discrete numerical schemes}
\author{F.~Guill\'en-Gonz\'alez\thanks{Dpto. Ecuaciones Diferenciales y An\'alisis Num\'erico and IMUS, 
Universidad de Sevilla, Facultad de Matem\'aticas, C/ Tarfia, S/N, 41012 Sevilla (SPAIN). Email: guillen@us.es, angeles@us.es},
M.A.~Rodr\'{\i}guez-Bellido$^*$~and
 D.A.~Rueda-G\'omez$^*$\thanks{Escuela de Matem\'aticas, Universidad Industrial de Santander, A.A. 678, Bucaramanga (COLOMBIA). Email:  diaruego@uis.edu.co}}

\date{}
\maketitle

\begin{abstract}
We consider a chemo-repulsion model with quadratic production in a bounded domain.
% which is a nonlinear parabolic system coupling two variables: the cell density 
%$u(\x,t)$ 
%and the chemical concentration.
%$v(\x,t)$. 
Firstly, we obtain global in time weak solutions, and give a regularity criterion (which is  satisfied  for $1D$ and $2D$ domains) to deduce uniqueness and global regularity. After, we study  two cell-conservative and unconditionally energy-stable first-order  time schemes: a (nonlinear and positive) Backward Euler scheme  and a linearized coupled version,  proving solvability, convergence towards weak solutions  and error estimates. In particular, the linear scheme does not preserve positivity and the uniqueness  of the nonlinear scheme is proved assuming small  time step  with respect to a strong norm of the discrete solution. 
 This hypothesis is reduced to  small  time step in $nD$ domains ($n\le 2$)  where global in time strong estimates  are proved.
Finally, we show the behavior of the schemes through some numerical simulations.
\end{abstract}



\section{Introduction}

Chemotaxis is understood as the biological process of the movement of living organisms in res\-pon\-se to a chemical stimulus which can be given towards a higher (attractive) or lower (repulsive) concentration of a chemical substance. At the same time, the presence of living organisms can produce or consume chemical substance. A  chemo-repulsive  model with generic production can be given by the following parabolic PDE's system:
\begin{equation}   \label{modelf000}
\left\{
\begin{array}
[c]{lll}%
\partial_t u - \Delta u = \nabla\cdot (u\nabla v)\quad \mbox{in}\ \Omega,\ t>0,\\
\partial_t v - \Delta v + v =  f(u)
 \quad \mbox{in}\ \Omega,\ t>0,
% \\
%\frac{\partial u}{\partial \n}=\frac{\partial v}{\partial \n}=0\ \ \mbox{on}\ \partial\Omega,\ t>0,
%\\ 
%u(\x ,0)=u_0(\x )\geq 0,\ v(\x ,0)=v_0(
%\x )\geq 0\ \ \mbox{in}\ \Omega,
\end{array}
\right. \end{equation}
where $\Omega\subset \mathbb{R}^d$, $d=2,3$, is an open bounded domain with boundary $\partial \Omega$. The unknowns for this model are $u(\x , t) \geq 0$, the cell density, and $v(\x , t) \geq 0$, the chemical concentration. Moreover, $f(u)$ is a  function which is nonnegative for any $u\geq 0$. In this paper, we consider the particular case in which the production term is quadratic, that is $f(u)=u^2$, and then we focus on the following initial-boundary value problem:
\begin{equation}  \label{modelf00}
\left\{
\begin{array}
[c]{lll}%
\partial_t u - \Delta u = \nabla\cdot (u\nabla v)\quad \mbox{in}\ \Omega,\ t>0,\\
\partial_t v - \Delta v + v =  u^2
 \quad \mbox{in}\  \Omega,\ t>0,\\
 \displaystyle
\frac{\partial u}{\partial \n}=\frac{\partial v}{\partial \n}=0\ \ \mbox{on}\ \partial\Omega,\ t>0,\\
u(\x ,0)=u_0(\x )\geq 0,\ v(\x ,0)=v_0(
\x )\geq 0\ \ \mbox{in}\ \Omega.
\end{array}
\right. \end{equation}

In the case of linear production, that is $f(u)=u$, model (\ref{modelf000}) is well-posed in the following sense (\cite{Cristian}): there exist global in time weak solutions (based on an energy inequality) and, for $1D$ and $2D$ domains, there exists a unique global in time strong solution.
However, it is interesting to analyze the case of nonlinear  production which is justified biologicaly when the
process of signal production through cells need no longer dependence on the population density in a linear manner, for instance when saturation effects at large (or short) densities are taken into account (see \cite{Win1} and references therein). In particular, we focus in model (\ref{modelf00}) because the quadratic production term allows to control an energy in $L^2(\Omega)$-norm for $u$ (see (\ref{wsd})-(\ref{eneruva})), which is very useful for performing numerical analysis; and, as far as we know,  there are not works studying problem  (\ref{modelf00}).\\

The aim of this paper is twofold; first to obtain existence of global in time weak solution, which is unique and  regular in $1D$ and $2D$ domains, and second to design two energy-stable time-discrete  schemes approaching model (\ref{modelf00}). Some other properties like solvability, positivity, convergence towards weak solutions  and error estimates will be studied.  \\

There are some works about numerical analysis for chemotaxis models. For instance, for the Keller-Segel system (i.e.~with chemo-attraction and linear production), in \cite{Filbet} it was studied the convergence of a finite volume scheme. Some error estimates were proved for a fully discrete discontinuous Finite Element (FE) method in \cite{Eps}. In \cite{Zuhr}, the nonnegativity of numerical solutions to a generalized Keller-Segel model was analyzed. A mixed FE approximation was studied in \cite{Marrocco}; and in \cite{Saito1,Saito2}, were proved error estimates for a conservative FE approximation. In \cite{C5:FMD4}, the authors studied unconditionally energy stable fully discrete FE schemes for a chemo-repulsion model with linear production. The convergence of a combined finite volume-nonconforming FE scheme was studied in \cite{CST}, in the case where the chemotaxis occurs in heterogeneous medium. In \cite{Zhang1}, the convergence of a mass-conservative characteristic FE method was studied and some error estimates were proved. \\

In order to develop our analysis, we reformulate (\ref{modelf00})  by introducing the auxiliary variable ${\boldsymbol\sigma}=\nabla v$. 
Then,  the model (\ref{modelf00}) is rewritten as follows:
\begin{equation} \left\{
\begin{array}
[c]{lll}%
\partial_t u - \nabla\cdot (\nabla u) = \nabla\cdot (u{\boldsymbol \sigma})\quad \mbox{in}\ \Omega,\ t>0,\\
\partial_t {\boldsymbol \sigma} - \nabla(\nabla\cdot {\boldsymbol \sigma}) +\mbox{rot}(\mbox{rot }{\boldsymbol \sigma}) + {\boldsymbol \sigma}  = \nabla(u^2) \quad \mbox{in}\ \Omega,\ t>0,\\
\displaystyle
\frac{\partial u}{\partial \n}=0\quad \mbox{on}\ \partial\Omega,\ t>0,\\
{\boldsymbol \sigma}\cdot \n=0, \ \ \left[\mbox{rot }{\boldsymbol \sigma} \times \n\right]_{tang}=0 \quad \mbox{on}\ \partial\Omega,\ t>0,\\
u(\x ,0)=u_0(\x )\geq 0,\ {\boldsymbol \sigma}(\x ,0)=\nabla v_0(
\x )\quad \mbox{in}\ \Omega,
\end{array}
\right.  \label{modelf01}
\end{equation}
where (\ref{modelf01})$_2$ is obtained by applying the gradient operator to equation (\ref{modelf00})$_2$  and adding the term $\mbox{rot}(\mbox{rot }{\boldsymbol \sigma})$ using the fact that $\mbox{rot }{\boldsymbol \sigma}=\mbox{rot}(\nabla v)=0$. 
Once solved (\ref{modelf01}),  $v$ can be recovered 
from $u^2$ solving
\begin{equation}  \label{modelf01eqv}
\left\{
\begin{array}
[c]{lll}%
\partial_t v -\Delta v + v = u^2  \quad \mbox{in}\ \Omega,\ t>0,\\
\displaystyle
\frac{\partial v}{\partial \n}=0\quad\mbox{on}\ \partial\Omega,\ t>0,\\
 v(\x ,0)=v_0(\x )\geq 0\quad \mbox{in}\ \Omega.
\end{array}
\right. 
\end{equation}
The use of the variable ${\boldsymbol\sigma}$ let  simplify the notation throughout the paper. 
On the other hand,  the role of ${\boldsymbol\sigma}$ will be different for fully discrete schemes 
 where problem (\ref{modelf00}) and (\ref{modelf01})-(\ref{modelf01eqv}) are not equivalent. 
In fact, as will be analyzed in the forthcoming paper \cite{FMD2}, by using a FE spatial discretization it will be very convenient  to use the variables $(u,{\boldsymbol\sigma})$ in order to obtain an unconditionally energy-stable scheme. \\

The outline of this paper is as follows. In Section 2,  the notation and  preliminary results  are given.  
In Section 3,  the continuous problem (\ref{modelf00}) is analyzed, defining the concept of global in time weak solutions, 
and obtaining global in time strong regularity of the model by assuming  the regularity criterion (\ref{def-strongA}), which is satisfied for $1D$ and $2D$ domains. 
In Section 4, the Backward Euler scheme corresponding to problem  (\ref{modelf01})-(\ref{modelf01eqv}) is analyzed, 
including cell-conservation, unconditional energy-stability, solvability, positivity and error estimates of the scheme. 
In particular, uniqueness of solution of the scheme is proved under a hypothesis that assumes small time step  multiplied by a strong norm of the scheme (the discrete version of (\ref{def-strongA})), which can be simplified in the case of $1D$ and $2D$ domains where strong estimates are obtained for the scheme. 
Moreover, the existence of weak solutions of model (\ref{modelf00}) is proved throughout the convergence of this scheme when the time step  goes to $0$. 
In Section 5, a linearized coupled scheme for model (\ref{modelf01})-(\ref{modelf01eqv}) is proposed, and again some properties of this linear scheme are analyzed, comparing to the previous nonlinear scheme. Finally, in Section 6,  some numerical simulations using FE approximations associated to both time schemes are shown, in order to verify numerically the theoretical results obtained in terms of positivity and unconditional energy-stability.

\section{Notations and preliminary results}

 Recall some functional spaces which will be used throughout this paper. We consider the Lebesgue spaces $L^p(\Omega)$, 
$1\leq p\leq \infty$, and the Sobolev spaces $H^m(\Omega)$, by denoting  their norms by $\Vert\cdot \Vert_{L^p}$ and $\Vert\cdot\Vert_{m}$, respectively. In particular,  the $L^2(\Omega)$-norm will be represented by $\Vert
\cdot\Vert_0$. Corresponding Sobolev spaces of vector valued functions will be denoted by $\H^1(\Omega)$, $\L^2(\Omega),$ and so on; and we consider the vectorial space 
$$
\H^{1}_{\sigma}(\Omega):=\{\boldsymbol \sigma\in \H^{1}(\Omega); \boldsymbol \sigma\cdot \n=0 \mbox{ on } \partial\Omega\}.
$$ 
%and $H_*^1(\Omega):=\{h\in H^1(\Omega): \int_\Omega h=0\}$.  
%From now on,  we will consider that 
The following  norms are equivalent in $H^1(\Omega)$, $H^2(\Omega)$   (\cite{necas}) and 
 $\H_{\sigma}^1(\Omega)$ (\cite[Corollary 3.5]{Nour}) 
respectively:
\begin{equation}\label{H1poin}
\Vert u \Vert_{1}^2\sim \Vert \nabla u\Vert_{0}^2 + \left( \int_\Omega u\right)^2 \quad \forall u\in H^1(\Omega),
\end{equation}
\begin{equation}\label{H2poin}
\Vert v \Vert_{2}^2\sim\Vert \Delta v\Vert_{0}^2 + \Vert  v\Vert_{0}^2 \quad \forall v\in H^2(\Omega),
\end{equation}
\begin{equation}\label{H1div}
\Vert {\boldsymbol\sigma} \Vert_{1}^2\sim\Vert {\boldsymbol\sigma}\Vert_{0}^2 + \Vert \mbox{rot }{\boldsymbol\sigma}\Vert_0^2 + \Vert \nabla \cdot {\boldsymbol\sigma}\Vert_0^2 \quad \forall {\boldsymbol\sigma}\in \H^{1}_{\sigma}(\Omega).
\end{equation}
In particular, (\ref{H1div}) implies that, for all $\boldsymbol\sigma=\nabla v\in \H^{1}_{\sigma}(\Omega)$,  the following  norms are equivalent
\begin{equation}\label{H1divGrad}
\Vert\nabla v\Vert_{1}^2\sim\Vert \nabla v\Vert_{0}^2  + \Vert \Delta v\Vert_0^2.
\end{equation}
 If $Z$ is  a Banach space, then $Z'$ will denote its topological dual. Moreover, the letters $C,K$ will denote different positive
constants always independent of the time step. The following linear elliptic operators are introduced
\begin{equation}\label{B001} 
\widehat A u =g \quad \Longleftrightarrow \quad 
\left\{\begin{array}{l}
-\Delta u+ \int_\Omega u = g \ \mbox{ in } \Omega,\\
\displaystyle\frac{\partial u}{\partial\n}=0 \ \mbox{ on }  \partial\Omega, \vspace{0,2 cm}\\
\end{array}\right.
\end{equation} 
\begin{equation}\label{B101} 
A v =g \quad \Longleftrightarrow \quad 
\left\{\begin{array}{l}
-\Delta v+ v = g \ \mbox{ in } \Omega,\\
\displaystyle\frac{\partial v}{\partial\n}=0 \ \mbox{ on }  \partial\Omega, \vspace{0,2 cm}\\
\end{array}\right.
\end{equation} 
and 
\begin{equation} \label{B201}
B{\boldsymbol \sigma} ={\textbf {\textit f}} \quad \Longleftrightarrow \quad 
\left\{\begin{array}{l}
-\nabla(\nabla\cdot {\boldsymbol \sigma}) + \mbox{rot(rot }{\boldsymbol \sigma}\mbox{)} + {\boldsymbol \sigma}  = {\textbf {\textit f}} \ \mbox{ in }  \Omega,\\
{\boldsymbol \sigma}\cdot \n=0, \ \ \left[\mbox{rot }{\boldsymbol \sigma} \times \n\right]_{tang}={\bf 0} \ \mbox{ on }  \partial\Omega.
\end{array}\right. 
\end{equation}
The corresponding variational forms are given by $\widehat A,A:H^1(\Omega)\rightarrow H^1(\Omega)'$ and $B:\H^{1}_{\sigma}(\Omega)\rightarrow \H^{1}_{\sigma}(\Omega)'$ such that
\begin{equation*} 
\langle \widehat A u,\bar{u}\rangle=(\nabla u, \nabla \bar{u})+\left(\int_\Omega u\right) \left(\int_\Omega\bar{u}\right), \ \ \forall u,\bar{u}\in {H}^{1}(\Omega),
\end{equation*}
\begin{equation*}  
\langle A v,\bar{v}\rangle=(\nabla v, \nabla \bar{v})+(v, \bar{v}), \ \ \forall v,\bar{v}\in {H}^{1}(\Omega),
\end{equation*}
\begin{equation*} 
\langle B{\boldsymbol\sigma},\bar{{\boldsymbol\sigma}}\rangle
=(\nabla \cdot {\boldsymbol\sigma}, \nabla\cdot \bar{\boldsymbol\sigma})
+(\mbox{rot }{\boldsymbol \sigma},\mbox{rot }\bar{\boldsymbol \sigma})
+({\boldsymbol\sigma},\bar{{\boldsymbol\sigma}}),\ \ \forall {\boldsymbol\sigma},\bar{{\boldsymbol\sigma}}\in \H^{1}_{\sigma}(\Omega).
\end{equation*} 
One assumes the $H^2$ and $H^3$-regularity of problems (\ref{B001}) and (\ref{B101}) (see for instance \cite{Demengel2}). Consequently, there exist some constants $C>0$ such that
\begin{equation}\label{H2us}
\Vert u\Vert_{2}\leq  C\Vert \widehat A u\Vert_{0} \ \ \forall u\in H^2(\Omega);
 \quad \Vert v\Vert_{3}\leq C\Vert A v \Vert_{1} \ \ \forall v\in H^3(\Omega).
\end{equation}

If the right hand side of problem (\ref{B201}) is given by ${\textbf {\textit f}}=\nabla h$ with $h\in H^1(\Omega)$, then taking ${\boldsymbol \sigma}=\nabla v$,  the $H^2$-regularity of problem (\ref{B201}) can be proved as follows:
\begin{lem}\label{lemH2s01}
 If ${\textbf {\textit f}}=\nabla h$ with $h\in H^1(\Omega)$, then the solution ${\boldsymbol \sigma}$ of problem (\ref{B201}) belongs to $\H^2(\Omega)$. Moreover, 
\begin{equation}\label{eH201}
\Vert {\boldsymbol \sigma}\Vert_2\leq C\, \Vert \nabla h\Vert_0
\ \big(= C\, \Vert B {\boldsymbol \sigma} \Vert_0 \big).
\end{equation}
\end{lem}
\begin{proof}
%First, we assume that $h \in H^1_*(\Omega)$, hence $\Vert h \Vert_1\leq C\, \Vert \nabla h\Vert_0$. Then, from $H^3$-regularity of problem (\ref{B101}) taking $g=h$, we have that $v\in H^3(\Omega)$ with $-\Delta v+ v = h$ and $\Vert v\Vert_3\leq C\,\Vert h \Vert_1\leq C\, \Vert \nabla h\Vert_0$. Then, taking ${\boldsymbol \sigma}=\nabla v$, we have that ${\boldsymbol \sigma}\in \H^2(\Omega)$ solves (\ref{B201}), and (\ref{eH201}) holds. 
%
%
From $H^3$-regularity of problem (\ref{B101}) taking  the zero mean value function $g=h - \frac{1}{\vert\Omega\vert}\int_\Omega h$,  
we have that $v\in H^3(\Omega)$, 
$ =Av = h-\frac{1}{\vert\Omega\vert}\int_\Omega h $ and 
$\Vert v\Vert_3\leq C\,\Vert h -\frac{1}{\vert\Omega\vert}\int_\Omega h \Vert_1 \leq
% C\, \Vert \nabla (h-\frac{1}{\vert\Omega\vert}\int_\Omega h)\Vert_0 =
 C \Vert \nabla h\Vert_0$. Then, taking ${\boldsymbol \sigma}=\nabla v$, one has that ${\boldsymbol \sigma}\in \H^2(\Omega)$ solves (\ref{B201}) with ${\textbf {\textit f}}= \nabla (h-\frac{1}{\vert\Omega\vert}\int_\Omega h)=\nabla h$, and (\ref{eH201}) holds.
\end{proof}
Along this paper, the following  classical interpolation inequalities will be repeatedly used
\begin{equation}\label{in2D}
\Vert w\Vert_{L^4}\leq C\Vert w\Vert_0^{1/2}\Vert w\Vert_{1}^{1/2} \  \ \forall w\in H^1(\Omega) \ \ \mbox{ (in 2D domains)},
\end{equation}
\begin{equation}\label{in3D}
\Vert w\Vert_{L^3}\leq C\Vert w\Vert_0^{1/2}\Vert w\Vert_{1}^{1/2} \  \ \forall w\in H^1(\Omega) \ \ \mbox{ (in 3D domains)}.
\end{equation}
Finally, in order to obtain uniform in time strong estimates for the continuous problem and the numerical schemes,  the following results will be used (see 
%\cite{He}, 
\cite{Temam} and \cite{Shen}, respectively):
%\begin{lem} \label{tmaD}
%Assume  $\delta,\beta,k>0$ and $z^n\geq 0$ satisfying 
%\begin{equation*}
% \frac{z^{n+1} - z^n}{k} + \delta z^{n+1} \leq \beta , \ \ \forall n\geq 0.
%\end{equation*}
%Then, for any $n_0\geq 0$,
%\begin{equation*}
%z^n \leq (1+\delta k)^{-(n-n_0)} z^{n_0} + \delta^{-1} \beta, \ \ \forall n\geq n_0.
%\end{equation*}
%\end{lem}

\begin{lem}{\bf (Uniform Gronwall Lemma)} \label{GL}
Let $g=g(t)$, $h=h(t)$, $z=z(t)$ be three positive locally integrable functions defined in  $(0,+\infty)$ with $z'(t)$  locally integrable in $(0,+\infty)$, such that
$$
z' (t)\leq g(t) z(t)+h(t) \quad \mbox{a.e.} \ t\geq 0.
$$
If for any $T > 0$ there exist $a_1(T)$, $a_2(T)$ and $a_3(T)$ such that
\begin{equation*}
\int_t^{t+T} g(s) ds  \leq a_1(T), \ \  \int_t^{t+T} h(s) ds  \leq a_2(T), \ \ \int_t^{t+T} z(s) ds  \leq a_3(T) ,\quad \forall\, t\geq 0,
\end{equation*}
 then
\begin{equation*}
z(t)\leq \left(a_2(T)  + \frac{a_3(T)}{T}\right) \mbox{exp} (a_1(T)),\  \ \forall t\geq T.
\end{equation*}
\end{lem}
\begin{lem}{\bf (Uniform discrete Gronwall lemma)}\label{LGD001}
 Let $k>0$ and $z^n,g^n,h^n\geq 0$ such that 
\begin{equation}\label{Gb01}
\frac{z^{n+1}-z^n}{k} \leq   g^n z^n +  h^n, \quad \forall n\geq 0.
\end{equation}
If for any $r\in \mathbb{N}$ with  $t_r =kr$, there exist $a_1(t_r)$, $a_2(t_r)$ and $a_3(t_r)$, such that
\begin{equation*}\label{Gb03}
k\underset{n=n_0}{\overset{n_0+r-1}{\sum}} g^n\leq a_1(t_r), \ \ k\underset{n=n_0}{\overset{n_0+r-1}{\sum}} h^n\leq a_2(t_r), \ \ k\underset{n=n_0}{\overset{n_0+r-1}{\sum}} z^n\leq a_3(t_r),
\quad \forall\, n_0\ge 0,
\end{equation*}
then
\begin{equation*}\label{Gb04}
z^n\leq \left(a_2(t_r)+ \frac{a_3(t_r)}{t_r} \right) \mbox{exp}\left\{ a_1(t_r)\right\}, \ \ \forall n\geq r.
\end{equation*}
\end{lem}

As consequence of Lemma \ref{LGD001} (estimating $z^n$ for any $n\ge r$) and the classical discrete Gronwall Lemma (estimating $z^n$ for $n=1,\dots,r-1$), 
  the following result can be proved:
\begin{cor}\label{CorUnif}
Assume conditions of Lemma \ref{LGD001}. Let  $k_0\in \mathbb{N} $ be fixed, then the following estimate holds for all $k\le k_0$:
\begin{equation}\label{CC02}
z^n \leq C(z^0,k_0) \ \ \ \forall n=n(k)\geq 0.
\end{equation}
\end{cor}
\begin{proof}
We fix $T=2k_0$ and $k\le k_0$, and  choose $r_0\in \mathbb{N}$ such that  $k(r_0-1)<T\le k\,r_0:=t_{r_0}$. Then, from Lemma~\ref{LGD001} we have
\begin{eqnarray}\label{C01}
&z^n&\!\!\!\leq 
%\left(a_2(t_{r_0})+ \frac{a_3(t_{r_0})}{t_{r_0}} \right) \mbox{exp}\left\{ a_1(t_{r_0})\right\}\nonumber\\&&\!\!\!\leq
\left(a_2(t_{r_0})+ \frac{a_3(t_{r_0})}{T} \right) \mbox{exp}\left\{ a_1(t_{r_0})\right\}:= C_1(k_0), \ \ \forall n\geq r_0.
\end{eqnarray}
On the other hand, applying the classical discrete Gronwall Lemma to (\ref{Gb01}), one has 
\begin{equation}\label{C02}
z^n\leq \left(a_2(t_{r_0})+z^0 \right) \mbox{exp}\left\{ a_1(t_{r_0})\right\}:=C_2(z^0,k_0), \ \ \forall n< r_0.
\end{equation}
Therefore, from (\ref{C01})-(\ref{C02}) 
 the bound (\ref{CC02}) is deduced.
\end{proof}

\section{Analysis of the continuous model}
In this section, the weak and strong regularity of problem (\ref{modelf00}) is analyzed. Sometimes, we distinguish  for simplicity between $2D$ or $3D$ domains, although all results valid for $2D$ are also valid for $1D$ domains.

\subsection{The $(u,v)$-problem (\ref{modelf00})}
%We will  start by giving the following definition of weak-strong solutions of problem (\ref{modelf00}).
\begin{defi} \label{ws00}{\bf (Weak-strong solutions of (\ref{modelf00}))} 
Given $(u_0, v_0)\in L^2(\Omega)\times H^1(\Omega)$ with $u_0\geq 0$, $v_0\geq 0$ a.e.~$\x\in \Omega$ and let $m_0=\frac1{|\Omega|} \int u_0$. A pair $(u,v)$ is called weak-strong solution of problem (\ref{modelf00}) in $(0,+\infty)$, if $u\geq 0$, $v\geq 0$ a.e.~$(t,\x)\in (0,+\infty)\times \Omega$,
\begin{equation}\label{wsa}
(u-m_0,v-m_0^2) \in L^{\infty}(0,+\infty;L^2(\Omega)\times H^1(\Omega)) 
 \cap L^{2}(0,+\infty;H^1(\Omega)\times H^2(\Omega)),  %\ \ \forall T>0,
\end{equation}
\begin{equation}\label{wsa-bis}
(\partial_t u, \partial_t v) \in L^{q'}(0,T;H^1(\Omega)' \times L^2(\Omega)), \ \ \forall T>0,
\end{equation}
where $q'=2$ in the $2$-dimensional case $(2D)$ and $q'=4/3$ in the $3$-dimensional case $(3D)$ ($q'$ is the conjugate exponent of $q=2$ in $2D$ and $q=4$ in $3D$); the following variational formulation holds
\begin{equation}\label{wf01}
\int_0^T \langle \partial_t u,\overline{u}\rangle + \int_0^T (\nabla u,  \nabla \overline{u}) +\int_0^T (u\nabla v,\nabla \overline{u})=0, \ \ \forall \overline{u}\in L^q(0,T;H^{1}(\Omega)), \ \ \forall T>0,
\end{equation}
the following equation holds pointwisely
\begin{equation}\label{wf02}
\partial_t v +A v=u^2 \ \ \mbox{ a.e. } (t,\x)\in (0,+\infty)\times\Omega,
\end{equation}
the initial conditions $(\ref{modelf00})_4$ are satisfied and the following energy inequality (in integral version) holds  a.e.~$t_0,t_1$ with $t_1\geq t_0\geq 0$:
\begin{equation}\label{wsd}
\mathcal{E}(u(t_1),v(t_1)) - \mathcal{E}(u(t_0),v(t_0))
 + \int_{t_0}^{t_1} \left(\Vert \nabla u(s)  \Vert_{0}^2 
 +\frac{1}{2} \Vert \nabla v(s) \Vert_{1}^2
 % +\frac{1}{2} \Vert \Delta v(s) \Vert_{0}^2
 \right)\ ds \leq0,
\end{equation}
 where
\begin{equation}\label{eneruva}
\mathcal{E}(u,v)=\displaystyle
\frac{1}{2}\Vert u\Vert_{0}^2 + \frac{1}{4}\Vert \nabla v\Vert_{0}^{2}.
\end{equation}
\end{defi}
\begin{obs} Observe that any weak-strong solution of (\ref{modelf00}) 
%(or weak solution of (\ref{modelf01})) 
is $u$-conservative, because the total cell mass $\displaystyle\int_\Omega u(t)$ remains constant in time. Indeed, by taking $\overline{u}=1$ in (\ref{wf01}), one holds 
%(or \eqref{wf01US}),
\begin{equation}\label{CCU}
%\frac{d}{dt}\left(\int_\Omega u\right)=0, \ \ \mbox{ i.e. } \ 
\int_\Omega u(t)=\int_\Omega u_0 = m_0 |\Omega|, \ \ \forall t>0.
\end{equation}
In particular, $u-m_0$ is a zero mean value function, hence $\|u-m_0\|_1$ is equivalent to $\| \nabla u\|_0$. 
%
%Moreover, integrating (\ref{wf02})
%%(or (\ref{modelf01eqv})$_1$) 
%in $\Omega$, then $\int_\Omega v$ satisfies the equation
%\begin{equation}\label{nuevo-1}
%\frac{d}{dt}\left(\int_\Omega v\right)=\int_\Omega u^2 -\int_\Omega
%v.
%\end{equation}
\end{obs}

\begin{obs}
In 2D domains, we can take $\overline{u}=u$ in (\ref{wf01}), and testing  (\ref{wf02}) by $-\displaystyle\frac{1}{2}\Delta v$,  integrating by parts and using (\ref{H1divGrad}), we  arrive at the following energy law  (in differential  version):
\begin{equation}\label{weak}
 \frac{d}{dt}  \mathcal{E}(u(t),v(t))
 + \Vert \nabla u (t) \Vert_{0}^2  
 +\frac{1}{2} \Vert \nabla v (t)\Vert_{1}^2
%  +\frac{1}{2} \Vert \Delta v (t)\Vert_{0}^2
 =0
 \quad \hbox{a.e.~$t>0$}.
\end{equation}
Naturally, this equality is also true in $3D$ domains for regular enough solutions.
\end{obs}
\subsubsection{Justification of the Weak-Strong Regularity \eqref{wsa}-\eqref{wsa-bis}}

Observe that from the energy law (\ref{wsd}), and using  (\ref{H1divGrad}), 
 one can deduce 
\begin{equation}\label{e2}
\left\{\begin{array}{l}
(u, \nabla v)\in L^{\infty}(0,+\infty;L^2(\Omega)\times \L^2(\Omega)) , \\
(\nabla u,\nabla v) \in L^{2}(0,+\infty;\L^2(\Omega)\times \H^1(\Omega)) .
\end{array}\right.
\end{equation}
In particular, using \eqref{H1poin} and \eqref{CCU}, one has 
\begin{equation}\label{u-m_0}
u-m_0  \in L^{2}(0,+\infty;H^1(\Omega)).
\end{equation}
%From (\ref{e2}), one has 
%\begin{equation*}\label{e3-b}
%u  \in L^{2}(0,T;H^1(\Omega)), \ \ \forall T>0.
%\end{equation*}
%From (\ref{nuevo-1}),  the function $y(t)=\displaystyle\int_{\Omega} v(\x ,t) \, d\x =\Vert v(t) \Vert_{L^1}$ (where Remark \ref{rvpos} has been taken into account) satisfies  $y'(t)+y(t)= \Vert u(t) \Vert_{L^2}^2$. 
%Therefore, $y(t)=y(0) \, e^{-t} + \displaystyle\int_0^t e^{-(t-s)} \, \Vert u(s) \Vert_{L^2}^2\, ds$, and using (\ref{e2})$_1$,
%\begin{equation*}\label{e6b}
%\Vert v(t) \Vert_{L^1} 
%%\le e^{-t} \Vert v_0\Vert_{L^1}+ \int_0^t e^{-(t-s)} \Vert u(s)\Vert_{L^2}^2 \, ds
% \leq  \Vert v_0\Vert_{L^1} + \Vert u\Vert_{ L^\infty(0,+\infty; L^2)}^2, \ \ \forall t\geq 0.
%\end{equation*}

Rewriting (\ref{wf02}) as 
$$
\partial_t (v-m_0^2) -\Delta (v-m_0^2) + (v-m_0^2) = (u+m_0)(u-m_0)
$$
and testing by $(v-m_0^2)$ and using \eqref{u-m_0}, one has 
$$
 \frac{d}{dt}  \Vert  v-m_0^2 \Vert_{0}^2 +\Vert   v-m_0^2 \Vert_{1}^2  
 \le C \Vert  u +m_0\Vert_{L^{2}}^2 \Vert  u -m_0 \Vert_{L^{3}}^2  
  \le C  \Vert  u -m_0 \Vert_1^2 \in L^1(0,+\infty).
$$
Therefore, $v-m_0^2 \in L^{\infty}(0,+\infty;L^2(\Omega))\cap L^2(0,+\infty;H^1(\Omega))$ which, together with (\ref{e2}), imply 
\begin{equation*}
v-m_0^2 \in L^{\infty}(0,+\infty;H^1(\Omega)) \cap L^2(0,+\infty;H^2(\Omega)).
\end{equation*}
Finally, the time derivative regularity \eqref{wsa-bis} is deduced from the system \eqref{wf01}-\eqref{wf02} using $L^p$-interpolation inequalities and the  regularity \eqref{wsa}.
\begin{obs}
In $2D$ domains, by using the interpolation inequality \eqref{in2D} (arguing  as for the Navier-Stokes equations \cite{QM}),  one  can  deduce the uniqueness of weak-strong solutions of  (\ref{modelf00}).
\end{obs}

\subsection{The $(u,\boldsymbol\sigma)$-problem (\ref{modelf01})}

%We also give the  definition of weak solutions of the reformulated problem (\ref{modelf01}):
\begin{defi} \label{ws00USV}{\bf (Weak solutions of (\ref{modelf01}))} Given $(u_0,{\boldsymbol\sigma}_0)\in L^2(\Omega)\times \L^2(\Omega)$ with $u_0\geq 0$ a.e.~$\x\in \Omega$ and $m_0=\frac1{|\Omega|}\int_\Omega u_0$. A pair $(u,{\boldsymbol\sigma})$ is called weak solution of problem $(\ref{modelf01})$ in $(0,+\infty)$, if $u\geq 0$ a.e.~$(t,\x)\in (0,+\infty)\times\Omega$,
\begin{equation*}\label{wsaUS}
(u-m_0, {\boldsymbol\sigma})\in L^{\infty}(0,+\infty;L^2(\Omega)\times \L^2(\Omega))  \cap L^{2}(0,+\infty;H^1(\Omega)\times \H^1(\Omega)),   
\end{equation*}
\begin{equation*}\label{wsaUS-bis}
(\partial_t u, \partial_t {\boldsymbol\sigma})
 \in L^{q'}(0,T;H^1(\Omega)' \times \H^1(\Omega)'),  \ \ \forall T>0,
%\partial_t u \in L^{q'}(0,T;H^1(\Omega)'), \ \  \partial_t {\boldsymbol\sigma} \in L^{q'}(0,T;\H^1(\Omega)'),  \ \ \forall T>0,
\end{equation*}
where $q'$ is as in Definition \ref{ws00}; the following variational formulations hold
\begin{equation}\label{wf01US}
\int_0^T \langle \partial_t u,\overline{u}\rangle + \int_0^T (\nabla u,  \nabla \overline{u}) +\int_0^T (u\,{\boldsymbol\sigma},\nabla \overline{u})=0, \ \ \forall \overline{u}\in L^q(0,T;H^{1}(\Omega)), \ \ \forall T>0,
\end{equation}
\begin{equation}\label{wf01bUS}
\int_0^T \langle \partial_t {\boldsymbol\sigma}, \overline{\boldsymbol\sigma} \rangle +\int_0^T \langle B {\boldsymbol\sigma},\overline{\boldsymbol\sigma} \rangle=2 \int_0^T (u\nabla u,\overline{\boldsymbol\sigma}), \ \forall \overline{\boldsymbol\sigma}\in L^q(0,T;\H^{1}(\Omega)), \ \ \forall T>0,
\end{equation}
the initial conditions $(\ref{modelf01})_5$ are satisfied and the following energy inequality (in integral version) holds  a.e.~$t_0,t_1$ with  $t_1\geq t_0\geq 0$:
\begin{equation}\label{wsdUS}
\mathcal{E}(u(t_1),{\boldsymbol\sigma}(t_1)) - \mathcal{E}(u(t_0),{\boldsymbol\sigma}(t_0))
 + \int_{t_0}^{t_1} (\Vert \nabla u(s)  \Vert_{0}^2 + \frac12\, \Vert  {\boldsymbol\sigma}(s) \Vert_{1}^2 )\ ds \leq0,
\end{equation}
 where 
 \begin{equation}\label{ener01}
\mathcal{E}(u,{\boldsymbol \sigma})=\displaystyle
\frac{1}{2}\Vert u\Vert_{0}^2 + \frac{1}{4}\Vert {\boldsymbol
\sigma}\Vert_{0}^{2}.
\end{equation}
\end{defi}

\subsection{Problems \eqref{modelf00} and \eqref{modelf01}-\eqref{modelf01eqv}  are equivalent} 

\begin{lem}\label{eqREM}
If ${\boldsymbol\sigma}_0=\nabla v_0$, problems (\ref{modelf00}) and (\ref{modelf01})-(\ref{modelf01eqv}) are equivalent in the following sense: 
\begin{itemize}
\item 
If $(u,v)$ is a weak-strong solution of (\ref{modelf00}) then $(u,{\boldsymbol\sigma})$ with ${\boldsymbol\sigma}=\nabla v$ is a weak solution of  (\ref{modelf01}).
\item
Reciprocally, if $(u,{\boldsymbol\sigma})$ is a weak solution of (\ref{modelf01}) and $v=v(u^2)$ is the unique strong solution of problem (\ref{modelf01eqv}), 
 then ${\boldsymbol\sigma}=\nabla v$ and  $(u,v)$ is a weak-strong solution of (\ref{modelf00}). Indeed, since $u^2 \in  L^p(0,T; L^p(\Omega))\cap  L^{q'}(0,T;L^2(\Omega))$ for 
 $p=5/3$ in $3D$ or $p=2$ in $2D$ and  $q'$ is given in Definition \ref{ws00},
  by applying the $L^p$-regularity of problem \eqref{modelf01eqv}, \cite[Theorem 10.22]{feireisl}, 
  one has $v\in L^p(0,T;W^{2,p}(\Omega))\cap L^\infty(0,T;W^{2-2/p,p}(\Omega))\cap
 L^{q'}(0,T;H^2(\Omega))$.
\end{itemize} 
\end{lem}
\begin{proof}
Suppose that $(u,v)$ is a weak-strong solution of (\ref{modelf00}), then testing  (\ref{wf02}) by $-\nabla\cdot \overline{\mathbf{w}}$, for any $\overline{\mathbf{w}}\in L^q(0,T;\H^1(\Omega))$, and taking into account that $\mbox{rot}(\nabla v)=0$,  one obtains
\begin{equation}\label{equvaCONT}
\int_0^T\langle\partial_t \nabla v, \overline{\mathbf{w}}\rangle + \int_0^T \langle  B \nabla v,\overline{\mathbf{w}}\rangle = 2\int_0^T(u\nabla u, \overline{\mathbf{w}}), \ \ \forall\,
\overline{\mathbf{w}}\in L^q(0,T;\H^1(\Omega)).
\end{equation}
Then, defining ${\boldsymbol\sigma}=\nabla v$ and assuming the hypothesis ${\boldsymbol\sigma}_0=\nabla v_0$, from (\ref{equvaCONT}) one can conclude that $(u,{\boldsymbol\sigma})$ is a weak solution of  (\ref{modelf01}). On the other hand, if $(u,{\boldsymbol\sigma})$ is a weak solution of (\ref{modelf01}) and $v=v(u^2)$ is the unique strong solution of problem (\ref{modelf01eqv}), reasoning as above, 
 it can be  concluded that  $\nabla v$ satisfies (\ref{equvaCONT}). Therefore, from (\ref{wf01bUS}) and (\ref{equvaCONT}),  one obtains
\begin{equation}\label{equva01CONT}
\int_0^T\langle\partial_t ({\boldsymbol\sigma} - \nabla v), \overline{\mathbf{w}} \rangle + \int_0^T \langle  B ({\boldsymbol\sigma}-\nabla v),\overline{\mathbf{w}} \rangle = 0, \ \ \forall\,
\overline{\mathbf{w}} \in L^q(0,T;\H^1(\Omega)).
\end{equation}
Then, since 
%${\boldsymbol\sigma}-\nabla v \in L^\infty(0,T;\L^p(\Omega))$, 
${\boldsymbol\sigma}-\nabla v \in L^\infty(0,T;\H^1(\Omega)')$, 
taking 
%$\overline{\mathbf{w}}=B^{-1}({\boldsymbol\sigma}-\nabla v) \in L^\infty(0,T;\W^{2,p}(\Omega)) \hookrightarrow L^q(0,T;\H^1(\Omega))$
$\overline{\mathbf{w}}=B^{-1}({\boldsymbol\sigma}-\nabla v) \in L^\infty(0,T;\H^1(\Omega))$
 in (\ref{equva01CONT}), 
we deduce
\begin{equation*}
\frac{1}{2} \Vert B^{-1}({\boldsymbol\sigma}(T)-\nabla v(T))\Vert_1^2  + \int_0^T \Vert {\boldsymbol\sigma}-\nabla v\Vert_0^2=\frac{1}{2} \Vert B^{-1}({\boldsymbol\sigma}(0)-\nabla v(0))\Vert_1^2=0,
\end{equation*}

where, in the last equality, the relation ${\boldsymbol\sigma}(0)=\nabla v(0)$  was used, and therefore 
 one can deduce ${\boldsymbol\sigma}=\nabla v$. Thus, $(u,v)$ is a weak-strong solution of (\ref{modelf00}).
\end{proof}

\begin{obs}\label{rvpos}
Since $v_0\ge 0$ in $\Omega$, then the unique strong solution $v=v(u^2)$ of problem (\ref{modelf01eqv}) satisfies $v\ge 0$ in $(0,+\infty)\times \Omega$.
\end{obs}

Later, in Section \ref{SecDiscret}, uniform estimates of a time discrete scheme  approximating  problem (\ref{modelf01})-(\ref{modelf01eqv}) will be proved,  
which will allow  to pass to the limit in the discrete problem in order to obtain the existence of weak solutions of problem (\ref{modelf01})-(\ref{modelf01eqv}) (in the sense of Definition~\ref{ws00USV}).
%and strong solution of (\ref{modelf01eqv}). 
Finally,  taking into account Lemma~\ref{eqREM}, one also has the existence of  weak-strong solutions of problem (\ref{modelf00}) (in the sense of Definition~\ref{ws00}).

\subsection{A regularity criterium implying global in time strong regularity}
Strong regularity of problem (\ref{modelf00}) is going to be deduced in a formal manner, without justifying  the computations and assuming sufficient regularity for the initial data $(u_0,v_0)$. In fact, a rigorous proof could be made via a regularization argument or a Galerkin approximation.
%, using the eigenfunctions of the operator $A$. 

\

In order to abbreviate, we introduce the notation: 
$$(\hat u, \hat v)= (u-m_0,v-m_0^2).$$
We  consider the following  formulation of  the problem \eqref{modelf00}:
\begin{equation}  \label{fuertes10-c}
\left\{
\begin{array}
[c]{lll}%
\partial_t \hat u+ \widehat A \hat u = \nabla\cdot ((\hat u+m_0)\nabla v), 
%\ \ \mbox{ a.e.~$\x \in \Omega$},
\\
\partial_t \hat v +A \hat v =  (\hat u+ 2 m_0)\hat u.
% \ \ \mbox{ a.e.} \ \x \in \Omega.
\end{array}
\right.  
\end{equation}
\begin{lem}\label{Lm:StrongIneq} {\bf(Strong inequality for $(\hat u,\hat v)$) }  It holds
\begin{equation}\label{StrongIneq}
\frac{d}{dt} 
\Big(\Vert  (\hat u, \hat v)\Vert_{1\times 2}^2 
+ \Vert   \partial_t \hat v\Vert_{0}^2 \Big)
 + \Vert (\hat u, \hat v)\Vert_{2\times 3}^2
 +  \Vert  (\partial_t \hat  u, \partial_t \hat v)\Vert_{0\times 1}^2
\le
 C_1 (\Vert  (\hat u,\hat v)\Vert_{1\times 2}^2)^d + C_2 \Vert (\hat u, \hat v)\Vert_{1\times 2}^2 
\end{equation}
where $d=2$ for $2D$ domains and $d=3$ for $3D$ domains.
\end{lem}
\begin{proof}
By testing \eqref{fuertes10-c}$_1$ by $\widehat A \hat u + \partial_t \hat u$ and \eqref{fuertes10-c}$_2$ by $A( A \hat v + \partial_t   \hat v)$,  bounding the right hand side using either \eqref{in2D} in $2D$ or \eqref{in3D} in $3D$, 
%the $H^1$, $H^2$ and $H^3$-inequalities 
%$\Vert  \hat u\Vert_1   \le C\, \Vert \widehat A^{1/2} \hat u\Vert_{0}$, 
%$\Vert  \hat u\Vert_2   \le C\, \Vert \widehat A \hat u\Vert_{0}$,
%$\Vert  \hat v\Vert_2   \le C\, \Vert A \hat v\Vert_{0}$
%and 
%$\Vert  \hat v\Vert_3   \le C\, \Vert A^{3/2} \hat v\Vert_{0}$, 
and the Young inequality, 
 one has
 \begin{eqnarray*}\label{fuertesb}
&&\displaystyle  \frac{d}{dt} \Vert (\nabla\hat u,A\hat v)\Vert_{0}^2
+ \Vert (\partial_t \hat u, \partial_t \hat v)\Vert_{0\times 1}^2
+ \Vert (\widehat A \hat u, A \hat v)\Vert_{0\times 1}^2 
%+ \frac{k}{2} \Vert (\delta_t u_n,\delta_t {\boldsymbol \sigma}_n)\Vert_{1}^2
\\
 &&\quad \leq C(\Vert  (\hat u+m_0) ( \nabla \hat v+ \hat u)\Vert_1^2  
 %+ \Vert 2(\hat u_n+m_0)\nabla \hat u_n\Vert_0^2).
   \leq C\,\Vert (\hat u+m_0,\nabla \hat v)\Vert^d_1\Vert (\hat u,\nabla \hat v)\Vert_{2}
 \\
&& \quad \le \frac12 \Vert (\widehat A \hat u, A\hat v )\Vert_{0\times 1}^2
 +C_1 \Big(\Vert (\nabla \hat u,A \hat v) \Vert_{0}^2\Big)^d + C_2 \Vert  (\hat u, \hat v)\Vert_{1\times 2}^2 
%\cancel{ + C_2 \Vert \widehat A^{1/2} \hat u,A \hat v  \Vert_{0}^2 }
\end{eqnarray*}
hence \eqref{StrongIneq} holds, taking into account that $\Vert \nabla \hat u \Vert_0$ is equivalent to $\Vert \hat u \Vert_1$ thanks to (\ref{H1poin}),
$\Vert A \hat v \Vert_0$ is equivalent to $\Vert \hat v \Vert_3$ thanks to (\ref{B101}), and 
$\Vert \hat A \hat u \Vert_0$ is equivalent to $\Vert \hat u \Vert_2$ thanks to (\ref{B001}). 
%
%
%\begin{eqnarray*}\label{fuertesb}
%&&\displaystyle  \frac{d}{dt} \Vert (\widehat A^{1/2}\hat u,A\hat v)\Vert_{0}^2
%+ \Vert (\partial_t \hat u, \partial_t A^{1/2} \hat v)\Vert_{0}^2
%+ \Vert (\widehat A \hat u, A^{3/2}\hat v)\Vert_{0}^2 
%%+ \frac{k}{2} \Vert (\delta_t u_n,\delta_t {\boldsymbol \sigma}_n)\Vert_{1}^2
%\\
% &&\quad \leq C(\Vert  (\hat u+m_0) ( \nabla \hat v+ \hat u)\Vert_1^2  
% %+ \Vert 2(\hat u_n+m_0)\nabla \hat u_n\Vert_0^2).
%   \leq C\,\Vert (\hat u+m_0,\nabla \hat v)\Vert^d_1\Vert (\hat u,\nabla \hat v)\Vert_{2}
% \\
%&& \quad \le \frac12 \Vert (\widehat A \hat u, A^{3/2}\hat v )\Vert_{0}^2
% +C_1 \Big(\Vert \widehat A^{1/2} \hat u,A \hat v \Vert_{0}^2\Big)^d 
% + C_2 \Vert \widehat A^{1/2} \hat u,A \hat v  \Vert_{0}^2 
%\end{eqnarray*}
%hence \eqref{StrongIneq} holds. 
%{\bf K: Dar poco detalle, solo hay que testear por $\widehat A \hat u + \partial_t \hat  u$ y $A( A v + \partial_t   u)$, integrar por partes en la ecuacion de $v$, acotar a derecha la norma $L^2$ al cuadrado  de las no linealidades y absorber con la parte izquierda}
\end{proof}

\begin{cor}\label{Cor:StrongReg} {\bf(Strong regularity for $(u,v)$) }
Let $(u_0, v_0)\in H^1(\Omega)\times H^2(\Omega)$. Assume 
\begin{itemize}
\item either  $dD$ domains with $d\le 2$, or
\item
 $3D$ domains and the following regularity criterion:
\begin{equation}\label{def-strongA}
(u, v) \in  L^{\infty}(0,+\infty;H^1(\Omega)\times H^2(\Omega)).
\end{equation}
\end{itemize}
  Then, the following  regularity holds
\begin{equation}\label{ssa}
(\hat u,\hat v) \in L^{\infty}(0,+\infty;H^1(\Omega)\times H^2(\Omega)) 
 \cap L^{2}(0,+\infty;H^2(\Omega)\times H^3(\Omega)),  %\ \ \forall T>0,
\end{equation}
\begin{equation}\label{ssa-bis}
(\partial_t u,\partial_t v) \in L^{2}(0,+\infty;L^2(\Omega)\times H^1(\Omega)).
\end{equation}
\end{cor}
\begin{proof}
By using the weak regularity \eqref{wsa}, it suffices to apply Lemma \ref{GL} to \eqref{StrongIneq}, directly for $2D$ domains, and using previously \eqref{def-strongA}  for $3D$ domains.
 \end{proof}

\bigskip

Moreover, taking the time derivative of  (\ref{modelf00})$_2$ and testing by  $\partial_t v$, one obtains
$$
\begin{array}{l}
\displaystyle\frac{1}{2} \frac{d}{dt} \Vert \partial_t v \Vert_0^2  + \Vert \partial_t v \Vert_1^2  
%= 2 (u \, \partial_t u, \partial_t v) 
\le 2\Vert u \Vert_{L^6} \Vert \partial_t u \Vert_0 \Vert \partial_t v \Vert_{L^3}
%\\
%\noalign{\vspace{-1ex}}\\
%\qquad 
 \le  \varepsilon \Vert \partial_t v \Vert_1^2  
 + C_{\varepsilon} \Vert u \Vert_1^2 \Vert \partial_t u \Vert_0^2, %\in L^1(0,+\infty),
\end{array}
$$
hence, since $\Vert u \Vert_1^2 \Vert \partial_t u \Vert_0^2  \in L^1(0,+\infty)$, one arrives at 
\begin{equation}\label{prev-a}
\partial_t v \in  L^{\infty}(0,+\infty;L^2(\Omega))\cap L^2(0,+\infty;H^1(\Omega)).
\end{equation}

\subsection{Global in time higher regularity}
Denote  $\widetilde{u}=\partial_t u$ and $\widetilde{v}=\partial_t v$. Then, derivating in time  (\ref{modelf00})  one deduces that $(\widetilde{u},\widetilde{v})$ satisfies
\begin{equation}
\left\{
\begin{array}
[c]{lll}%
\partial_t \widetilde{u} -\Delta \widetilde{u}  = -\nabla\cdot(\widetilde{u}\nabla v) - \nabla\cdot(u\nabla \widetilde{v}),\\
\partial_t \widetilde{v} - \Delta \widetilde{v} + \widetilde{v}=2 u \widetilde{u}.
\end{array}
\right.  \label{pbch001}
\end{equation}
\begin{lem}\label{Lm:RegIneq} Under conditions of Corollary \ref{Cor:StrongReg},
% regularity (\ref{def-strongA}). 
it holds
\begin{equation} \label{ut-vt}
 \frac d {dt} \left(  \Vert  \widetilde{u}\Vert^2_{0} +\frac{1}{2}\Vert \nabla  \widetilde{v} \Vert^2_{0}\right) + \Vert  \widetilde{u}\Vert^{2}_{1} 
 + \displaystyle\frac{1}{2} \Vert   \nabla \widetilde{v}\Vert^{2}_{1}  
 \leq  C  
\Vert \widetilde{u} \Vert_0^2 .
\end{equation}
\end{lem}
\begin{proof}
Testing  (\ref{pbch001})$_1$ by $\widetilde{u}$  and (\ref{pbch001})$_2$ by $\displaystyle - \frac{1}{2}\Delta \widetilde{v}$ and adding, the terms $(u\nabla \tilde v,\tilde u)$ cancel,  taking into account that $\displaystyle \int_\Omega \widetilde{u}=0$, using the $3D$ interpolation inequality (\ref{in3D}) and regularity (\ref{def-strongA}),  one has
\begin{eqnarray*}\label{pbch002}
&\displaystyle  \frac{1}{2} \frac d {dt} &\!\!\!\!\!\left(  \Vert  \widetilde{u}\Vert^2_{0} +\frac{1}{2}\Vert \nabla  \widetilde{v} \Vert^2_{0}\right) + \Vert  \widetilde{u}\Vert^{2}_{1} 
 + \displaystyle\frac{1}{2} \Vert   \nabla \widetilde{v}\Vert^{2}_{1}  
 = -( \widetilde{u}\nabla v,\nabla  \widetilde{u})+( \widetilde{u}\nabla u,\nabla  \widetilde{v})
 \nonumber\\
&& \!\!\!\!\!\!\!
\leq \Vert \widetilde{u} \Vert_{L^3} \Big(\Vert \nabla v \Vert_{L^6} \Vert \nabla \widetilde{u}\Vert_{0} + \Vert \nabla \widetilde{v} \Vert_{L^6} \Vert \nabla u \Vert_{0} \Big)
% \nonumber\\
%&& \!\!\!\!\!\!\!
\leq  \frac{1}{2} \Big( \Vert \widetilde{u} \Vert_{1}^2 +  \frac{1}{2} \Vert \nabla \widetilde{v} \Vert_{1}^2 \Big)
+ C  
%\Big( \Vert\nabla v\Vert_{1}^4 +   \Vert \nabla u \Vert_{0}^4 \Big)
\Vert \widetilde{u} \Vert_0^2 ,
\end{eqnarray*}
hence \eqref{ut-vt} holds.
\end{proof}
Therefore, 
since $ \Vert \widetilde{u} \Vert_0^2\in L^1(0,+\infty)$ (owing to (\ref{ssa-bis})) 
 and using (\ref{prev-a}), one has 
\begin{cor}\label{Cor:Reg} 
Let $(u_0, v_0)\in H^2(\Omega)\times H^3(\Omega)$. Under hypotheses of Corollary \ref{Cor:StrongReg},  the following  regularity holds
\begin{equation*}
(\partial_t u, \partial_t v) \in L^{\infty}(0,+\infty;L^2(\Omega)\times H^1(\Omega))
\cap L^2(0,+\infty;H^1(\Omega)\times H^2(\Omega)).
\end{equation*}
\end{cor}

%Lemma~\ref{GL} and 
%Moreover, integrating in time (\ref{pbch002}) and using (\ref{prev-a}) and (\ref{high001A}), 
%and taking into account that $(\nabla u,\nabla v) \in L^\infty(0,+\infty;\L^2(\Omega)\times \H^1(\Omega))\cap L^2(0,+\infty;\H^1(\Omega)\times \H^2(\Omega))$, 
% it could be also deduced that
%\begin{equation}\label{high001B}
%(\partial_t u, \partial_t v) \in  L^2(0,+\infty;H^1(\Omega)\times H^2(\Omega)).
%\end{equation} 
Finally, from system  (\ref{pbch001}) and taking into account (\ref{ssa})-(\ref{ssa-bis}), one has 
\begin{equation*}
(\partial_{tt} u , \partial_{tt} v) \in L^2(0,+\infty;H^1(\Omega)'\times L^2(\Omega)) .
\end{equation*}
\begin{obs}
By following with a bootstrap argument, it is possible to obtain more regularity for $(u,v)$. For instance, by assuming the $H^3$ and $H^4$-regularity of problems (\ref{B001}) and (\ref{B101}), one obtains  $(\hat u,\hat v) \in L^{\infty}(0,+\infty;H^2(\Omega)\times H^3(\Omega)) 
 \cap L^{2}(0,+\infty;H^3(\Omega)\times H^4(\Omega))$, which in particular implies $(u,v) \in L^\infty(0,+\infty;L^\infty(\Omega)\times L^\infty(\Omega))$.
Therefore, any global in time weak-strong solution satisfying \eqref{def-strongA} does not blow-up, neither at finite time nor infinite one. But we stop here, because the regularity obtained so far is sufficient  to guarantee the hypotheses required later in order to prove error estimates  (see Theorems \ref{erteo}, \ref{erteoV}, \ref{erteolin} and \ref{erteolinV} below).
\end{obs}
%
%\subsection{Proof of (\ref{def-strongA}) in 2D domains}
%%In order to prove (\ref{def-strongA}) in 2D domains, 
%By making $(\nabla (\ref{modelf00})_1 ,\nabla u) 
%+\displaystyle\frac{1}{2} (\Delta(\ref{modelf00})_2 ,\Delta v)$
%%. Then, integrating by parts, 
%and arguing as in (\ref{strong-uv-1}) without using extra regularity \eqref{def-strongA} and considering the $2D$ interpolation inequality (\ref{in2D}), one has
%\begin{equation}\label{strong-uv-1new}
%\begin{array}{l}
%\displaystyle\frac{1}{2} \displaystyle \frac{d}{dt}
%\left(
%\Vert \nabla u \Vert_0^2 +\displaystyle\frac{1}{2}\Vert  \Delta v \Vert_0^2 
%\right)  + \Vert \Delta u \Vert_0^2 + \displaystyle\frac{1}{2}\Vert \nabla v \Vert_0^2 
%+  \frac{1}{2}\Vert \nabla (\Delta v) \Vert_0^2 
%\\
%\noalign{\vspace{-.5 ex}}
%%\\
%\quad \le C \Vert \nabla u \Vert_{L^4}^2 \Vert \nabla^2 v \Vert _0 \le
%C \, \Vert \nabla u \Vert_0 \Vert \nabla u \Vert_{1}
% \Vert \nabla v \Vert_1 
%  \le \varepsilon \Vert \nabla u \Vert_{1}^2 + C_{\varepsilon} \, \Vert \nabla u \Vert_0^2
%  \Vert \nabla v\Vert_{1}^2.
%\end{array}
%\end{equation}
%Therefore, if we add (\ref{weak}) and (\ref{strong-uv-1new}), then  using  (\ref{H1divGrad}) and taking $\varepsilon$ small enough, we have 
%\begin{equation*}
% \frac{d}{dt} \left( \Vert  u \Vert_{1}^2 +\frac{1}{2}\Vert  \nabla v \Vert_{1}^2  \right)
% +   \Vert \nabla u \Vert_1^2 + \Vert \nabla  v \Vert_2^2 
%\le
% C \Vert \nabla u \Vert_0^2 \, \Vert \nabla v \Vert_1^2.
%\end{equation*}
%Therefore, since $\Vert \nabla u \Vert_0^2\in L^1(0,+\infty)$ (owing to \eqref{e2}$_2$), Lemma~\ref{GL} implies 
%(\ref{def-strongA}).
%

\section{Euler time discretization}\label{SecDiscret}
In this section,  the Euler time discretization for the problem (\ref{modelf01})-(\ref{modelf01eqv}) is studied, proving  its solvability, unconditional stability (in weak norms, see Definition \ref{enesf} below),   
and convergence towards weak solutions of (\ref{modelf01})-(\ref{modelf01eqv}). 
Moreover, some additional properties  as positivity, 
%of the cell and chemical variables, 
$u$-conservation, and error estimates are also studied.
 Finally, unique solvability of the scheme will be proved under a hypothesis that relates the time step and a strong norm of the scheme (see \eqref{uniq01} below, which is the discrete version of (\ref{def-strongA})). 
 This  strong norm  is  bounded  in the case of $1D$ and $2D$ domains.\\

Let us consider a fixed partition of the time interval $[0,+\infty)$ given by $t_n=nk$, where $k>0$ denotes the time step (that we take constant for simplicity). Taking into account the $(u,v)$-problem  (\ref{modelf00}) and the $(u,\boldsymbol\sigma)$-problem (\ref{modelf01}), the two following first-order, nonlinear and coupled (Backward Euler) schemes are considered (hereafter, we denote $\delta_t a_n=
(a_n - a_{n-1})/k$): 
\begin{itemize}
\item{\underline{\emph{Scheme \textbf{UV}}:}\\
{\bf Initialization:} 
We take $(u_0,v_0) =( u(0), v(0))$.\\
{\bf Time step} n: Given $(u_{n-1},v_{n-1})\in  H^{1}(\Omega)\times {H}^{2}(\Omega)$ 
with $u_{n-1}\geq 0$ and $v_{n-1}\geq 0$, 
compute $(u_{n},v_{n})\in  H^{1}(\Omega)\times {H}^{2}(\Omega)$ 
with $u_{n}\geq 0$, $v_{n}\geq 0$ and solving
\begin{equation}
\left\{
\begin{array}
[c]{lll}%
(\delta_t u_n,\overline{u}) + (\nabla u_n, \nabla \overline{u}) +(u_n \nabla v_n,\nabla \overline{u})=0, \ \ \forall \overline{u}\in H^{1}(\Omega),\\
\delta_t v_n + A v_n  -u_n^2=0\ \ \mbox{ a.e.~in $\Omega$}.
\end{array}
\right.  \label{modelfuv02}
\end{equation}
}
%\end{itemize}
% On the other hand, 
% the  Backward Euler scheme related to the reformulation in the $(u,\boldsymbol\sigma)$-problem (\ref{modelf01}) will be also considered:
 %Then, one also has the following  first-order, nonlinear and coupled scheme: 
%\begin{itemize}
\item{\underline{\emph{Scheme \textbf{US}}:}\\
{\bf Initialization:} 
We take $(u_0,v_0) =( u(0), v(0))$ and ${\boldsymbol\sigma}_0=\nabla v_0$.
\\
{\bf Time step} n: Given $(u_{n-1},{\boldsymbol \sigma}_{n-1})\in  H^{1}(\Omega)\times \H^{1}_{\sigma}(\Omega)$ 
with $u_{n-1}\geq 0$, 
compute $(u_{n},{\boldsymbol \sigma}_{n})\in  H^{1}(\Omega)\times \H^{1}_{\sigma}(\Omega)$ with $u_{n}\geq 0$ and solving
\begin{equation}
\left\{
\begin{array}
[c]{lll}%
(\delta_t u_n,\overline{u}) + (\nabla u_n, \nabla \overline{u})  +(u_n{\boldsymbol \sigma}_n,\nabla \overline{u})=0, \ \ \forall \overline{u}\in H^{1}(\Omega),\\
(\delta_t {\boldsymbol \sigma}_n,\overline{\boldsymbol \sigma}) +
\langle B {\boldsymbol \sigma}_n, \overline{\boldsymbol
\sigma}\rangle  -
2(u_n\nabla u_n,\overline{\boldsymbol \sigma})=0,\ \ \forall
\overline{\boldsymbol \sigma}\in \H^{1}_{\sigma}(\Omega).
\end{array}
\right.  \label{modelf02}
\end{equation}}

Once \eqref{modelf02} is solved, given $v_{n-1}\in {H}^{2}(\Omega)$
with $v_{n-1}\geq 0$, then 
$v_n=v_n(u_n^2) \in H^2(\Omega)$ (with $v_{n}\geq 0$) can be recovered by solving: 
\begin{equation}\label{edovf}
\delta_t v_n + A v_n = u_n^2 \ \ \mbox{ a.e.~in $\Omega$}.
\end{equation}
\end{itemize}

\begin{obs}[\bf Regularity and positivity  of $v_n$] \label{pvr}
It is not difficult to prove that, given $u_n \in H^1(\Omega)$ and $v_{n-1}\in {H}^{2}(\Omega)$, there exists a unique $v_n \in H^2(\Omega)$ solution of (\ref{edovf}). Even more, using the $H^3$-regularity of problem (\ref{B101}), we can prove that $v_n \in H^3(\Omega)$. Moreover, if $v_{n-1}\geq 0$ then $v_n\geq 0$. Indeed, testing  (\ref{edovf}) by  $(v_n)_{-}=\min\{v_n,0\}\leq 0$, and taking into account that 
%$(v_n)_{-}=0$ if $(v_n)\geq 0$, as well as 
$(v_n)_{-}\in H^1(\Omega)$ with
$\nabla((v_n)_{-})=\nabla (v_n)$ if $(v_n)\leq 0$, and
$\nabla((v_n)_{-})=0$ if $(v_n)> 0$, we obtain 
\begin{equation*}\label{posv1}
\displaystyle\frac{1}{k} \Vert(v_n)_{-}\Vert_0^2 -
\frac{1}{k}\int_\Omega v_{n-1} (v_n)_{-} + 
\Vert\nabla((v_n)_{-}) \Vert_0^2 + 
\Vert(v_n)_{-} \Vert_0^2 = \int_\Omega u_n^2
 (v_n)_{-}\leq 0.
\end{equation*} 
Then, since $v_{n-1}\geq 0$, 
%from (\ref{posv1}) 
we conclude $\Vert(v_n) _{-}\Vert_{1}^2 \leq 0$, hence 
$(v_n)_{-}= 0$ a.e.~in $\Omega$,  i.e.~$v_n\geq 0$ in $\Omega$.
\end{obs}
\begin{lem}[\bf Equivalence of both schemes] 
\label{eqLEM}
If ${\boldsymbol\sigma}_{n-1}=\nabla v_{n-1}$, the schemes \textbf{UV} and \textbf{US} are equivalent in the following sense:
\begin{itemize}
\item  If $(u_n,v_n)$ solves scheme \textbf{UV} then $(u_n,{\boldsymbol\sigma}_n)$ with ${\boldsymbol\sigma}_n=\nabla v_n$ solves  scheme \textbf{US}.
\item
Reciprocally, if $(u_n,{\boldsymbol\sigma}_n)$ solves scheme \textbf{US} and $v_n=v_n(u_n^2)$ is the unique solution of (\ref{edovf}), then ${\boldsymbol\sigma}_n=\nabla v_n$ and  $(u_n,v_n)$ solves scheme \textbf{UV}.
\end{itemize}
\end{lem}
\begin{proof}
Suppose that $(u_n,v_n)$ is a solution of the scheme \textbf{UV}, then testing  (\ref{modelfuv02})$_2$ by $-\nabla\cdot \overline{\mathbf{w}}$, for any $\overline{\mathbf{w}}\in \H^{1}_{\sigma}(\Omega)$, and taking into account that $\mbox{rot}(\nabla v_n)=0$, we obtain
\begin{equation}\label{equva}
(\delta_t \nabla v_n, \overline{\mathbf{w}}) + \langle  B \nabla v_n,\overline{\mathbf{w}}\rangle = 2(u_n\nabla u_n, \overline{\mathbf{w}}), \ \ \forall\,
\overline{\mathbf{w}}\in \H^{1}_{\sigma}(\Omega).
\end{equation}
Then, defining ${\boldsymbol\sigma}_n=\nabla v_n$ and assuming the hypothesis ${\boldsymbol\sigma}_{n-1}=\nabla v_{n-1}$, from (\ref{equva}) we conclude that $(u_n,{\boldsymbol\sigma}_n)$ is solution of  the scheme \textbf{US}.  On the other hand, if $(u_n,{\boldsymbol\sigma}_n)$ is a solution of  the scheme \textbf{US} and $v_n$ satisfies (\ref{edovf}), reasoning as above, we conclude that  $\nabla v_n$ satisfies (\ref{equva}). Therefore, from (\ref{modelf02})$_2$ and (\ref{equva}), we obtain
\begin{equation}\label{equva01}
(\delta_t ({\boldsymbol\sigma}_n-\nabla v_n), \overline{\boldsymbol\sigma}) + \langle B ({\boldsymbol\sigma}_n-\nabla v_n),\overline{\boldsymbol\sigma}\rangle = 0, \ \ \forall\,
\overline{\boldsymbol\sigma}\in \H^{1}_{\sigma}(\Omega).
\end{equation}
Then, taking $\overline{\boldsymbol\sigma}={\boldsymbol\sigma}_n-\nabla v_n$ in (\ref{equva01}) and using the formula
$a(a-b)=\displaystyle\frac{1}{2}(a^2-b^2)+\displaystyle\frac{1}{2}(a-b)^2$, we deduce
\begin{equation*}
\delta_t\left(\frac{1}{2} \Vert {\boldsymbol\sigma}_n-\nabla v_n\Vert_0^2 \right) + \frac{k}{2} \Vert \delta_t({\boldsymbol\sigma}_n-\nabla v_n)\Vert_0^2 + \Vert {\boldsymbol\sigma}_n-\nabla v_n\Vert_1^2=0,
\end{equation*}
which implies that ${\boldsymbol\sigma}_n=\nabla v_n$ using that  ${\boldsymbol\sigma}_{n-1}=\nabla v_{n-1}$. Thus, we conclude that $(u_n,v_n)$ is solution of  the scheme \textbf{UV}.
\end{proof}
\begin{obs}
Since both time schemes \textbf{UV} and \textbf{US} are equivalent, we will study the scheme \textbf{US} 
%with the variable  ${\boldsymbol\sigma}_n$ 
in order to facilitate the notation throughout the paper. 
On the other hand, both schemes will furnish different 
 fully discrete schemes  when considering for instance a Finite Element spatial approximation, which will be analyzed in the forthcoming paper \cite{FMD2}. In fact,  it will be necessary to use the variable  ${\boldsymbol\sigma}_n$  in order to obtain a fully discrete unconditionally energy-stable scheme.
 \end{obs}
\subsection{Conservation, Solvability, Energy-Stability and Convergence}
\subsubsection{Conservation} 
Taking $\bar{u}=1$ in (\ref{modelf02})$_1$ we see that the scheme  \textbf{US} is $u$-conservative, that is:
\begin{equation}\label{consu}
\int_\Omega u_n=\int_\Omega u_{n-1}=\cdot\cdot\cdot=\int_\Omega
u_{0}.
\end{equation}
In particular, if we denote $\hat u_n=u_n-m_0$, then $\displaystyle\int_\Omega \hat u_n=0$.
%Moreover, integrating  (\ref{edovf}) in $\Omega$,  the following discrete in time  equation for $\displaystyle\int_\Omega v_n$ is deduced:
%\begin{equation} \label{consu-1}
%\delta_t\left(\int_\Omega v_n\right)
%+ \int_\Omega v_n =\int_\Omega
%u_n^{2} .
%\end{equation}
\subsubsection{Solvability} 
\begin{tma} {\bf(Unconditional existence and conditional uniqueness)} \label{USus}
There exists $(u_n,{\boldsymbol\sigma}_n)\in  H^{1}(\Omega)\times \H^{1}_{\sigma}(\Omega)$ solution of  the scheme \textbf{US}, such that $u_n\geq 0$. Moreover, if 
\begin{equation}\label{uniq01}
\displaystyle k\, \Vert (u_n,{\boldsymbol \sigma}_n)
\Vert_1^4 \quad \hbox{is small enough,}
\end{equation}
 then the solution of  the scheme \textbf{US} is unique.
\end{tma}
\begin{proof}
Let $(u_{n-1},{\boldsymbol \sigma}_{n-1})\in H^1(\Omega)\times \H^1_{\sigma}(\Omega)$ be given, with $u_{n-1}\geq 0$. The proof is divided into two parts.\\
\underline{\emph{Part 1:}} Existence of $(u_n,{\boldsymbol\sigma}_n)\in  H^{1}(\Omega)\times \H^{1}_{\sigma}(\Omega)$ solution of the  scheme \textbf{US}, such that $u_n\geq 0$. 

We consider the following auxiliary problem:
\begin{equation}
\left\{
\begin{array}
[c]{lll}%
(\delta_t u_n,\overline{u}) + (\nabla u_n, \nabla \overline{u}) +((u_n)_+{\boldsymbol \sigma}_n,\nabla \overline{u})=0, \ \ \forall \overline{u}\in H^{1}(\Omega),\\
(\delta_t {\boldsymbol \sigma}_n,\overline{\boldsymbol \sigma}) +
\langle B {\boldsymbol \sigma}_n, \overline{\boldsymbol
\sigma}\rangle  -
2(u_n\nabla u_n,\overline{\boldsymbol \sigma})=0,\ \ \forall
\overline{\boldsymbol \sigma}\in \H^{1}_{\sigma}(\Omega),
\end{array}
\right.  \label{modelf02mod}
\end{equation}
where $(u_n)_+=\max\{u_n,0\}$. In fact, it is the same scheme \textbf{US} but changing $u_n$ by $(u_n)_+$ in the chemotaxis term. 
\\
\textbf{A. Positivity of $u_n$:} First, we will see that if $(u_n,{\boldsymbol\sigma}_n)$ is a  solution of (\ref{modelf02mod}), then $u_n\geq 0$. Taking  $\overline{u}=(u_n)_{-}=\min\{u_n,0\}\leq 0$ in (\ref{modelf02mod}), and using that $((u_n)_+{\boldsymbol \sigma}_n,\nabla((u_n)_-))=0$,
%the fact that $(u_n)_{-}=0$ if $(u_n)\geq 0$, $\nabla((u_n)_{-})=\nabla (u_n)$ if $(u_n)\leq 0$, and $\nabla((u_n)_{-})=0$ if $(u_n)> 0$, 
we obtain 
\begin{equation*}\label{posu1}
\displaystyle\frac{1}{k} \Vert(u_n)_{-}\Vert_0^2 -
\frac{1}{k}\int_\Omega u_{n-1} (u_n)_{-} + 
\Vert\nabla((u_n)_{-}) \Vert_0^2 = 0.
\end{equation*}
Then, using the fact that $u_{n-1}\geq0$, one has that 
$\Vert(u_n) _{-}\Vert_{1}^2 \leq 0$, and thus  $u_n\geq 0$ in
$\Omega$. \\
\textbf{B. Existence of solution of (\ref{modelf02mod}):} It can be proved by using the Leray-Schauder fixed-point theorem (see Appendix A). \\
Then, from parts A and B, we conclude that there exists $(u_n,{\boldsymbol \sigma}_n)$ solution of (\ref{modelf02mod}) with $u_n\geq 0$. In particular, taking into account that $u_n=(u_n)_+$, we conclude that  $(u_n,{\boldsymbol \sigma}_n)$ is also a solution of  the scheme \textbf{US}, with $u_n\geq 0$.\\

\underline{\emph{Part 2:}} Uniqueness of solution $(u_{n},{\boldsymbol
\sigma}_{n})$ of the scheme \textbf{US}.

Let  $(u_n^1,{\boldsymbol
\sigma}_n^1),(u_n^2,{\boldsymbol \sigma}_n^2)\in H^1(\Omega)\times  \H^{1}_{\sigma}(\Omega)$ be two possible solutions of (\ref{modelf02}).
Then, defining $u_n=u_n^1-u_n^2$ and ${\boldsymbol
\sigma}_n={\boldsymbol \sigma}_n^1-{\boldsymbol \sigma}_n^2$, one has that $(u_n,{\boldsymbol \sigma}_n)\in H^1(\Omega)\times  \H^{1}_{\sigma}(\Omega)$ satisfies
\begin{eqnarray}\label{uniq001}
\displaystyle\frac{1}{k}(u_n,\overline{u})+ (\nabla u_{n},
\nabla \overline{u}) +(u_{n}^1{\boldsymbol \sigma}_n,\nabla \overline{u})
+(u_{n}{\boldsymbol \sigma}_n^2,\nabla \overline{u})=0, \quad \forall\,
\overline{u}\in H^1(\Omega),
\end{eqnarray}
\vspace{-1 cm}
\begin{eqnarray}\label{uniq002}
\displaystyle\frac{1}{k}({\boldsymbol \sigma}_n,\overline{\boldsymbol
\sigma})+ \langle B {\boldsymbol
\sigma}_n,\overline{\boldsymbol \sigma}\rangle 
- 2(u_{n}^1\nabla u_n,\overline{\boldsymbol \sigma})
- 2(u_{n}\nabla u_n^2,\overline{\boldsymbol \sigma})=0,\quad \forall\, \overline{\boldsymbol
\sigma}\in \H^{1}_{\sigma}(\Omega).
\end{eqnarray}
Taking $\overline{u}=u_n$, $\overline{\boldsymbol \sigma}=\displaystyle
\frac{1}{2}{\boldsymbol \sigma}_n$ in (\ref{uniq001})-(\ref{uniq002}) and adding the resulting expressions, the terms $(u_{n}^1{\boldsymbol \sigma}_n,\nabla u_n)$  cancel, and using the fact that $\displaystyle\int_\Omega u_n=0$,  one obtains
\begin{eqnarray*}
&\displaystyle\frac{1}{2k}&\!\!\!\!\Vert (u_n,{\boldsymbol
\sigma}_n)\Vert_{0}^2+ \frac{1}{2}\Vert (u_n,{\boldsymbol \sigma}_n)\Vert_{1}^2 
%\\ &&
\leq \Vert u_{n}\Vert_{L^3} 
\Big(
\Vert {\boldsymbol \sigma}_n^2\Vert_{L^6} \Vert \nabla u_n\Vert_0 + 
\Vert \nabla u_n^2\Vert_0 \Vert {\boldsymbol \sigma}_n\Vert_{L^6} \Big) 
\\
&&\leq C\Vert u_{n}\Vert_{0}^{1/2} \Vert u_{n}\Vert_{1}^{1/2} \Big(
 \Vert {\boldsymbol \sigma}_n^2\Vert_{1} \Vert u_n\Vert_1
+  \Vert  u_n^2\Vert_1 \Vert {\boldsymbol \sigma}_n\Vert_{1} \Big)
%\\ &&
\leq \displaystyle\frac{1}{4}\Vert (u_n,{\boldsymbol \sigma}_n)\Vert_{1}^2 
+ C\Vert u_{n}\Vert_{0}^2 \, \Vert (u_n^2,{\boldsymbol \sigma}_n^2) \Vert_1^4.
\end{eqnarray*}
%which implies that
%\begin{equation*}\label{uniq003}
%\displaystyle\frac{1}{2}\Vert (u_n,{\boldsymbol
%\sigma}_n)\Vert_{0}^2+ \frac{k}{4}\Vert (u_n,
%{\boldsymbol \sigma}_n)\Vert_{1}^2 \leq C \, k \,
%\Vert (u_n^2,{\boldsymbol \sigma}_n^2) \Vert_1^4 \,
%\Vert u_n\Vert_{0}^2.
%\end{equation*}
Therefore, assuming 
%$k \, \Vert (u_n^2,{\boldsymbol \sigma}_n^2) \Vert_1^4$  small enough (from 
hypothesis \eqref{uniq01}, one has $\Vert (u_n,{\boldsymbol \sigma}_n)\Vert_{1}=0$, thus $u_n^1=u_n^2$ and ${\boldsymbol\sigma}_n^1={\boldsymbol\sigma}_n^2$.
\end{proof}
\begin{obs}
In the case of $2D$ domains, using estimate (\ref{strong01}) (see Theorem \ref{stlemne} below), the uniqueness restriction (\ref{uniq01}) can be relaxed to $k\, K_0^2$ small enough,
 where $K_0$ is the constant appearing in (\ref{strong01})  depending on data $(\Omega,u_0,{\boldsymbol\sigma}_0)$, but it is independent of $k$ and $n$.
\end{obs}
\subsubsection{Energy-Stability and uniform weak estimates} 
\begin{defi}\label{enesf}
A numerical scheme with solution $(u_n,{\boldsymbol \sigma}_n)$ is called energy-stable with respect to the  energy $\mathcal{E}(u,{\boldsymbol \sigma})$ given in  (\ref{ener01}) if this energy is time decreasing, that is, 
\begin{equation}\label{stabf02}
\mathcal{E}(u_n,{\boldsymbol \sigma}_n)\leq \mathcal{E}(u_{n-1},{\boldsymbol
\sigma}_{n-1}), \ \ \forall n.
\end{equation}
\end{defi}
 
%In next Lemma, we obtain unconditional energy-stability for the scheme \textbf{US}.
%in the sense of Definition \ref{enesf}.
\begin{lem} {\bf (Unconditional energy-stability)} \label{estinc1}
 The scheme \textbf{US} is unconditionally energy-stable with respect to $\mathcal{E}(u,{\boldsymbol \sigma})$. In fact, for any $(u_n,{\boldsymbol
\sigma}_n)$  solution of scheme \textbf{US}, the following discrete energy law holds
\begin{eqnarray}\label{lawenerfydisce}
&\delta_t \mathcal{E}(u_n,{\boldsymbol \sigma}_n)&\!\!\!\!\!+ 
\frac{k}{2} 
\Vert \delta_t u_n\Vert_{0}^2 +
\frac{k}{4} \Vert \delta_t {\boldsymbol \sigma}_n\Vert_{0}^2 
 + \Vert \nabla u_n\Vert_{0}^{2} +
\displaystyle\frac{1}{2}\Vert  {\boldsymbol
\sigma}_n\Vert_{1}^{2}=0.
\end{eqnarray}
%where $\mathcal{E}(u_n,{\boldsymbol \sigma}_n)$ was defined in (\ref{ener01}).
\end{lem}
\begin{proof}
Taking $\overline{u}=u_n$ in (\ref{modelf02})$_1$ and $\overline{\boldsymbol \sigma}=\frac{1}{2}{\boldsymbol \sigma}_n$ in (\ref{modelf02})$_2$ and adding the resulting expressions, the chemotaxis and production terms cancel, obtaining
%\begin{equation}\label{modelfle03}
%\int_\Omega u_n \cdot \delta_t u_n + \Vert \nabla
%u_n\Vert_{0}^{2} + \displaystyle\frac{1}{2}\int_\Omega {\boldsymbol \sigma}_n\cdot
%\delta_t {\boldsymbol \sigma}_n + \frac{1}{2} \Vert
% {\boldsymbol \sigma}_n\Vert_{1}^{2} =0.
%\end{equation}
%Using the formula
%$a(a-b)=\displaystyle\frac{1}{2}(a^2-b^2)+\displaystyle\frac{1}{2}(a-b)^2$ {\color{red} (esta fÃ³rmula ya se ha usado en la expresiÃ³n que aparece despuÃ©s de (\ref{equva01}))}
%we deduce that
%\begin{equation}\label{modelfle02}
%\int_\Omega u_n \cdot \delta_t u_n
%+ \frac{1}{2}\int_\Omega {\boldsymbol \sigma}_n\cdot
%\delta_t {\boldsymbol \sigma}_n
%=\delta_t\left(\frac{1}{2} \Vert u_n\Vert_{0}^2 + \frac{1}{4}  \Vert{\boldsymbol
%\sigma}_n\Vert_{0}^2 \right) + \frac{k}{2} 
%\Vert \delta_t u_n\Vert_{0}^2
%+ \frac{k}{4} \Vert \delta_t
%{\boldsymbol \sigma}_n\Vert_{0}^2.
%\end{equation}
%Thus, from (\ref{modelfle03})-(\ref{modelfle02}), we deduce 
(\ref{lawenerfydisce}).
\end{proof}

\begin{obs}
Comparing the energy law  \eqref{weak}  and the discrete version  (\ref{lawenerfydisce}), we can say that the scheme \textbf{US} introduces the following two first order ``numerical  dissipation terms'': 
$$\frac{k}{2} 
\Vert \delta_t u_n\Vert_{0}^2\quad \hbox{and}  \quad \frac{k}{4} \Vert \delta_t {\boldsymbol \sigma}_n\Vert_{0}^2.
$$
\end{obs}

\

Starting from the (local in time) discrete energy law (\ref{lawenerfydisce}),  global in time estimates 
%for any $(u_n,{\boldsymbol \sigma}_n)$ solution 
of  the scheme \textbf{US} will be deduced. 
We will denote by   $C,C_i,K_i$ to different positive constants possibly  depending on the continuous data $(u_0, {\boldsymbol \sigma}_0,\Omega)$, but independent of time size $k$ and time step $n$.

\begin{tma} \label{welem} {\bf(Uniform weak estimates for $(u_n,{\boldsymbol \sigma}_n)$) }
Let $(u_n,{\boldsymbol
\sigma}_n)$ be a solution of  the scheme \textbf{US}. Then, the following estimates hold
\begin{equation}\label{weak01}
\Vert (u_n, {\boldsymbol \sigma}_n)\Vert_{0}^{2}
%+k^2 \underset{m=1}{\overset{n}{\sum}}\Vert (\delta_t u_m,\delta_t {\boldsymbol \sigma}_m) \Vert_0^2
+ k \underset{m=1}{\overset{n}{\sum}}\Vert (\hat u_m, {\boldsymbol \sigma}_m)\Vert_{1}^2\leq K_0, \ \ \ \forall n\geq 1.
\end{equation}
%\begin{equation}\label{weak02}
%k \underset{m=n_0+1}{\overset{n_0+n}{\sum}}\Vert  (u_m , {\boldsymbol
%\sigma}_m)\Vert_{1}^2 \leq C_0 + C_1 (nk),  \ \ \ \forall n\geq 1,
%\end{equation}
%where $n_0 \geq 0$ is any integer.
%and $C_0,C_1$ are  positive constants depending on the data $(u_0, {\boldsymbol \sigma}_0)$ and $(\Omega, u_0)$ respectively, but independent of $n_0$, $k$ and $n$.
\end{tma}
\begin{proof}
It suffices  adding (\ref{lawenerfydisce}) from $m=1,\cdots,n$ and bound the initial energy.
%\begin{eqnarray*}\label{nolin2}
%& \displaystyle\frac{1}{4}\Vert (u_n, {\boldsymbol\sigma}_n)\Vert_{0}^{2}
%+ \frac{k^2}{4} \underset{m=1}{\overset{n}{\sum}}\Vert (\delta_t u_m,\delta_t {\boldsymbol \sigma}_m)\Vert_{0}^2 
%+\displaystyle\frac{k}{2} \underset{m=1}{\overset{n}{\sum}}\Vert (\nabla u_m,  {\boldsymbol \sigma}_m)\Vert_{L^2\times H^1}^2 
%\leq \displaystyle\frac{1}{2}\Vert (u_0,{\boldsymbol \sigma}_0)\Vert_{0}^{2},
%\end{eqnarray*}
%which implies (\ref{weak01}). 
%Moreover, starting again from (\ref{lawenerfydisce}), 
%but now summing for $m$ from $n_0+1$ to $n_0+n$, using (\ref{weak01}) and the Poincar\'e inequality for the zero-mean value function $u_m -m_0$, where $ \displaystyle m_0= \frac{1}{\vert\Omega\vert}\int_\Omega u_0 =  \frac{1}{\vert\Omega\vert}\int_\Omega u_m$,  one has 
%$$\displaystyle k \underset{m=n_0+1}{\overset{n_0+n}{\sum}}\Vert  (u_m - m_0 ,  {\boldsymbol
%\sigma}_m)\Vert_{1}^2\leq C_0,$$ 
%and thus,  (\ref{weak02}) follows.
\end{proof}

%\begin{obs}
%If we consider a fully discrete scheme corresponding to a Finite Element approximation of \textbf{US}, that is, if we take any finite-dimensional subspaces  $U_h\subset H^1(\Omega)$ and ${\boldsymbol\Sigma}_h\subset \H^1_{\sigma}(\Omega)$ instead of $H^1(\Omega)$ and $ \H^1_{\sigma}(\Omega)$ respectively, then the proofs of conservation, solvability (without positivity),  unconditional energy-stability and uniform weak estimates of the scheme \textbf{US} (that is Theorem \ref{USus}, Lemma \ref{estinc1} and Theorem \ref{welem}, respectively) can be reproduced almost line by line.
%\end{obs}

\begin{cor} {\bf(Uniform strong estimates for $v_n$) }\label{CorvL1}
If $v_n=v_n(u_n^2)$ is the solution of (\ref{edovf}) and $\hat v_n=v_n-m_0^2$, then
% following estimates hold  
%where $K_0>0$  depends on the data $(u_0,{\boldsymbol \sigma}_0,v_0)$, but it  is independent of $k$ and $n$. 
%Moreover, 
\begin{equation}\label{weak01v}
\Vert v_n\Vert_{1}^{2} + 
 k \underset{m=1}{\overset{n}{\sum}}\Vert  \hat v_m \Vert_{2}^2 \leq K_1,  \ \ \ \forall n\geq 1.
\end{equation}
%with  $K_1>0$  depending on the data $(u_0,{\boldsymbol \sigma}_0,v_0,\Omega)$, but independent of $n_0$, $k$ and $n$.
\end{cor}
\begin{proof}
Rewriting (\ref{edovf}) as 
$$
 \delta_t \hat v_n -\Delta \hat v_n + \hat v_n =
 (\hat u_n+2 \, m_0) \hat u_n
$$
and testing by $\hat v_n$ 
%and using \eqref{weak01}, 
one has 
$$
 \delta_t  \Vert \hat v_n \Vert_{0}^2 +\Vert \hat  v_n \Vert_{1}^2  
 \le C \Vert   (\hat u_n+2 \, m_0) \Vert_{L^{3/2}}^2 \Vert  \hat u_n \Vert_{L^6}^2  
  \le C  \Vert  \hat u_n \Vert_1^2.
  % \in l^1(0,+\infty).
$$
By adding from $m=1,\cdots,n$ and using \eqref{weak01}, one has 
$$
\Vert 
 \hat v_n\Vert_{0}^{2} +  k \underset{m=1}{\overset{n}{\sum}}\Vert  \hat v_m \Vert_{1}^2 \leq K,  \ \ \ \forall n\geq 1.
$$
%that $\hat v_n$ is bounded in $l^{\infty}(0,+\infty;L^2(\Omega))\cap l^2(0,+\infty;H^1(\Omega))$ 
%Therefore, Lemma \ref{tmaD} implies 
%\begin{equation}\label{weak02UV}
%\Vert v_n \Vert_{L^1}\leq K_0 \ \ \ \forall n\geq 0.
%\end{equation}
 Finally, since ${\boldsymbol\sigma}_n=\nabla v_n$, then (\ref{weak01v}) holds from (\ref{weak01}) and this last estimate. 
\end{proof}

\subsubsection{Convergence to weak solutions}
Starting from the previous stability estimates, following the ideas of \cite{Tem}  the convergence towards weak solutions as $k\rightarrow 0$ can be proved. For this, let us to introduce the functions:
\begin{itemize}
\item $(\widetilde{u}_{k},\widetilde{\boldsymbol\sigma}_{k})$ are continuous functions on $[0,+\infty)$, linear on each interval $(t_{n-1},t_n)$ and equal to $(u_n,{\boldsymbol\sigma}_n)$ at $t=t_n$, for all $n\ge 0$;
\item $({u}_{k},{\boldsymbol\sigma}_{k} )$ are the piecewise constant functions taking values $(u_{n},{\boldsymbol\sigma}_n)$ on $(t_{n-1},t_n]$, for all $n\ge 1$.
\end{itemize}
\begin{tma} {\bf (Convergence of  $(u,{\boldsymbol\sigma})$)}\label{TCon}
There exist a subsequence $(k')$ of $(k)$, with $k'\downarrow 0$, and a weak solution $(u,{\boldsymbol\sigma})$ of (\ref{modelf01}) in $(0,+\infty)$, such that $(\widetilde u_{k'}-m_0,\widetilde{\boldsymbol\sigma}_{k'})$ and $(u _{k'}-m_0,{\boldsymbol\sigma} _{k'})$ converge to $(u-m_0,{\boldsymbol\sigma})$ weakly-$\star$ in $L^\infty(0,+\infty;L^2(\Omega)\times \L^2(\Omega))$, weakly in $L^2(0,+\infty;H^1(\Omega)\times \H^1(\Omega))$ and strongly in $L^2(0,T;L^2(\Omega)\times \L^2(\Omega))$, for any $T>0$. 
\end{tma}
\begin{proof}
Observe that (\ref{modelf02}) can be rewritten as:
\begin{equation}
\left\{
\begin{array}
[c]{lll}%
\vspace{0,2 cm}\displaystyle\Big(\frac{d}{dt}\widetilde{u}_k(t),\overline{u}\Big) + (\nabla u _k(t), \nabla \overline{u}) +(u_k (t){\boldsymbol \sigma}_k (t),\nabla \overline{u})=0, \ \ \forall \overline{u}=\overline{u}(t), \ \ \mbox{for } t\in [0,+\infty)\setminus\{t_n\},\\ 
\displaystyle\Big(\frac{d}{dt}\widetilde{\boldsymbol \sigma}_k(t),\overline{\boldsymbol \sigma}\Big) +
\langle B {\boldsymbol \sigma} _k(t), \overline{\boldsymbol \sigma}\rangle  -
2(u _k(t)\nabla u _k(t),\overline{\boldsymbol \sigma})=0,\ \ \forall
\overline{\boldsymbol \sigma}=\overline{\boldsymbol \sigma}(t),  \ \ \mbox{for } t\in [0,+\infty)\setminus\{t_n\},
\end{array}
\right.  \label{modelf02convA}
\end{equation}
where $\overline{u}(t)|_{I_n}=\overline{u}_n \in H^1(\Omega)$ and $\overline{\boldsymbol \sigma}(t)|_{I_n}=\overline{\boldsymbol \sigma}_n \in \H^1_\sigma(\Omega)$,  with $I_n:=[t_{n-1},t_n]$.  From Theorem \ref{welem} we have that
$(\widetilde{u}_{k},\widetilde{\boldsymbol\sigma}_{k})$ and $({u}_{k} ,{\boldsymbol\sigma}_{k} )$ are bounded in $L^\infty(0,+\infty; L^2(\Omega) \times \L^2(\Omega)) \cap L^2(0,T; H^1(\Omega) \times \H^1(\Omega))$. Moreover, using (\ref{lawenerfydisce}), it is not difficult to prove that $\widetilde{u}_{k} - {u}_{k} $ and $\widetilde{\boldsymbol\sigma}_{k} - {\boldsymbol\sigma}_{k} $ converge to $0$ in $L^2(0,T;L^2(\Omega))$ as $k\rightarrow 0$, for any $T>0$. More precisely, we have $\Vert\widetilde{u}_{k} - {u}_{k} ,\widetilde{\boldsymbol\sigma}_{k} - {\boldsymbol\sigma}_{k} \Vert_{L^2(0,T;L^2(\Omega)) }\le (C_0 k/3)^{1/2}$. Therefore, there exist a subsequence $(k')$ of $(k)$ and limit functions $u$ and ${\boldsymbol\sigma}$ verifying the following convergence as $k' \rightarrow 0$:
\begin{equation*}
(\widetilde{u}_{k'}-m_0,\widetilde{\boldsymbol\sigma}_{k'}) , \ \  ({u}_{k'}-m_0 ,{\boldsymbol\sigma}_{k'} )  \rightarrow (u-m_0,{\boldsymbol\sigma}) \ \ \mbox{ in }  \left\{\begin{array}{l}
L^\infty(0,+\infty; L^2(\Omega) \times \L^2(\Omega))\mbox{-weak*}\\
L^2(0,+\infty; H^1(\Omega) \times \H^1(\Omega))\mbox{-weak}.
\end{array}\right. 
\end{equation*}
On the other hand, one can deduce that $\frac{d}{dt}(\widetilde{u}_{k'},\widetilde{\boldsymbol\sigma}_{k'})$ is bounded in $L^{4/3}(0,T;H^1(\Omega)'\times \H^1(\Omega)')$. Therefore,  a compactness result of Aubin-Lions type implies that the sequence $(\widetilde{u}_{k'},\widetilde{\boldsymbol\sigma}_{k'})$ is compact in $L^2(0,T;L^2(\Omega)\times \L^2(\Omega))$ for any $T>0$. 
 This implies the strong convergence of both subsequences $(\widetilde{u}_{k},\widetilde{\boldsymbol\sigma}_{k})$ and $({u}_{k} ,{\boldsymbol\sigma}_{k} )$ in $L^2(0,T;L^2(\Omega)\times \L^2(\Omega))$; and passing to the limit in  (\ref{modelf02convA}), one obtains that $(u,{\boldsymbol\sigma})$ satisfies (\ref{wf01US})-(\ref{wf01bUS}). Now, in order to obtain that $(u,{\boldsymbol\sigma})$ satisfies the energy inequality (\ref{wsdUS}), we test (\ref{modelf02convA})$_1$ by $\overline{u}=u_k (t)$ and (\ref{modelf02convA})$_2$ by $\overline{\boldsymbol \sigma}=\frac12{\boldsymbol \sigma}_k (t)$, and taking into account that $\widetilde{u}_k |_{I_n}= u_n + \frac{t_n  - t}{k} (u_{n-1} - u_n)$ (and $\widetilde{\boldsymbol\sigma}_k$ is defined in the same way), one has 
\begin{equation*}
\frac{d}{dt}\left( \frac{1}{2}\Vert \widetilde{u}_k(t) \Vert_0^2 + \frac{1}{4} \Vert\widetilde{\boldsymbol \sigma}_k(t)\Vert_0^2\right) + (t_n - t)\Vert (\delta_t u_n ,\delta_t {\boldsymbol\sigma}_n) \Vert_0^2 +
  \Vert \nabla u_k (t)\Vert_{0}^{2} +
\displaystyle\frac{1}{2}\Vert {\boldsymbol \sigma}_k (t)\Vert_{1}^{2}= 0,
\end{equation*}
for any $t\in (t_{n-1},t_n)$, which implies that
\begin{equation}\label{lawenerfydisceCONV}
\frac{d}{dt}\mathcal{E}(\widetilde{u}_k(t),\widetilde{\boldsymbol \sigma}_k(t)) + 
  \Vert \nabla u_k (t)\Vert_{0}^{2} +
\displaystyle\frac{1}{2}\Vert {\boldsymbol \sigma}_k (t)\Vert_{1}^{2}\leq 0, \ \ \mbox{for } t\in [0,+\infty)\setminus\{t_n\}.
\end{equation}
Then, integrating (\ref{lawenerfydisceCONV}) in time from $t_0$ to $t_1$, with any $t_0<t_1\in [0,+\infty)$, and taking into account that 
$$
\int_{t_0}^{t_1} \displaystyle\frac{d}{dt}\mathcal{E}(\widetilde{u}_k(t),\widetilde{\boldsymbol \sigma}_k(t))= \mathcal{E}(\widetilde{u}_k(t_1),\widetilde{\boldsymbol \sigma}_k(t_1)) - \mathcal{E}(\widetilde{u}_k(t_0),\widetilde{\boldsymbol \sigma}_k(t_0)) \quad  \forall t_0<t_1,
$$
since $\mathcal{E}(\widetilde{u}_k(t),\widetilde{\boldsymbol \sigma}_k(t))$ is continuous in time, we deduce
\begin{equation}\label{lawenerfydisceCONVnewa}
\mathcal{E}(\widetilde{u}_k(t_1),\widetilde{\boldsymbol \sigma}_k(t_1)) - \mathcal{E}(\widetilde{u}_k(t_0),\widetilde{\boldsymbol \sigma}_k(t_0))  + \int_{t_0}^{t_1} (  \Vert \nabla u_k (t)\Vert_{0}^{2} +
\displaystyle\frac{1}{2}\Vert {\boldsymbol \sigma}_k (t)\Vert_{1}^{2}) dt \leq 0, \quad \forall t_0<t_1.
\end{equation}
Now, we will prove that 
\begin{equation}\label{FIN}
\mathcal{E}(\widetilde{u}_{k'}(t),\widetilde{\boldsymbol \sigma}_{k'}(t)) \rightarrow \mathcal{E}({u}(t),{\boldsymbol \sigma}(t)) \ \ \mbox{a.e.} \  t\in[0,+\infty).
\end{equation}
Indeed, for any $T>0$,
\begin{eqnarray}\label{cc00}
 & \Vert \mathcal{E}&\!\!\!\!\!( \widetilde{u}_{k'}(t),\widetilde{\boldsymbol \sigma}_{k'}(t))  - \mathcal{E}({u}(t),{\boldsymbol \sigma}(t)) \Vert_{L^1(0,T)} =\displaystyle\int_0^T  \vert \mathcal{E}(\widetilde{u}_{k'}(t),\widetilde{\boldsymbol \sigma}_{k'}(t)) - \mathcal{E}({u}(t),{\boldsymbol \sigma}(t)) \vert dt\nonumber\\
 && \!\!\!\!  =  \int_0^T \left\vert  \frac{1}{2}\left(\Vert \widetilde{u}_{k'}(t) \Vert_0^2- \Vert {u}(t) \Vert_0^2 \right) + \frac{1}{4} \left( \Vert\widetilde{\boldsymbol \sigma}_{k'}(t)\Vert_0^2 -  \Vert{\boldsymbol \sigma}(t)\Vert_0^2\right)  \right\vert dt\nonumber\\
 && \!\!\!\!  \leq \frac{1}{2}  \Vert \widetilde{u}_{k'}  -  {u} \Vert_{L^2(0,T;L^2)} 
 \Big(\Vert \widetilde{u}_{k'}  \Vert_{L^2(0,T;L^2)} + \Vert {u}\Vert_{L^2(0,T;L^2)} \Big)
 \nonumber\\
 && \!\!\!\! + \frac{1}{4}\Vert\widetilde{\boldsymbol \sigma}_{k'}- {\boldsymbol \sigma}\Vert_{L^2(0,T;L^2)}  
 \Big(\Vert\widetilde{\boldsymbol \sigma}_{k'}\Vert_{L^2(0,T;L^2)} + \Vert{\boldsymbol \sigma}\Vert_{L^2(0,T;L^2)} \Big).
\end{eqnarray}
Taking into account that $(\widetilde{u}_{k'},\widetilde{\boldsymbol \sigma}_{k'})\rightarrow (u,{\boldsymbol \sigma})$ strongly in $L^2(0,T;L^2(\Omega))$ for any $T>0$, from (\ref{cc00}) we conclude that $\mathcal{E}(\widetilde{u}_{k'}(t),\widetilde{\boldsymbol \sigma}_{k'}(t)) \rightarrow \mathcal{E}({u}(t),{\boldsymbol \sigma}(t))$ strongly in $L^1(0,T)$ for all $T>0$, which implies (\ref{FIN}). In particular,  owing to (\ref{FIN}), one has
\begin{eqnarray*}
\underset{k'\rightarrow 0}{\lim \mbox{inf} }\ \Big[\mathcal{E}(\widetilde{u}_{k'}(t_1),\widetilde{\boldsymbol \sigma}_{k'}(t_1)) - \mathcal{E}(\widetilde{u}_{k'}(t_0),\widetilde{\boldsymbol \sigma}_{k'}(t_0))\Big]  = \mathcal{E}({u}(t_1),{\boldsymbol \sigma}(t_1)) - \mathcal{E}({u}(t_0),{\boldsymbol \sigma}(t_0)),
\end{eqnarray*}
a.e.~$t_1,t_0: t_1\geq t_0\geq 0$. On the other hand, since $(u _{k'},{\boldsymbol\sigma} _{k'})\rightarrow (u,{\boldsymbol\sigma})$ weakly in $L^2(0,T;H^1(\Omega)\times \H^1(\Omega))$, one has 
$$
\underset{k'\rightarrow 0}{\lim \mbox{inf} } \int_{t_0}^{t_1} (  \Vert \nabla u_{k'} (t)\Vert_{0}^{2} +
\displaystyle\frac{1}{2}\Vert {\boldsymbol \sigma}_{k'} (t)\Vert_{1}^{2}) dt \ge \int_{t_0}^{t_1} (  \Vert \nabla u(t)\Vert_{0}^{2} +
\displaystyle\frac{1}{2}\Vert {\boldsymbol \sigma}(t)\Vert_{1}^{2}) dt \quad \ \forall t_1\geq t_0\geq 0.
$$
Finally, taking  $\liminf$ as $k'\rightarrow 0$ in the inequality (\ref{lawenerfydisceCONVnewa}),  the energy inequality (\ref{wsdUS})  is deduced.
\end{proof}

Analogously, if we introduce the functions:
\begin{itemize}
\item $\widetilde{v}_{k}$ are continuous functions on $[0,+\infty)$, linear on each interval $(t_{n-1},t_n)$ and equal to $v_n,$ at $t=t_n$, $n\ge 0$;
\item ${v}_{k} $ are the piecewise constant functions taking values $v_{n}$ on $(t_{n-1},t_n]$, $n\ge 1$,
\end{itemize}
the following result can be proved:
\begin{lem}{\bf (Convergence of  $v$)}
There exist a subsequence $(k')$ of $(k)$, with $k'\downarrow 0$, and a strong solution $v$ of (\ref{modelf01eqv}) in $(0,+\infty)$, such that $\widetilde{v}_{k'}-m_0^2$ and $v _{k'} -m_0^2$ converge to $v -m_0^2$ weakly-$\star$ in $L^\infty(0,+\infty;H^1(\Omega))$, weakly in $L^2(0,{\color{blue}+\infty};H^2(\Omega))$ and strongly in $L^2(0,T;H^1(\Omega))\cap C([0,T];L^p(\Omega))$, for $1\leq p <6$ and any $T>0$.
\end{lem}
\begin{proof}
Observe that (\ref{edovf}) can be rewritten as:
\begin{equation}
\displaystyle\frac{d}{dt}\widetilde{v}_k(t) +  A  v _k(t) = (u _k(t))^2 \ \ \mbox{for } t\in [0,+\infty)\setminus\{t_n\}. \label{modelf02convAv}
\end{equation}
From estimate (\ref{weak01v}), 
$\widetilde{v}_{k}$ and ${v}_{k} $ are bounded in $L^\infty(0,+\infty; H^1(\Omega)) \cap L^2(0,T; H^2(\Omega))$. Moreover, it is not difficult to prove that $\widetilde{v}_{k} - {v}_{k} \to 0$ in $L^2(0,T;H^1(\Omega))$ as $k\rightarrow 0$, for any $T>0$. Therefore, there exist a subsequence $(k')$ of $(k)$ and a limit function $v$ satisfying the following convergence as $k' \rightarrow 0$:
\begin{equation*}
\widetilde{v}_{k'} -m_0^2, \ \  {v}_{k'}-m_0^2  \rightarrow v -m_0^2 \ \ \mbox{ in }  \left\{\begin{array}{l}
L^\infty(0,+\infty; H^1(\Omega))\mbox{-weak*}\\
L^2(0,+\infty; H^2(\Omega))\mbox{-weak}.
\end{array}\right. 
\end{equation*}
On the other hand, one can deduce that $\frac{d}{dt}\widetilde{v}_{k'}$ is bounded in $L^{4/3}(0,T;L^2(\Omega))$. Therefore,  a compactness result of Aubin-Lions type implies that $\widetilde{v}_{k'} \rightarrow v$ in $L^2(0,T;H^1(\Omega))\cap C([0,T];L^p(\Omega))$, for $1\leq p <6$. 
 As before, this  strong convergence is satisfied by both subsequences $\widetilde{v}_{k}$ and ${v}_{k} $.
 %in $L^2(0,T;H^1(\Omega))\cap C([0,T];L^p(\Omega))$; 
 Finally,  passing to the limit in  (\ref{modelf02convAv}), we obtain that $v$ satisfies (\ref{wf02}).
\end{proof}
\subsection{Uniform strong estimates}\label{UES}
In this section, global in time  estimates in strong and more regular  norms for $(u_n,{\boldsymbol \sigma}_{n})$ any solution of the scheme \textbf{US} will be obtained. 
We will again denote by   $C,C_i,K_i$ to different positive constants possibly  depending on the continuous data $(u_0, {\boldsymbol \sigma}_0,\Omega)$, but independent of time size $k$ and time step $n$.

 First, we are going to show $H^2$-regularity for $(u_n,{\boldsymbol \sigma}_{n})$.
\begin{prop} \label{regUS}
Let $(u_{n-1},{\boldsymbol \sigma}_{n-1})\in H^{1}(\Omega)\times \H^{1}_{\sigma}(\Omega)$. If $(u_n,{\boldsymbol \sigma}_{n})\in H^{1}(\Omega)\times \H^{1}_{\sigma}(\Omega)$ is any solution of the scheme \textbf{US}, then $(u_n,{\boldsymbol \sigma}_{n})\in H^{2}(\Omega)\times \H^{2}(\Omega)$. 
Moreover, the following estimate holds
\begin{equation}\label{H2est1}
\Vert (\hat u_n,{\boldsymbol\sigma}_n)\Vert_{2}
\leq C\left( \Vert (\delta_t \hat u_n, \delta_t {\boldsymbol\sigma}_n)\Vert_{0} + \Vert (\hat u_n,{\boldsymbol\sigma}_n)\Vert_{1}^3 + \Vert (\hat u_n,{\boldsymbol\sigma}_n) \Vert_1\right).
\end{equation}
%where $C$ is a constant depending on data $(\Omega, u_0, {\boldsymbol\sigma}_0)$, but independent of $k$ and $n$. 
\end{prop}
\begin{proof}
Recall that we have  assumed the $H^2$ and $H^3$-regularity of problem (\ref{B001}) and  (\ref{B101}), 
which imply the $H^2$-regularity of problem (\ref{B201}) in the case of ${\textbf {\textit f}}=\nabla h$ for some $h\in H^1(\Omega)$. Moreover, observe that scheme \textbf{US} can be rewritten in terms of $(\hat u_n,{\boldsymbol \sigma}_{n})$ and the operators $\widehat A$ and $B$
(defined in (\ref{B001}) and in (\ref{B201}), respectively)
 as 
\begin{align*} \label{problema1}
\left\{ \begin{array}{lll}
 \displaystyle \widehat A \hat u_n = -\delta_t \hat u_n + \nabla\cdot((\hat u_n+m_0){\boldsymbol \sigma}_{n}),  \vspace{0,2 cm}\\
 \displaystyle B {\boldsymbol \sigma}_{n} =\nabla( -\delta_t v_{n} + (\hat u_n+m_0)^2).
 \end{array}\right.
\end{align*}
Now, since $(\hat u_n,{\boldsymbol \sigma}_{n})\in H^{1}(\Omega)\times \H^{1}_{\sigma}(\Omega) \hookrightarrow {L}^{6}(\Omega)\times \L^{6}(\Omega)$, 
then 
%$\nabla u_n \in \L^2(\Omega)$, $\nabla \cdot {\boldsymbol \sigma}_{n} \in L^2(\Omega)$ and, 
%from Sobolev embeddings, $(u_n,{\boldsymbol \sigma}_{n})\in {L}^{6}(\Omega)\times \L^{6}(\Omega)$. 
%Consequently, we get  
$\nabla\cdot ((\hat u_n+m_0){\boldsymbol \sigma}_{n} ) 
%= {\boldsymbol \sigma}_{n}\cdot \nabla u_n + u_n \nabla\cdot {\boldsymbol \sigma}_{n} 
\in L^{3/2}(\Omega)$,
% and, using that $u_n,\delta_t u_{n}\in L^2(\Omega)\hookrightarrow L^{3/2}(\Omega)$, 
hence from classical elliptic regularity  $\hat u_n \in W^{2,3/2}(\Omega)$. 
In particular,  since $\hat u_n\in W^{2,3/2}(\Omega)$, then $\nabla \hat u_n \in \W^{1,3/2}(\Omega)\hookrightarrow \L^3(\Omega)$, and therefore $\nabla  (\hat u_n+m_0)^2\in \L^2(\Omega)$. Thus, using 
  the $H^2$-regularity of problem (\ref{B201}), 
 %the fact that $\delta_t {\boldsymbol \sigma}_{n}\in \L^2(\Omega)$ 
%and taking into account that $-\delta_t{\boldsymbol \sigma}_{n} + 2u_n\nabla u_n= \nabla(-\delta_t v_{n} +  u_n^2)$,
%with $-\delta_t v_{n} +  u_n^2 \in H^1(\Omega)$, 
one has that ${\boldsymbol \sigma}_{n}\in \H^2(\Omega)$. Finally, taking into account  that ${\boldsymbol \sigma}_{n}\in \H^2(\Omega)$ and $\hat u_n \in W^{2,3/2}(\Omega)$, then $\nabla\cdot ((\hat u_n+m_0){\boldsymbol \sigma}_{n} ) \in L^{2}(\Omega)$, 
%and  since $u_n, \delta_t u_{n}\in L^2(\Omega)$, 
hence we conclude that $\hat u_n\in H^2(\Omega)$. 

Besides, from (\ref{H2us})-(\ref{eH201}), the interpolation inequality (\ref{in3D}), and using the H\"older and Young inequalities, one has
\begin{eqnarray}\label{H2001}
&\Vert \hat u_n\Vert_{2}&\!\!\!\!\leq C\left( \Vert \delta_t \hat u_{n} \Vert_{0} + \Vert \hat u_n \nabla \cdot {\boldsymbol \sigma}_{n}\Vert_{0} + \Vert  {\boldsymbol \sigma}_{n} \cdot \nabla \hat u_n\Vert_{0}  +\Vert  {\boldsymbol \sigma}_{n} \Vert_1 \right)\nonumber\\
&& \!\!\!\!\leq C (\Vert \delta_t \hat u_{n} \Vert_{0}+\Vert {\boldsymbol \sigma}_{n}\Vert_1 )+ \displaystyle\frac{1}{2}\Vert {\boldsymbol \sigma}_{n}\Vert_{2}+ \frac{1}{4}\Vert \hat u_n\Vert_{2}  + C\Vert \hat u_n\Vert_{1}^2 \Vert {\boldsymbol \sigma}_{n}\Vert_{1} +C\Vert  \hat u_n\Vert_{1} \Vert  {\boldsymbol \sigma}_{n}\Vert_{1}^2, 
\end{eqnarray}
\begin{equation}\label{H2002}
\Vert {\boldsymbol \sigma}_n\Vert_{2}
\leq C\left( \Vert \delta_t {\boldsymbol \sigma}_n \Vert_{0} 
+ \Vert \hat  u_n \nabla \hat u_{n}\Vert_{0}  + \Vert \hat  u_n \Vert_{1}\right) 
%
%\leq C\left( \Vert \delta_t {\boldsymbol \sigma}_n \Vert_{0} 
%+ \Vert \hat  u_n\Vert_{L^6} \Vert \nabla \hat u_{n}\Vert_{L^3}
%+ 
% \right)
%\nonumber\\ && \!\!\!\!
\leq C \left(\Vert \delta_t {\boldsymbol \sigma}_n \Vert_{0}  + \Vert \hat  u_n \Vert_{1}\right) 
+ \displaystyle\frac{1}{4}\Vert \hat  u_n\Vert_{2} 
+ C\Vert  \hat u_n\Vert_{1}^3.
\end{equation}
Then, adding (\ref{H2001}) and (\ref{H2002}), one arrives at (\ref{H2est1}).
\end{proof}

From Proposition \ref{regUS}, 
%we have $(u_n,{\boldsymbol \sigma}_{n})\in H^{2}(\Omega)\times \H^{2}(\Omega)$, 
we  can consider the following pointwise differential formulation of  the scheme \textbf{US}:
\begin{equation}  \label{fuertes10}
\left\{
\begin{array}
[c]{lll}%
\delta_t \hat u_n + \widehat A \hat u_n = \nabla\cdot ((\hat u_n+m_0){\boldsymbol \sigma}_n)  \ \ \mbox{ a.e.~$\x \in \Omega$},\\
\delta_t {\boldsymbol \sigma}_n +B {\boldsymbol \sigma}_n=
2 (\hat u_n+m_0) \nabla \hat u_n  \ \ \mbox{ a.e.} \ \x \in \Omega.
\end{array}
\right.  
\end{equation}

\begin{lem}\label{Lm:StrongIneq-n} {\bf(Strong inequality for $(u_n,{\boldsymbol\sigma}_n)$) }
It holds 
\begin{equation}\label{StrongIneq-n}
\delta_t
\Vert  (\hat u_n, {\boldsymbol\sigma}_n) \Vert_{1}^2 
 + \Vert (\hat u_n, {\boldsymbol\sigma}_n)\Vert_{2}^2
 +  \Vert  (\delta_t \hat  u_n, \delta_t {\boldsymbol\sigma}_n) \Vert_{0}^2
\le
 C_1 \Big(\Vert  (\hat u_n, {\boldsymbol\sigma}_n)  \Vert_{1}^2\Big)^d +
 C_2 \Vert  (\hat u_n, {\boldsymbol\sigma}_n) \Vert_{1}^2 
\end{equation}
where $d=2$ for $2D$ domains and $d=3$ for $3D$ domains.
\end{lem}
\begin{proof}
By testing \eqref{fuertes10}$_1$ by $\widehat A \hat u_n + \delta_t \hat u_n $ and \eqref{fuertes10}$_2$ by  $B {\boldsymbol\sigma}_n + \delta_t   {\boldsymbol\sigma}_n$,  bounding the right hand side using either \eqref{in2D} in $2D$ or \eqref{in3D} in $3D$, the $H^2$-inequality $\Vert ( \hat u_{n},  {\boldsymbol \sigma}_{n})\Vert_{2} \le C\, \Vert (\widehat A \hat u_{n}, B {\boldsymbol \sigma}_{n})\Vert_{0}$ and the Young inequality, 
 one has
\begin{eqnarray*}\label{fuertesb}
&&\displaystyle  \delta_t \Vert (\hat u_n,{\boldsymbol \sigma}_n)\Vert_{1}^2
+ \Vert (\delta_t \hat u_{n}, \delta_t {\boldsymbol \sigma}_{n})\Vert_{0}^2
+ \Vert (\widehat A \hat u_{n}, B {\boldsymbol \sigma}_{n})\Vert_{0}^2 
%+ \frac{k}{2} \Vert (\delta_t u_n,\delta_t {\boldsymbol \sigma}_n)\Vert_{1}^2
\\
 &&\quad \leq C(\Vert  (\hat u_n+m_0) ( {\boldsymbol \sigma}_n + \hat u_n)\Vert_1^2  
 %+ \Vert 2(\hat u_n+m_0)\nabla \hat u_n\Vert_0^2).
   \leq C\Big(\Vert (\hat u_n+m_0,{\boldsymbol\sigma}_n)\Vert^2_1\Big)^{d/2}\Vert (\hat u_n,{\boldsymbol\sigma}_n)\Vert_{2}
 \\
&& \quad \le \frac12 \Vert (\widehat A \hat u_{n}, B {\boldsymbol \sigma}_{n})\Vert_{0}^2
 +C_1 \Big(\Vert  (\hat u_n, {\boldsymbol\sigma}_n)  \Vert_{1}^2\Big)^d + C_2 \Vert (\hat u_n, {\boldsymbol\sigma}_n) \Vert_{1}^2 
\end{eqnarray*}
hence \eqref{StrongIneq-n} holds.
\end{proof}

\begin{cor}\label{Cor:StrongReg-n} {\bf(Strong estimates for $(u_n,{\boldsymbol\sigma}_n)$) }
Let $(u_0, v_0)\in H^1(\Omega)\times H^2(\Omega)$. Assuming 
 the following regularity criterion:
\begin{equation}\label{strong01}
\Vert (u_n, {\boldsymbol\sigma}_n)\Vert_{1}^{2}\leq K_0, \ \ \ \forall n\geq 0,
\end{equation}
then the following  estimate holds
\begin{equation}\label{ssa-n}
%\Vert (u_n, {\boldsymbol \sigma}_n)\Vert_{1}^{2}
%+k^2 \underset{m=1}{\overset{n}{\sum}}\Vert (\delta_t u_m,\delta_t {\boldsymbol \sigma}_m) \Vert_0^2
 k \underset{m=1}{\overset{n}{\sum}}
\Big(\Vert (\hat u_m, {\boldsymbol \sigma}_m)\Vert_{2}^2
+\Vert (\delta_t\hat u_m, \delta_t {\boldsymbol \sigma}_m)\Vert_{0}^2
\Big)\leq K_1, \ \ \ \forall n\geq 1.
\end{equation}
\end{cor}

\begin{proof}
From \eqref{StrongIneq-n} and \eqref{strong01}
$$
\delta_t
\Vert  (\hat u_n, {\boldsymbol\sigma}_n) \Vert_{1}^2 
 + \Vert  (\hat u_n, {\boldsymbol\sigma}_n) \Vert_{2}^2
 +  \Vert  (\delta_t \hat  u_n, \delta_t {\boldsymbol\sigma}_n) \Vert_{0}^2
\le C\, \Vert  (\hat u_n, {\boldsymbol\sigma}_n) \Vert_{1}^2 
$$
 hence by applying \eqref{weak01} the estimate  \eqref{ssa-n} is deduced.
\end{proof}

 Corollaries \ref{CorvL1} and  \ref{Cor:StrongReg-n} and the equality ${\boldsymbol \sigma}_n=\nabla v_n$ yield to 
%regular estimates for $v_n$.
\begin{cor} {\bf(Strong estimates for $v_n$)}
Let $v_n$ be the solution of (\ref{edovf}). Under the hypothesis of Corollary  \ref{Cor:StrongReg-n}, the following estimate holds
\begin{equation*}
\Vert v_n\Vert_{2}^{2} + 
k \underset{m=1}{\overset{n}{\sum}}\Vert \hat v_m\Vert_{3}^2\leq K_2,  \ \ \ \forall n\geq 1.
\end{equation*}
\end{cor}

\subsection{Uniform regular estimates}\label{UES-bis}
Denote by $\widetilde{u}_n=\delta_t u_n$ and $\widetilde{\boldsymbol \sigma}_n=\delta_t {\boldsymbol \sigma}_n$. Then, by making the time discrete derivative of (\ref{modelf02}), and using the equality $\delta_t(a_nb_n)=(\delta_t a_n) b_{n-1} +  a_n (\delta_t b_{n})$, we obtain that $(\widetilde{u}_n,\widetilde{\boldsymbol \sigma}_n)$ satisfies
\begin{equation}
\left\{
\begin{array}
[c]{lll}%
(\delta_t \widetilde{u}_n,\overline{u}) + (\nabla \widetilde{u}_n, \nabla \overline{u}) + (\widetilde{u}_n{\boldsymbol \sigma}_{n-1},\nabla \overline{u})+(u_n\widetilde{\boldsymbol \sigma}_n,\nabla \overline{u})=0,\quad \forall \,\overline{u}\in H^{1}(\Omega),
\\
(\delta_t \widetilde{\boldsymbol \sigma}_n,\overline{\boldsymbol \sigma}) +
\langle B \widetilde{\boldsymbol \sigma}_n, \overline{\boldsymbol \sigma}\rangle  =
2(\widetilde{u}_n\nabla u_{n-1},\overline{\boldsymbol \sigma})+2(u_n\nabla \widetilde{u}_n,\overline{\boldsymbol \sigma}),
\quad \forall \,\overline{\boldsymbol \sigma}\in H^{1}(\Omega).
\end{array}
\right.  \label{masfuert1}
\end{equation}
\begin{lem}\label{Lm:RegIneq-n} {\bf(Weak inequality for $(\widetilde u_n,\widetilde{\boldsymbol\sigma}_n)$). } It holds
\begin{equation}\label{masfuert5}
\displaystyle  \delta_t\left(  \frac{1}{2}\Vert \widetilde{u}_n\Vert^2_{0} +\frac{1}{4}\Vert \widetilde{\boldsymbol \sigma}_n \Vert^2_{0}\right)
%+ \frac{k}{4} \Vert (\delta_t \widetilde{u}_n,\delta_t \widetilde{\boldsymbol \sigma}_n)\Vert^2_{0}
+ 
\displaystyle\frac{1}{4} \Vert (\widetilde{u}_n, \widetilde{\boldsymbol \sigma}_n)\Vert^{2}_{1}  \leq  C\Vert (\widetilde{u}_n, \widetilde{\boldsymbol \sigma}_n) \Vert_{0}^2.
\end{equation}
\end{lem}
\begin{proof}
Taking $\overline{u}=\widetilde{u}_n$ and $\displaystyle \overline{\boldsymbol \sigma}=\frac{1}{2}\widetilde{\boldsymbol \sigma}_n$ in (\ref{masfuert1}) and adding the resulting expressions, the terms $(u_n\widetilde{\boldsymbol \sigma}_n,\nabla \widetilde{u}_n)$ cancel, and taking into account \eqref{H1poin} because $\displaystyle \int_\Omega \widetilde{u}_n=0$, one has
\begin{eqnarray*}\label{masfuert2}
&\displaystyle  \delta_t &\!\!\!\!\!\left(  \frac{1}{2}\Vert \widetilde{u}_n\Vert^2_{0} +\frac{1}{4}\Vert \widetilde{\boldsymbol \sigma}_n \Vert^2_{0}\right)
%+ \frac{k}{4}  \Vert (\delta_t \widetilde{u}_n,\delta_t \widetilde{\boldsymbol \sigma}_n)\Vert^2_{0}
+ \displaystyle\frac{1}{2} \Vert (\widetilde{u}_n, \widetilde{\boldsymbol \sigma}_n)\Vert^{2}_{1} \nonumber\\
&& \!\!\!\!\!\!\! \leq \Vert \widetilde{u}_n \Vert_{L^3}\Vert {\boldsymbol \sigma}_{n-1} \Vert_{L^6} \Vert \nabla \widetilde{u}_n \Vert_{0} + \Vert \widetilde{u}_n \Vert_{L^6}\Vert \widetilde{\boldsymbol \sigma}_n \Vert_{L^3} \Vert \nabla u_{n-1} \Vert_{0}\nonumber\\
&& \!\!\!\!\!\!\! \leq \displaystyle\frac{1}{4} \Vert ( \widetilde{u}_n, \widetilde{\boldsymbol \sigma}_n)
 \Vert_{1}^2
 +  C\,\Vert (\widetilde{u}_n  , \widetilde{\boldsymbol \sigma}_n )\Vert_{0}^2, 
\end{eqnarray*}
where  (\ref{in3D}) and (\ref{strong01}) have been used in the last inequality. Consequently estimate 
\eqref{masfuert5} holds.
\end{proof}

\begin{cor}\label{Cor:RegEst-n} {\bf(Regular estimates for $(u_n,{\boldsymbol\sigma}_n)$) }
Assu\-me that $(u_0,{\boldsymbol\sigma}_0)\in H^2(\Omega)\times \H^2(\Omega)$. Under the hypothesis of Corollary \ref{Cor:StrongReg-n}, the following estimates hold
\begin{equation}\label{stdta}
\Vert (\delta_t u_n, \delta_t {\boldsymbol \sigma}_n)\Vert^2_{0} 
+
 k \underset{m= 1}{\overset{n}{\sum}}\Vert (\delta_t u_m,
\delta_t {\boldsymbol \sigma}_m) \Vert^{2}_{1}
\leq C_1, \ \ \forall n\geq 1,
\end{equation}
\begin{equation}\label{msus01}
\Vert (u_n,{\boldsymbol \sigma}_n)\Vert_{2} \leq C_2, \ \ \forall n\geq 0.
\end{equation}  
\end{cor}

\begin{obs}
In particular, from (\ref{msus01}) one has $
\Vert (u_n,{\boldsymbol \sigma}_n)\Vert_{L^\infty} \leq C_3$ for all $n\geq 0$.
% with $C_7>0$ a constant independent of  $k$ and $n$.
\end{obs}
\begin{proof}
Observe that from (\ref{modelf02}) one has that, for all $(\overline{u},\overline{\boldsymbol \sigma})\in H^1(\Omega)\times\H^{1}_{\sigma}(\Omega)$,
\begin{equation}
\left\{
\begin{array}
[c]{lll}%
(\delta_t u_1,\overline{u}) + (\nabla (u_1-u_0), \nabla \overline{u})+ (\nabla u_0, \nabla \overline{u})  = -(u_1({\boldsymbol \sigma}_1-{\boldsymbol\sigma}_0),\nabla \overline{u})-(u_1{\boldsymbol\sigma}_0,\nabla \overline{u}),\\
(\delta_t {\boldsymbol \sigma}_1,\overline{\boldsymbol \sigma}) +
\langle B ({\boldsymbol \sigma}_1  - {\boldsymbol \sigma}_0), \overline{\boldsymbol
\sigma}\rangle  + \langle B  {\boldsymbol \sigma}_0, \overline{\boldsymbol
\sigma}\rangle =
2(u_1\nabla (u_1-u_0),\overline{\boldsymbol \sigma})+2(u_1\nabla u_0,\overline{\boldsymbol \sigma}).
\end{array}
\right.  \label{edt1}
\end{equation}
Then, testing (\ref{edt1}) by $\overline{u}=\delta_t u_1$, $\overline{\boldsymbol\sigma}=\frac{1}{2}\delta_t {\boldsymbol\sigma}_1$ and adding, the terms $\frac{1}{k}(u_1\nabla (u_1-u_0),{\boldsymbol \sigma}_1 - {\boldsymbol\sigma}_0)$ cancel, and using the H\"older and Young inequalities and (\ref{strong01}), we  deduce
\begin{equation}\label{edt2}
\Vert (\delta_t u_1,\delta_t {\boldsymbol \sigma}_1)\Vert_0^2\leq C \, \Vert (u_0,{\boldsymbol \sigma}_0)\Vert_2^2.
\end{equation}
%{\color{red} (me sale ese $k$ en el denominador, Â¿estoy equivocada?)} ({\color{blue} DA: No sale el $k$ en el denominador. Observe que los tÃ©rminos que quedan a derecha para controlar son: $ \langle (A-I) u_0, \delta_t u_1\rangle$, $(u_1{\boldsymbol\sigma}_0,\nabla \delta_t u_1)$, $\langle B  {\boldsymbol \sigma}_0, \delta_t {\boldsymbol \sigma}_1\rangle$ y $2(u_1\nabla u_0,\delta_t {\boldsymbol \sigma}_1)$, y en todos los casos controlamos $\Vert (\delta_t u_1,\delta_t {\boldsymbol \sigma}_1)\Vert_0^2$ con los tÃ©rminos del LHS. AsÃ­, que no aparece el $k$ en el denominador.})
Therefore, estimate  (\ref{stdta}) is deduced from (\ref{masfuert5}) and using (\ref{ssa-n}) and (\ref{edt2}). 
%Moreover, multiplying (\ref{masfuert5}) by $k$, adding from $m=n_0 +1$ to $m=n_0+n$, and using  \eqref{ssa-n}  then (\ref{stdta}) holds.  
Finally, from (\ref{weak01}), (\ref{H2est1}), (\ref{strong01}) and (\ref{stdta}), it holds (\ref{msus01}).
\end{proof}

 Corollaries \ref{CorvL1} and \ref{Cor:RegEst-n} and the equality ${\boldsymbol \sigma}_n=\nabla v_n$ yield to  
 %a more regular estimate for $v_n$ can be deduced.
\begin{cor}\label{CormorestV} {\bf(More regular estimates for $v_n$)}
Let $v_n$ be the solution of (\ref{edovf}). Under hypothesis of Corollary \ref{Cor:RegEst-n}, the following estimate holds
\begin{equation*}
\Vert v_n\Vert_{3}^{2}\leq C_4,  \ \ \ \forall n\geq 0.
\end{equation*}
%where $K_0>0$ is a constant depending on $(\Omega,u_0, {\boldsymbol \sigma}_0,v_0)$, but independent of $k$ and $n$.
\end{cor}

\subsubsection{Proof of (\ref{strong01}) in 2D domains}
In this section, the validity of the estimate (\ref{strong01}) in $2D$ domains will be proved. With this aim, an auxiliary result will be first considered.
\begin{prop}
Let $(u_n,{\boldsymbol\sigma}_n)$ be any solution of the scheme \textbf{US} defined in $2D$ domains. Then, the following estimate holds
\begin{equation}\label{GD1}
\Vert (u_n,{\boldsymbol \sigma}_n)\Vert^2_{1}\leq C \Vert (u_{n-1},{\boldsymbol \sigma}_{n-1})\Vert^2_{1}.
\end{equation}
%where $K_1$ is a constant depending on data $(\Omega,u_0, {\boldsymbol \sigma}_0)$, but independent of $k$ and $n$.
\end{prop}
\begin{proof}
Taking $\overline{u}=u_n - u_{n-1}$ and $\overline{\boldsymbol \sigma}=\displaystyle\frac{1}{2}({\boldsymbol \sigma}_{n}-{\boldsymbol \sigma}_{n-1})$ in (\ref{modelf02}), and 
recalling that using (\ref{consu})  the equality 
$\Vert  u_n\Vert_{1}^2 - \Vert  u_{n-1}\Vert_{1}^2=\Vert \nabla u_n\Vert_{0}^2 - \Vert \nabla u_{n-1}\Vert_{0}^2$ holds, one has
\begin{eqnarray}\label{fuerte1}
&\!\!\!\!\displaystyle
%\frac{1}{2k} & \!\!\!\!\!\!   \Vert (u_n - u_{n-1},\boldsymbol \sigma_n-\boldsymbol \sigma_{n-1})\Vert_{0}^2 +  
\frac{1}{4} & \!\!\!\!\!\! \Vert ( u_n,{\boldsymbol \sigma}_n)\Vert_{1}^{2}-\frac{1}{2}\Vert
(  u_{n-1},{\boldsymbol \sigma}_{n-1})\Vert_{1}^{2} 
 + \displaystyle\frac{1}{4} \, \Vert (u_n-u_{n-1},\sigma_n-\sigma_{n-1})\Vert_1^2
 \nonumber\\
&& \le  \vert (u_n\nabla u_n,{\boldsymbol \sigma}_n-{\boldsymbol \sigma}_{n-1})
-(u_n{\boldsymbol \sigma}_n,\nabla (u_n-u_{n-1}))\vert
=| (u_n, {\boldsymbol \sigma}_n\cdot\nabla u_{n-1} - \nabla u_n\cdot {\boldsymbol \sigma}_{n-1})|
\nonumber
\\ &&
 =  
 | (u_n, ({\boldsymbol \sigma}_n-{\boldsymbol \sigma}_{n-1})\cdot\nabla u_{n-1} - (\nabla u_n-\nabla u_{n-1})\cdot {\boldsymbol \sigma}_{n-1})|.
% \\ &&
%\vert(u_n\nabla u_{n-1},{\boldsymbol \sigma}_n-{\boldsymbol \sigma}_{n-1}) -(u_n{\boldsymbol \sigma}_{n-1},\nabla (u_n-u_{n-1}))\vert.  \nonumber
\end{eqnarray}
Then, by using the H\"older and Young inequalities as well as the $2D$ interpolation inequality  (\ref{in2D}) and estimate (\ref{weak01}), we find
\begin{eqnarray}\label{fuerte3}
&&
 | (u_n, ({\boldsymbol \sigma}_n-{\boldsymbol \sigma}_{n-1})\cdot\nabla u_{n-1} - (\nabla u_n-\nabla u_{n-1})\cdot {\boldsymbol \sigma}_{n-1})|
% \vert (u_n\nabla u_{n-1}&\!\!\!\!\!,{\boldsymbol \sigma}_n-{\boldsymbol \sigma}_{n-1}) -(u_n{\boldsymbol \sigma}_{n-1},\nabla (u_n-u_{n-1}))\vert
\nonumber\\
&& \leq \Vert u_n\Vert_{L^4}\Big( \Vert \nabla u_{n-1}\Vert_{0}  \Vert {\boldsymbol \sigma}_n-{\boldsymbol \sigma}_{n-1}\Vert_{L^4} + \Vert \nabla (u_n-u_{n-1})\Vert_{0}  \Vert {\boldsymbol \sigma}_{n-1}\Vert_{L^4} \Big)\nonumber\\
&&\leq\displaystyle\frac{1}{8}\Vert (u_n-u_{n-1},  {\boldsymbol \sigma}_n-{\boldsymbol \sigma}_{n-1})\Vert_{1}^2+\frac{1}{8}\Vert u_n\Vert_{1}^{2}+C\Vert (u_{n-1},{\boldsymbol \sigma}_{n-1})\Vert_{1}^{2}.
\end{eqnarray}
Therefore, from (\ref{fuerte1})-(\ref{fuerte3}), it follows that
\begin{equation*}\label{fuerte5}
%\displaystyle \frac{k}{2} \Vert (\delta_t u_n,\delta_t \boldsymbol \sigma_n) \Vert_{0}^2 +
\frac{1}{8}\Vert
 (u_n,{\boldsymbol \sigma}_n)\Vert_{1}^{2} 
 %+ \frac{1}{8}\Vert (u_n-u_{n-1},{\boldsymbol \sigma}_n-{\boldsymbol \sigma}_{n-1})\Vert^2_{1} 
 \leq \left(\frac{1}{2}+C \right)\Vert (u_{n-1}, {\boldsymbol \sigma}_{n-1})\Vert_{1}^{2},
\end{equation*}
hence (\ref{GD1}) is deduced.
\end{proof}

\begin{tma} \label{stlemne} 
Let $(u_n,{\boldsymbol
\sigma}_n)$ be a solution of the scheme \textbf{US} defined in $2D$ domains. Then, the estimate (\ref{strong01}) holds.
\end{tma}
\begin{proof}
From \eqref{StrongIneq-n} and (\ref{GD1}), one can deduce
\begin{eqnarray*}
&& \displaystyle \delta_t \Vert (\hat u_n,{\boldsymbol \sigma}_n)\Vert_{1}^2
 % +k \Vert (\delta_t u_n,\delta_t {\boldsymbol \sigma}_n)\Vert_{1}^2 
 + \Vert (\widehat A \hat u_n, B{\boldsymbol \sigma}_n)\Vert_0^2 
 \leq 
  C\Vert (\hat u_{n-1},{\boldsymbol\sigma}_{n-1})\Vert^4_1
 +C\Vert (\hat u_{n-1},{\boldsymbol\sigma}_{n-1})\Vert^2_1 .
\end{eqnarray*}
Then, taking into account (\ref{weak01}),  
 Corollary \ref{CorUnif} can be applied  and  (\ref{strong01}) is deduced. 
\end{proof}

\subsection{Error estimates}% at finite time}
We will obtain error estimates for the scheme \textbf{US} with respect to a su\-ffi\-cien\-tly regular solution $(u,{\boldsymbol\sigma})$ of (\ref{modelf01}) and  $v$ of (\ref{modelf01eqv}). For any final time $T>0$, let us consider a fixed partition of $[0,T]$ given by $(t_n=nk)_{n=0}^N$, where $k=T/N>0$ is the time step. 
We will denote by   $C,C_i,K_i$ to different positive constants possibly  depending on the continuous solution $(u, v,{\boldsymbol \sigma}=\nabla v)$, but independent of time size $k$ and the length of the time interval $T$, giving the dependence of $T$  explicitly.
 We introduce the following notations for the errors in $t=t_{n}$:  
$$
e_u^n=u(t_n)-u_n,\quad e_{\boldsymbol\sigma}^n={\boldsymbol\sigma}(t_n)-{\boldsymbol\sigma}_n
\quad \hbox{and} \quad 
e_v^n=v(t_n)-v_n ,
$$
and for 
the discrete norms:
$$
\Vert  (e^n)\Vert^2_{l^\infty X} := \max_{n=1,\cdots,N} \| e^n\|^2_X,
\quad 
\Vert  (e^n)\Vert^2_{l^2 X} :=  k \, \sum_{n=1}^N 
 \| e^n\|^2_X.
$$
% and for the discrete in time derivative of the errors: $\displaystyle \delta_t e_u^n=\frac{e_u^n-e_u^{n-1}}{k}$ and $\displaystyle\delta_t e_{\boldsymbol\sigma}^n=\frac{e_{\boldsymbol\sigma}^n - e_{\boldsymbol\sigma}^{n-1}}{k}$. 

\subsubsection{Error estimates in weak norms for $(e_u^n,e_{{\boldsymbol \sigma}}^n)$}

Subtracting (\ref{modelf01}) at $t=t_n$ and the scheme \textbf{US}, then $(e_u^n,e_{{\boldsymbol \sigma}}^n)$ satisfies
\begin{equation}\label{erru}
\left(\delta_t e_u^n,\overline{u}\right) +  (\nabla e_u^n, \nabla\overline{u}) +(e_u^n{\boldsymbol \sigma}(t_n)+u_n e_{\boldsymbol \sigma}^n,\nabla \overline{u})=(\xi_1^n,\overline{u}), \ \ \forall \overline{u}\in H^1(\Omega),
\end{equation}
\begin{equation}\label{errs}
\left(\delta_t e_{\boldsymbol\sigma}^n,\overline{\boldsymbol \sigma}\right)+ \langle B
e_{\boldsymbol\sigma}^n, \overline{\boldsymbol \sigma}\rangle = 2(e_u^n\nabla u(t_n)+ u_n \nabla e_u^n ,\overline{\boldsymbol \sigma}) + (\xi_2^n,\overline{\boldsymbol \sigma}),\ \ \forall \overline{\boldsymbol \sigma}\in \H^1_{\sigma}(\Omega),
\end{equation}
where $\xi_1^n,\xi_2^n$ are the consistency errors: 
\begin{equation*}
(\xi_1^n,\xi_2^n)=\delta_t (u(t_n),{\boldsymbol\sigma}(t_n)) - (u_t(t_n),{\boldsymbol\sigma}_t(t_n))=-\displaystyle\frac{1}{k}\int_{t_{n-1}}^{t_n}(t-t_{n-1})(u_{tt}(t),{\boldsymbol \sigma}_{tt}(t)) dt.
\end{equation*}
\begin{tma}\label{erteo}
Let $(u_n,{\boldsymbol\sigma}_n)$ be a solution of the scheme \textbf{US} and assume the following regularity for the exact solution $(u,{\boldsymbol\sigma})$ of (\ref{modelf01}):
\begin{equation}\label{regU}
(u,{\boldsymbol\sigma}) \in L^{\infty}(0,+\infty;H^{1}(\Omega)\times \H^{1}(\Omega))  \ \ \mbox{and} \ \ (u_{tt},{\boldsymbol\sigma}_{tt})\in L^2(0,+\infty;H^1(\Omega)' \times \H^1_{\sigma}(\Omega)').
\end{equation}
Assuming that 
\begin{equation}\label{re01}
\displaystyle k\, \Vert (\nabla u, \nabla \cdot {\boldsymbol\sigma})
\Vert_{L^\infty(L^2)}^4 \quad \hbox{is small enough,}
\end{equation}
then the following error estimate holds
\begin{eqnarray}\label{priorierr}
\Vert  (e_{u}^n,e_{{\boldsymbol \sigma}}^n)\Vert^2_{l^\infty L^2\cap l^2 H^1}\leq 
K_1 T \exp(K_2T)\, k^2 .
\end{eqnarray}
 %where $C(T) = K_1 \exp(K_2T)$, with $K_1,K_2>0$ independent of $k$.
\end{tma}
\begin{proof}
Since $u_0=u(t_0)$, then $\int_\Omega e_u^{n}=\int_\Omega e_u^{0}=0$. Moreover, taking $\overline{u}=e_{u}^n$ in (\ref{erru}), $\overline{\boldsymbol
\sigma}=\displaystyle\frac{1}{2}e_{{\boldsymbol\sigma}}^n$ in (\ref{errs}), and adding the resulting expressions, the terms $(u_n e_{\boldsymbol \sigma}^n,\nabla e_{u}^n)$ cancel, and using the H\"older and Young inequalities and (\ref{in3D}), one obtains
\begin{eqnarray}\label{ester4}
&\delta_t&\!\!\!\!\!\left( \displaystyle \frac{1}{2} \Vert e_{u}^n\Vert_{0}^2 + \displaystyle \frac{1}{4} \Vert e_{{\boldsymbol \sigma}}^n\Vert_{0}^2 \right) 
%+\frac{k}{4} \, \Vert (\delta_t e_u^n, \delta_t e_{{\boldsymbol \sigma}}^n ) \Vert_0^2
+\frac{1}{2}\Vert  (e_{u}^n, e_{{\boldsymbol \sigma}}^n)\Vert_{1}^2  \nonumber\\
&&\!\!\!\!\!\!\!\!\!\!\!\!\!\!\!
\leq \displaystyle\frac{1}{4} \Vert (e_u^n,e_{{\boldsymbol \sigma}}^n)\Vert_{1}^2 
+C\Vert (\xi_1^n,\xi_2^n)\Vert^2_{(H^1)'\times(H_{\sigma}^1)'}
%+\frac{1}{8}\Vert (e^n_u,e_{{\boldsymbol \sigma}}^n)\Vert_{1}^2
+C\Vert (\nabla u(t_n),\nabla\cdot {\boldsymbol\sigma}(t_n))\Vert_{0}^4\Vert e_{u}^n\Vert_{0}^2 .
\end{eqnarray}
%In particular, 
%\begin{eqnarray}\label{ester4}
%&\delta_t \left( \displaystyle \frac{1}{2} \Vert e_{u}^n\Vert_{0}^2 + \displaystyle \frac{1}{4} \Vert e_{{\boldsymbol \sigma}}^n\Vert_{0}^2\right)&\!\!\!\!\!+\frac{1}{4}\Vert (e_{u}^n, e_{{\boldsymbol \sigma}}^n)\Vert_{1}^2\nonumber\\
%&&\!\!\!\!\! \leq C\Vert (\xi_1^n,\xi_2^n)\Vert^2_{(H^1)'\times(H_{\sigma}^1)'} + C\Vert (\nabla u(t_n),\nabla\cdot {\boldsymbol\sigma}(t_n))\Vert_{0}^4\Vert e_{u}^n\Vert_{0}^2.
%\end{eqnarray}
Using the H\"older inequality, one can bound
\begin{equation}\label{error1}
C\Vert (\xi_1^n,\xi_2^n)\Vert_{(H^1)'\times(H_{\sigma}^1)'}^2\leq C k \int_{t_{n-1}}^{t_n}\Vert (u_{tt}(t),{\boldsymbol \sigma}_{tt}(t))\Vert_{(H^1)'\times(H_{\sigma}^1)'}^2dt.
\end{equation}
Therefore, from (\ref{ester4}) and (\ref{error1}),   
\begin{eqnarray}\label{error3}
&\delta_t&\!\!\!\!\!\left( \displaystyle \frac{1}{2} \Vert e_{u}^n\Vert_{0}^2 +\frac{1}{4} \Vert e_{{\boldsymbol \sigma}}^n\Vert_{0}^2  \right)+\frac{1}{4}\Vert (e_{u}^n, e_{{\boldsymbol \sigma}}^n)\Vert_{1}^2 \nonumber\\
&&\!\!\!\!\!\leq C k \int_{t_{n-1}}^{t_n}\Vert (u_{tt}(t),{\boldsymbol \sigma}_{tt}(t))\Vert_{(H^1)'\times(H_{\sigma}^1)'}^2 dt +C\Vert (\nabla u,\nabla\cdot{\boldsymbol\sigma})\Vert_{L^\infty(L^2)}^4\Vert e_{u}^n\Vert_{0}^2.
\end{eqnarray}
 Then, multiplying (\ref{error3}) by $k$ and adding from $n=1$ to $n=r$ (recall that   $e_{u}^0=e_{{\boldsymbol \sigma}}^0=0$):
\begin{eqnarray}\label{eserr01}
&&\displaystyle \left[\frac{1}{4}- Ck\Vert (\nabla u, \nabla\cdot{\boldsymbol\sigma})\Vert_{L^\infty L^2}^4\right]\Vert (e_{u}^r,e_{{\boldsymbol \sigma}}^r)\Vert_{0}^2+\frac{k}{4} \underset{n=1}{\overset{r}{\sum}}\Vert (e_{u}^n, e_{{\boldsymbol \sigma}}^n)\Vert_{1}^2\nonumber\\
&&\!\!\!\!\! \leq C k^2 \Vert (u_{tt},{\boldsymbol \sigma}_{tt})\Vert_{L^2((H^1)'\times(H_{\sigma}^1)')}^2 dt +C\Vert (\nabla u, \nabla\cdot{\boldsymbol\sigma})\Vert_{L^\infty L^2}^4k \underset{n=0}{\overset{r-1}{\sum}} \Vert (e_{u}^n,e_{{\boldsymbol \sigma}}^n)\Vert_{0}^2.
\end{eqnarray}
Therefore, by applying hypothesis (\ref{re01}) and taking into account (\ref{regU}), we can use the   discrete Gronwall Lemma (see \cite{Heywood}, p.~369)  in (\ref{eserr01})  concluding (\ref{priorierr}). %{\color{red} (MA: ¿Escribir en los previos?)}
\end{proof}

\begin{obs}
From (\ref{priorierr}), one deduces that $\Vert (u_n,{\boldsymbol\sigma}_n)\Vert_1^2 \leq  K_1 T\exp(K_2T)$, for all $n=1,...,N$.
%, where the constant $C(T)$ depends on $u,u_{tt},{\boldsymbol\sigma},{\boldsymbol\sigma}_{tt},T$. 
In particular, this fact implies that in 3D domains, for finite time, the hypothesis (\ref{uniq01}) assuring the uniqueness of solution $(u_n,{\boldsymbol\sigma}_n)$ can be relaxed to $ k\,(K_1 T\exp(K_2T))^2$ small enough.
\end{obs}

\subsubsection{Error estimates in strong norms for $v_n$}

Subtracting (\ref{modelf01eqv}) at $t=t_n$ and (\ref{edovf}), then $e_v^n$ satisfies
\begin{equation} \label{errv}
\ \ \delta_t e_v^n + A
e_v^n = (u(t_n)+u_n)e_u^n+ \xi_3^n\ \ \mbox{ a.e.~$\x \in \Omega$},
\end{equation}
where $\xi_3^n$ is the consistency error associated to (\ref{edovf}), that is, 
\begin{equation}\label{xi3}
\xi_3^n=\delta_t (v(t_n)) - v_t(t_n)=-\displaystyle\frac{1}{k}\int_{t_{n-1}}^{t_n} (t-t_{n-1})v_{tt}(t) dt.
\end{equation}

\begin{tma}\label{erteoV}
Under hypotheses of Theorem \ref{erteo}. Let $v_n$ be the solution of (\ref{edovf}) and assume the following regularity for the exact solution $v$ of (\ref{modelf01eqv}):
\begin{equation}\label{regV}
v_{tt}\in L^2(0,+\infty;H^1(\Omega)').
\end{equation}
Then the following error estimate holds
\begin{equation}\label{priorierr001}
\Vert  e_{v}^n\Vert^2_{l^\infty H^1\cap l^2 H^2}\leq K_3 T \exp(K_4T)\,  k^2 .
\end{equation}
% where $C(T) = K_1 \exp(K_2T)$, with $K_1,K_2>0$ independent of $k$.
\end{tma}
\begin{proof}
Since $e_{\boldsymbol\sigma}^n=\nabla e_{v}^n$, taking into account (\ref{priorierr}), it suffices to prove 
\begin{equation}\label{nler0a}
\left(  \int_\Omega e^n_v \right)^2 \leq  K_3 T \exp(K_4T) \, k^2,\quad \forall\, n\ge 0.
\end{equation}
With this aim,  integrating (\ref{errv}) in $\Omega$,
\begin{equation} \label{nler0b}
\delta_t \left( \int_\Omega e_v^n\right)  + \int_\Omega 
e_v^n = \int_\Omega (u(t_n)+u_n)e_u^n+ \int_\Omega \xi_3^n .
\end{equation}
Multiplying (\ref{nler0b}) by $k \displaystyle\left(\int_\Omega e^n_v\right)$ and using (\ref{xi3}), one has
\begin{equation*} \label{nler0c}
(1+ \frac{k}{2}) \displaystyle\left( \int_\Omega e_v^n\right)^2  -\left( \int_\Omega 
e_v^{n-1} \right) ^2
%&\!\!\! \leq C\, k \left( \int_\Omega (u(t_n)+u_n)e_u^n \right)^2 + \left( \int_{t_{n-1}}^{t_n} \int_\Omega (t-t_{n-1})v_{tt}(\x ,t) d\x  dt\right)^2\nonumber\\
 \leq C\, k \,\Vert u(t_n)+u_n\Vert_0^2 \Vert e_u^n\Vert_0^2 
 + C\, k^2 \,\vert\Omega\vert \int_{t_{n-1}}^{t_n} \Vert v_{tt}(t) \Vert^2_{(H^1)'}.
\end{equation*}
Then, adding from $n=1$ to $n=r$ and taking into account that $u(t_n)+u_n$ is bounded in $l^\infty L^2$, we obtain (recall that $e_{v}^0=0$)
\begin{equation}\label{nler0d}
 \displaystyle\left( \int_\Omega e_v^r\right)^2+k \underset{n=1}{\overset{r}{\sum}}\displaystyle\left( \int_\Omega e_v^n\right)^2 \leq C \, k^2 \Vert v_{tt}(t)\Vert_{L^2(H^1)'} ^2+C\, k \underset{n=1}{\overset{r}{\sum}} \Vert e_{u}^n\Vert_{0}^2.
\end{equation}
Thus  (\ref{nler0a}) is deduced, using (\ref{regV}) and (\ref{priorierr})  in (\ref{nler0d}).
\end{proof}

\section{A linear scheme}\label{LINES}
In this section, we propose the following first-order in time, linear coupled scheme for  (\ref{modelf01}): 
\begin{itemize} 
\item{\underline{\emph{Scheme \textbf{LC}}:}\\
{\bf Initialization:} 
We take $(u_0,v_0)=(u(0),v(0))$ and  ${\boldsymbol\sigma}_0=\nabla v_0$.
\\
{\bf Time step} n: Given $(u_{n-1},{\boldsymbol \sigma}_{n-1})\in  H^{1}(\Omega)\times \H^{1}_{\sigma}(\Omega)$, compute $(u_{n},{\boldsymbol \sigma}_{n})\in  H^{1}(\Omega)\times \H^{1}_{\sigma}(\Omega)$ solving
\begin{equation}  \label{modelflin02}
\left\{
\begin{array}
[c]{lll}%
(\delta_t u_n,\overline{u}) + (\nabla u_n, \nabla\overline{u}) = -(u_{n-1}{\boldsymbol \sigma}_n,\nabla \overline{u}), \ \ \forall \overline{u}\in H^{1}(\Omega),\\
(\delta_t {\boldsymbol \sigma}_n,\overline{\boldsymbol \sigma}) +
\langle B {\boldsymbol \sigma}_n, \overline{\boldsymbol
\sigma}\rangle =
2(u_{n-1}\nabla u_n,\overline{\boldsymbol \sigma}),\ \ \forall
\overline{\boldsymbol \sigma}\in \H^{1}_{\sigma}(\Omega).
\end{array}
\right.  
\end{equation}}
\end{itemize}
Once solved  the scheme \textbf{LC}, given $v_{n-1}\in {H}^{2}(\Omega)$ with $v_{n-1}\geq 0$, 
then  $v_n=v_n(u_n^2)\in H^2(\Omega)$ can be recovered by solving (\ref{edovf}). 

\subsection{Conservation, unconditional well-posedness and energy-stability}
First of all, notice that scheme \textbf{LC}  is also $u$-conservative  (satisfying \eqref{consu}).
%and  has  the  behavior for $\int_\Omega v_n$ given in \eqref{consu-1}.
%\begin{equation*} \delta_t\left(\int_\Omega v_n\right)=\int_\Omega
%u_n^{2} -\int_\Omega v_n.
%\end{equation*}
\begin{tma} {\bf(Unconditional well-posedness)} \label{linsc1}
There exists a unique $(u_n,{\boldsymbol\sigma}_n)$ solution of the scheme \textbf{LC}.
\end{tma}
\begin{proof}
Let $(u_{n-1},{\boldsymbol
\sigma}_{n-1})\in  \mathbb{H} := H^1(\Omega)\times  \H^{1}_{\sigma}(\Omega)$ be given, and consider the following bilinear form ${a}:\mathbb{H}\times \mathbb{H}\rightarrow \mathbb{R}$, and the linear operator $l:\mathbb{H}\rightarrow
\mathbb{R}$ given by
\begin{eqnarray*}
&{a}((u_n,{\boldsymbol\sigma}_n),(\overline{u},\overline{\boldsymbol\sigma}))&\!\!\!=\displaystyle\frac{2}{k} (u_n,\overline{u})+\displaystyle\frac{1}{k} ({\boldsymbol\sigma}_n,\overline{{\boldsymbol\sigma}})+2(\nabla u_n, \nabla\overline{u}) +\langle B {\boldsymbol\sigma}_n,  \overline{\boldsymbol\sigma}\rangle \nonumber\\
&&+ 2(u_{n-1}{\boldsymbol\sigma}_n,\nabla \overline{u}) - 2(u_{n-1}\nabla u_n,\overline{\boldsymbol\sigma}), 
\end{eqnarray*}
$$l((\overline{u},\overline{\boldsymbol\sigma}))= \frac{2}{k}(
u_{n-1}, \overline{u}) + \frac{1}{k}({\boldsymbol\sigma}_{n-1}, \overline{{\boldsymbol\sigma}}),$$
for all $(u_n,{\boldsymbol\sigma}_n),(\overline{u},\overline{\boldsymbol\sigma})\in \mathbb{H} $. Then, using the H\"older inequality and Sobolev embeddings, we can verify that $a(\cdot,\cdot)$ is continuous and coercive on $\mathbb{H}$, and  $l\in  \mathbb{H}'$. Thus, from Lax-Milgram theorem, there exists a unique $(u_n,{\boldsymbol \sigma}_n)\in \mathbb{H}$ such that 
$$a((u_n,{\boldsymbol\sigma}_n),(\overline{u},\overline{\boldsymbol\sigma}))=l((\overline{u},\overline{\boldsymbol\sigma})), \ \ \forall (\overline{u},\overline{\boldsymbol\sigma})\in \mathbb{H}.$$
By taking alternatively $\overline{\boldsymbol\sigma}=0$ and  $\overline{u}=0$, one has that $(u_n,{\boldsymbol \sigma}_n)$ is the unique 
solution of (\ref{modelflin02}). 
\end{proof}

Moreover, following the proof of Lemma \ref{estinc1}, the unconditional energy-stability of  the scheme \textbf{LC} can be proved.
\begin{lem} {\bf(Unconditional energy-stability)} \label{LINEALes}
The scheme \textbf{LC} is unconditionally energy-stable with respect to $\mathcal{E}(u,{\boldsymbol \sigma})$. In fact,  the  discrete energy law \eqref{lawenerfydisce} holds.
%\begin{equation*}\label{lawenerfydiscelin}
%\delta_t \mathcal{E}(u_n,{\boldsymbol \sigma}_n)+ 
%\frac{k}{2} 
%\Vert \delta_t u_n\Vert_{0}^2  +\frac{k}{4} \Vert \delta_t {\boldsymbol \sigma}_n\Vert_{0}^2 + \Vert \nabla u_n\Vert_{0}^{2} +
%\displaystyle\frac{1}{2}\Vert  {\boldsymbol
%\sigma}_n\Vert_{1}^{2}=0.
%\end{equation*}
\end{lem}

\begin{obs}
By considering a fully discrete Finite Element scheme associated to \textbf{LC}, 
%via a Finite Element spatial appro\-xi\-mation,
% i.e.~taking $U_h\subset H^1(\Omega)$ and ${\boldsymbol\Sigma}_h\subset \H^1_{\sigma}(\Omega)$ instead of $H^1(\Omega)$ and $ \H^1_{\sigma}(\Omega)$ respectively, 
then its unconditional well-posedness and energy-stability 
%of this fully discrete scheme 
can be proved following line by line  Theorem~\ref{linsc1} and Lemma~\ref{LINEALes}, respectively.
\end{obs}

\begin{obs}
 Uniform weak estimates for the solution $(u_n,{\boldsymbol\sigma}_n)$ of the scheme \textbf{LC} can be proved analogously to Theorem \ref{welem}. Moreover, assuming the $H^2$-regularity of problem (\ref{B201}) (and even in the case that the right hand side is not the gradient of a function),   strong and more regular uniform estimates for $(u_n,{\boldsymbol\sigma}_n)$ can be deduced as in Subsections \ref{UES} and \ref{UES-bis}. 
\end{obs}

\begin{obs}
%Unlike the scheme \textbf{US}, in 
The positivity of $u_n$ in the scheme \textbf{LC} is not clear. In fact, in the numerical simulations of Subsection \ref{sim} below, very negative cell densities are obtained even when $h\rightarrow 0$, where $h>0$ is the spatial mesh size. 
\end{obs}

\subsection{Error estimates}
\begin{tma}\label{erteolin}{\bf (Error estimates in weak norms for $(u_n,{\boldsymbol\sigma}_n)$).}
Let $(u_n,{\boldsymbol\sigma}_n)$ be the solution of  the scheme \textbf{LC}, and assume the regularity (\ref{regU}). Then, the a priori error estimate (\ref{priorierr}) holds.  
\end{tma}
\begin{proof}
The proof follows the same steps of Theorem \ref{erteo}. 
But, in this case, the hypothesis of small time step (\ref{re01}) used in order to apply the Discrete Gronwall Lemma is not needed, since  the ``semi-explicit'' form of the terms $(u_{n-1}{\boldsymbol \sigma}_n,\nabla \overline{u})$ and $(u_{n-1}\nabla u_n,\overline{\boldsymbol \sigma})$ allow to bound them in a suitable way. 
\end{proof}

Moreover, although in this linear scheme \textbf{LC}  the relation ${\boldsymbol\sigma}_n=\nabla v_n$ is not clear, it will be possible to obtain error estimates for $v_n$.
\begin{tma}\label{erteolinV}{\bf (Error estimates in strong norms for $v_n$).}
Under hypothesis of Theorem \ref{erteolin}. Let $v_n$ be the solution of (\ref{edovf}) (corresponding to the scheme \textbf{LC}), and assume the regularity: 
\begin{equation}\label{regVlin}
v_{tt}\in L^2(0,+\infty;L^2(\Omega)).
\end{equation}
Then, the a priori error estimate 
%(\ref{priorierr001}) 
holds:
\begin{equation} \label{priorierr001vis}
\Vert  e_{v}^n\Vert^2_{l^\infty H^1\cap l^2 H^2}\leq  C_1(T)\, k^2 + C_2(T)^{3/2} k^{5/2}  
\end{equation}
where $C_i(T)=K_1 T \exp(K_2T)$.
\end{tma}
\begin{proof}
Since in the scheme \textbf{LC}  the relation ${\boldsymbol\sigma}_n=\nabla v_n$ is not clear, we will argue in a different way of Theorem \ref{erteoV}. 
Indeed,  testing (\ref{errv}) by $A e^n_v$, 
and using the H\"older and Young inequalities,  one obtains
\begin{equation}\label{newer002}
\displaystyle\frac{1}{2} \delta_t \left( \Vert e^n_v \Vert_1^2\right) 
%+ \frac{k}{2} \Vert \delta_t e^n_v\Vert_1^2 
+ \frac{1}{2}\Vert A e^n_v\Vert_0^2 \leq  C \,\Vert u(t_n) + u_n\Vert_{L^3}^2 \Vert e^n_u\Vert_1^2  + C\,\Vert \xi_3^n \Vert_0^2. 
\end{equation}
Observe that from (\ref{priorierr})  it follows that 
$\underset{n=1}{\overset{r}{\sum}} \Vert u(t_n) - u_n\Vert_1^2\leq C(T) k$ with $C(T)=K_1 T \exp(K_2 T))$, 
which implies using (\ref{regU}) that 
\begin{equation}\label{new00av}
\Vert u_n\Vert_1^2\leq C(T)\, k + \Vert u(t_n)\Vert_1^2 
\le C(T)\, k  + C.
\end{equation}
In particular, using \eqref{weak01}  and (\ref{ssa-n}), 
$$
\Vert u_n\Vert_{L^3}^2
\le C\, \Vert u_n\Vert_{0} \Vert u_n\Vert_{1}
\le C(T)^{1/2} k^{1/2}+C.
$$
 Then, multiplying (\ref{newer002}) by $k$, adding from $n=1$ to $n=r$ and using (\ref{new00av}) (recall that $e_{v}^0=0$),
\begin{eqnarray}\label{newer003}
\Vert e_{v}^r\Vert_{1}^2+k \underset{n=1}{\overset{r}{\sum}}\Vert A e^n_v \Vert_{0}^2
\leq \left(  C(T)^{1/2} k^{1/2} +C\right)k \underset{n=1}{\overset{r}{\sum}} \Vert e^n_u\Vert_1^2 + C k^2 \Vert v_{tt}\Vert_{L^2(L^2)}^2.
\end{eqnarray}
Therefore, using  (\ref{regVlin}) and (\ref{priorierr}) in (\ref{newer003}), 
 the a priori error estimate (\ref{priorierr001vis}) is deduced. 
\end{proof}

\section{Numerical simulations}
In this section some numerical si\-mu\-lations  of the schemes described in this paper will be compared.  One considers a Galerkin finite element discretization in space associated to the variational formulation of schemes \textbf{US}, \textbf{LC} and \textbf{UV}, where the $\mathbb{P}_1$-continuous approximation is taken  for $u_h$, ${\boldsymbol\sigma}_h$ and $v_h$ (where $h$ is the spatial parameter). 
In fact, the domain $\Omega=(0,2)^2$ has been discretized by using a structured mesh, and all the simulations  have been carried out using $\textbf{FreeFem++}$ software. The nonlinear schemes \textbf{US} and \textbf{UV} are approached by Newton's Method,  stopping when the relative error in $L^2$-norm is less than $tol=10^{-6}$.

\subsection{Positivity}\label{sim}
In this subsection, the schemes $\textbf{US}$ and $\textbf{LC}$ are compared in terms of positivity. 
For the fully discretization of  both schemes the positivity of the variable $u_h$ is not clear. In fact, for the  time-discrete scheme $\textbf{US}$ the existence of nonnegative solution $(u_n,v_n)$ was proved (see Theorem \ref{USus} and Remark \ref{pvr}), but for the time-discrete scheme $\textbf{LC}$, although the positivity of $v_n$ can be proved, the positivity of $u_n$ is not clear. 
For this reason, in Figure~\ref{fig:MU70},  the positivity of the variables $u_h$ and $v_h$ is compared in both schemes taking meshes  increasingly thinner ($h=\frac{1}{35}$, $h=\frac{1}{75}$ and $h=\frac{1}{150}$). In all cases, we choose $k=10^{-5}$ and the initial conditions  (see Figure \ref{fig:initcond}):
$$u_0=-10\,xy\,(2-x)(2-y)\exp(-10(y-1)^2-10(x-1)^2)+10.0001,$$ 
$$v_0=200\,xy\,(2-x)(2-y)\exp(-30(y-1)^2-30(x-1)^2)+0.0001.$$
\begin{figure}[htbp]
\centering 
\subfigure[Initial cell density $u_0$]{\includegraphics[width=77mm]{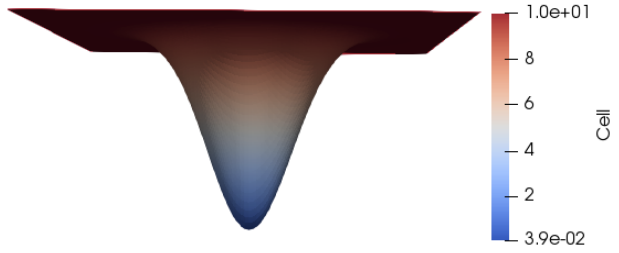}} \hspace{1 cm}
\subfigure[Initial chemical concentration $v_0$]{\includegraphics[width=70mm]{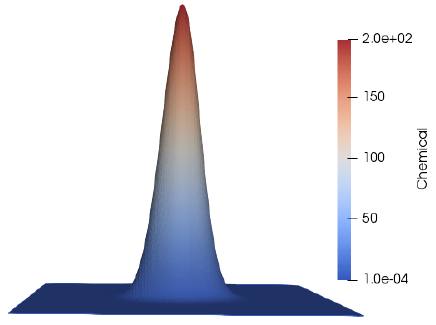}} 
\caption{Initial conditions.} \label{fig:initcond}
\end{figure}
In the case of the scheme $\textbf{US}$, it can be observed that $u_h$ is ne\-ga\-tive for some $\x \in \Omega$ in some times $t_n>0$, but when $h\rightarrow 0$ these values are closer to $0$; while in the case of  the scheme $\textbf{LC}$, when $h\rightarrow 0$ very negative cell densities $u_h$ are obtained for some $\x \in \Omega$ in some times $t_n>0$ (see Figure~\ref{fig:MU70}(a)-(c)). On the other hand, the same behavior is observed for the minimum of $v_h$ in both schemes. In fact, independently of $h$, if $v_0$ is positive, then $v_h$  is also positive (we show this behavior in Figure~\ref{fig:MU70}(d) for the case $h=\frac{1}{35}$, but the same holds for the cases $h=\frac{1}{75},\frac{1}{150}$).

\subsection{Unconditional Stability}
In this subsection,  the stability with respect to the energies   $\mathcal{E}(u,v)$ and  $\mathcal{E}(u,{\boldsymbol \sigma})$,  given in (\ref{eneruva}) and (\ref{ener01}) respectively, are numerically compared.
 Following line by line the proof of Lemma \ref{estinc1},  the unconditional energy-stability with respect to   $\mathcal{E}(u,{\boldsymbol \sigma})$ for the fully discrete schemes co\-rres\-ponding to schemes \textbf{US} and \textbf{LC} can be deduced. 
 In fact, if $(u_n,{\boldsymbol \sigma}_n)$ is any solution of the fully discrete schemes associated to 
\textbf{US} or \textbf{LC}, the following relation holds
\begin{equation}\label{ns01a}
RE(u_n,{\boldsymbol\sigma}_n):=\delta_t \mathcal{E}(u_n,{\boldsymbol \sigma}_n)+ \Vert \nabla u_n\Vert_{0}^{2}
+
\displaystyle\frac{1}{2}\Vert {\boldsymbol
\sigma}_n\Vert_{1}^{2} \leq 0, \ \ \forall n .
\end{equation}
However, considering the ``exact'' energy $\mathcal{E}(u,v)$ given in (\ref{eneruva}), in the case of fully discrete schemes, it is not clear how to prove unconditional energy-stability of schemes \textbf{US}, \textbf{LC} and \textbf{UV} with respect to this energy.
Therefore, it is interesting to study the behaviour of the corresponding residual  
 \begin{equation*}
RE(u_n,v_n):=\delta_t \mathcal{E}(u_n,v_n)+ \Vert \nabla u_n\Vert_{0}^{2}
+
\displaystyle\frac{1}{2}\Vert \Delta_h v_n\Vert_{0}^{2}+\frac{1}{2}\Vert \nabla v_n\Vert_{0}^{2}.
\end{equation*}
\begin{figure}[htbp]
	\centering 
	\subfigure[Minimum values of $u_h$, with $h=\frac{1}{35}$]{\includegraphics[width=77mm]{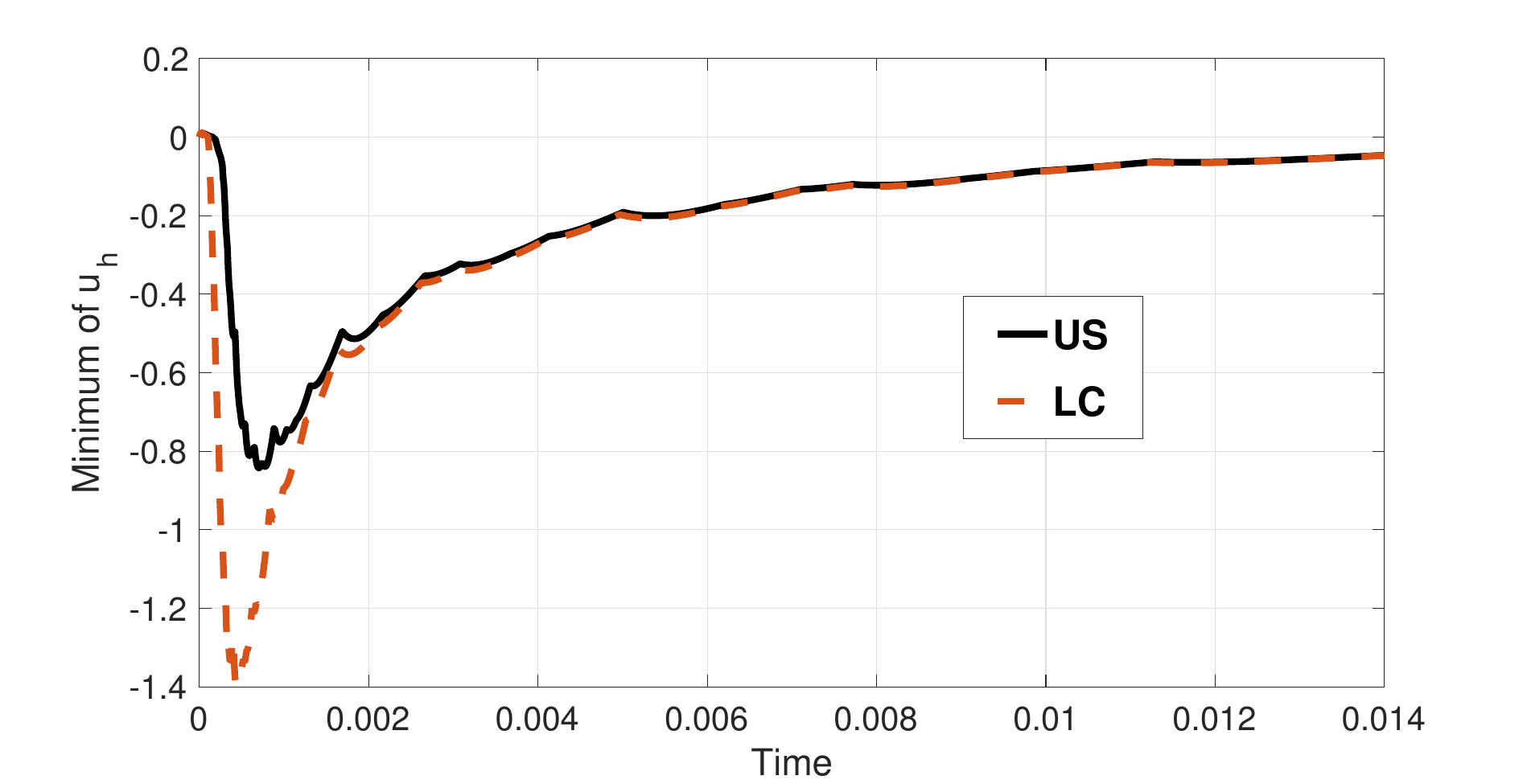}} \hspace{0.2 cm} 
	\subfigure[Minimum values of $u_h$, with $h=\frac{1}{75}$]{\includegraphics[width=77mm]{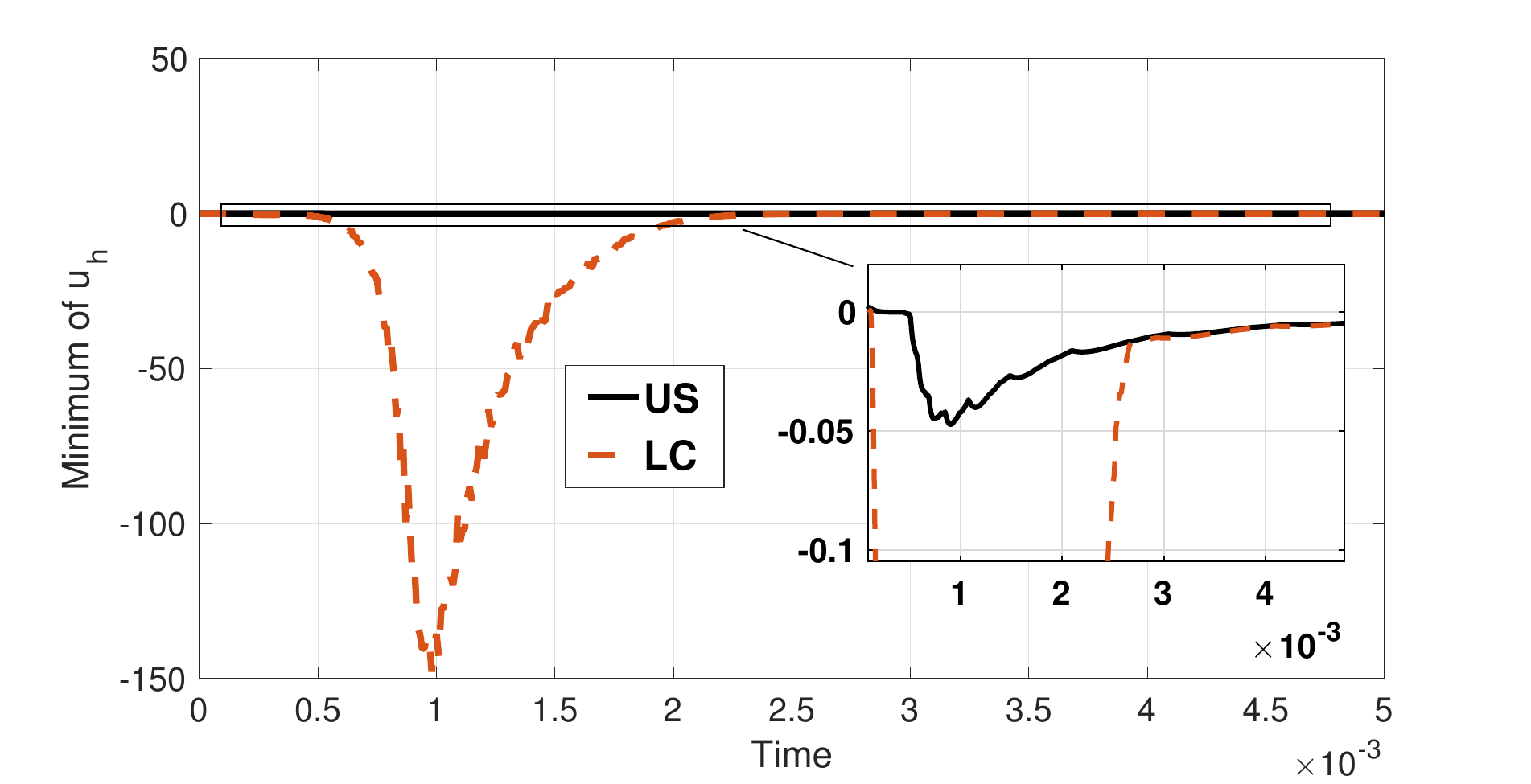}}
	\subfigure[Minimum values of $u_h$, with $h=\frac{1}{150}$]{\includegraphics[width=77mm]{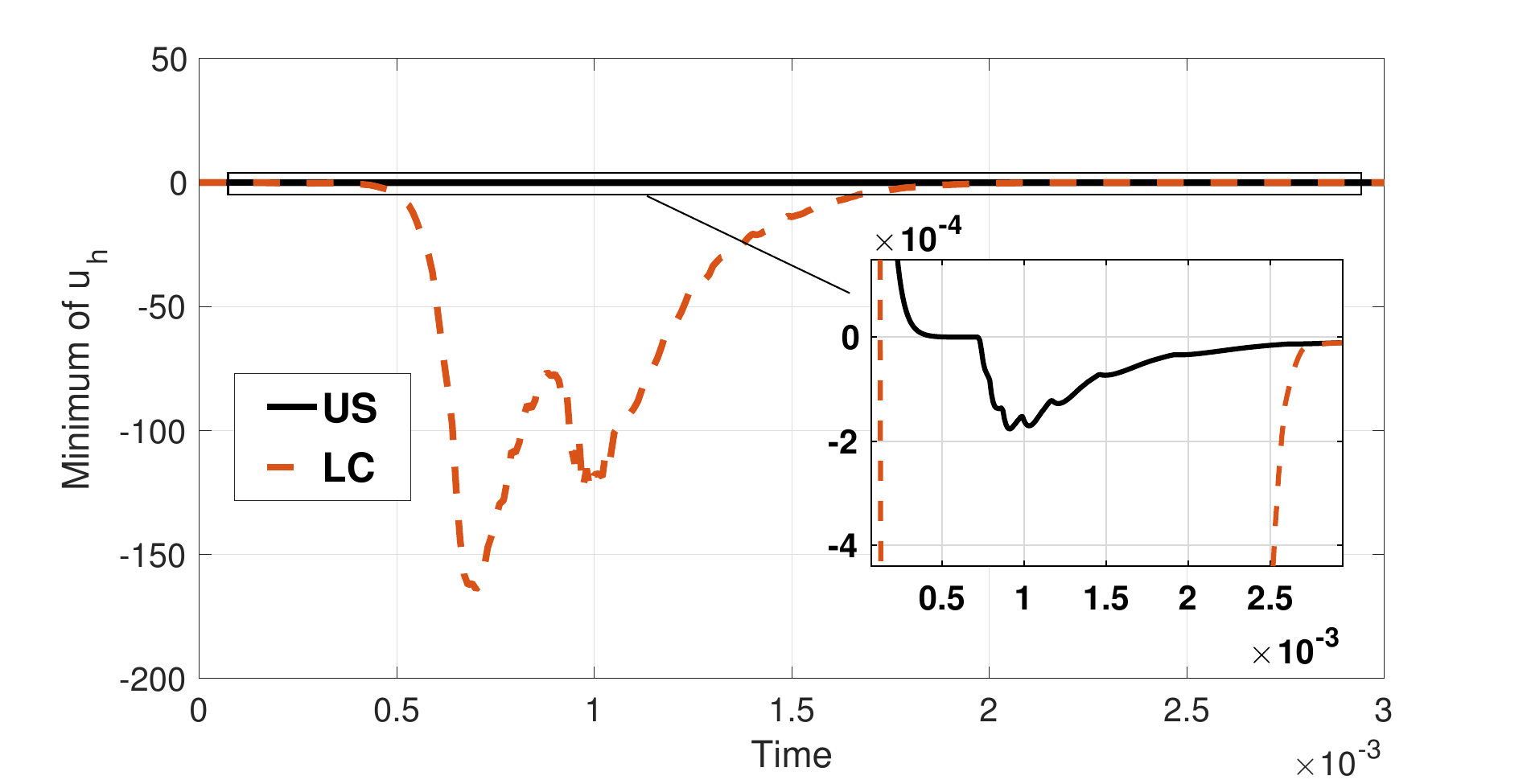}}\hspace{0.2 cm} 
	\subfigure[Minimum values of $v_h$, with $h=\frac{1}{35}$]{\includegraphics[width=77mm]{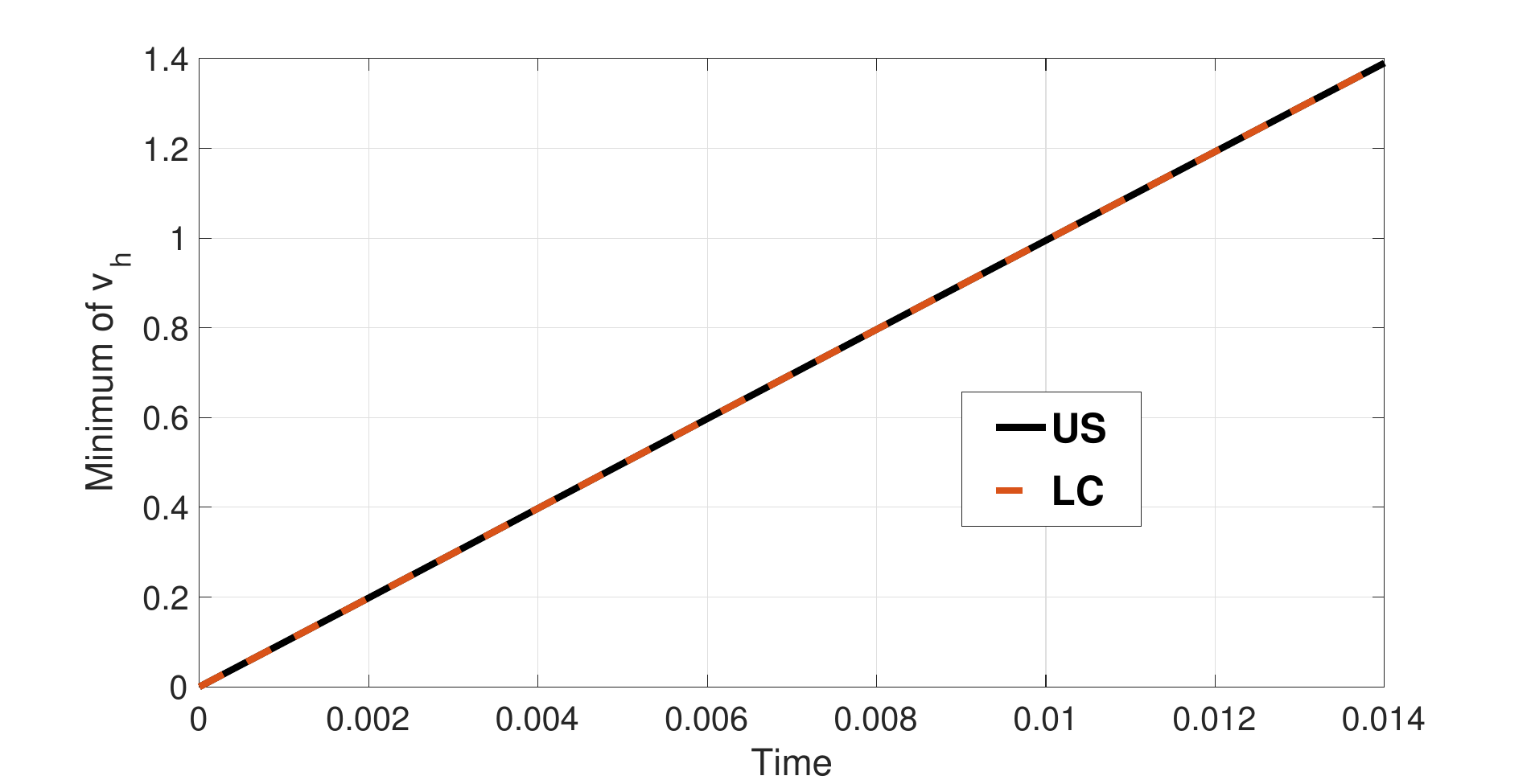}}
	\caption{Minimum of $u_h$ and $v_h$ computed by using the schemes \textbf{US} and \textbf{LC} with different values of $h$.} \label{fig:MU70}
\end{figure}
With this aim, we take $k=10^{-5}$ (in order to minimize the influence of the numerical dissipation terms $\frac{k}{2} 
\Vert \delta_t u_n\Vert_{0}^2$ and $\frac{k}{4} \Vert \delta_t {\boldsymbol \sigma}_n\Vert_{0}^2$), $h=\frac{1}{25}$ and the initial conditions 
$$u_0=-10xy\,(2-x)(2-y)\exp(-10(y-1)^2-10(x-1)^2)+10.0001$$ 
$$v_0=20xy\,(2-x)(2-y)\exp(-30(y-1)^2-30(x-1)^2)+0.0001.$$
We obtain that: 
\begin{enumerate}
\item[(a)] The schemes \textbf{US} and \textbf{LC} satisfy the energy decreasing in time property (\ref{stabf02}) for the modified energy $\mathcal{E}(u,{\boldsymbol \sigma})$, see Figure~\ref{fig:ENUS}(a).
\item[(b)] The schemes \textbf{US} and \textbf{LC} satisfy (\ref{ns01a}), see Figure~\ref{fig:ENUS}(b).
\item[(c)] The schemes \textbf{US}, \textbf{LC} and \textbf{UV} satisfy the energy decreasing in time property 
 for the exact energy $\mathcal{E}(u,v)$, that is, $\mathcal{E}(u_n,v_n)\le \mathcal{E}(u_{n-1},v_{n-1})$ for all $n$, see Figure~\ref{fig:ENUS}(c).
 \item[(d)] The schemes \textbf{US}, \textbf{LC} and \textbf{UV} have $RE(u_n,v_n)>0$ for some $t_n\geq 0$. However, it is observed that the  residual $RE(u_n,v_n)$ in the  schemes \textbf{US} and \textbf{LC} is much smaller than the  residual of the scheme  \textbf{UV}, see Figure~\ref{fig:ENUS}(d).
\end{enumerate}
\begin{figure}[htbp]
	\centering 
	\subfigure[$\mathcal{E}(u_n,{\boldsymbol\sigma}_n)$ in the schemes \textbf{US} and \textbf{LC}]{\includegraphics[width=77mm]{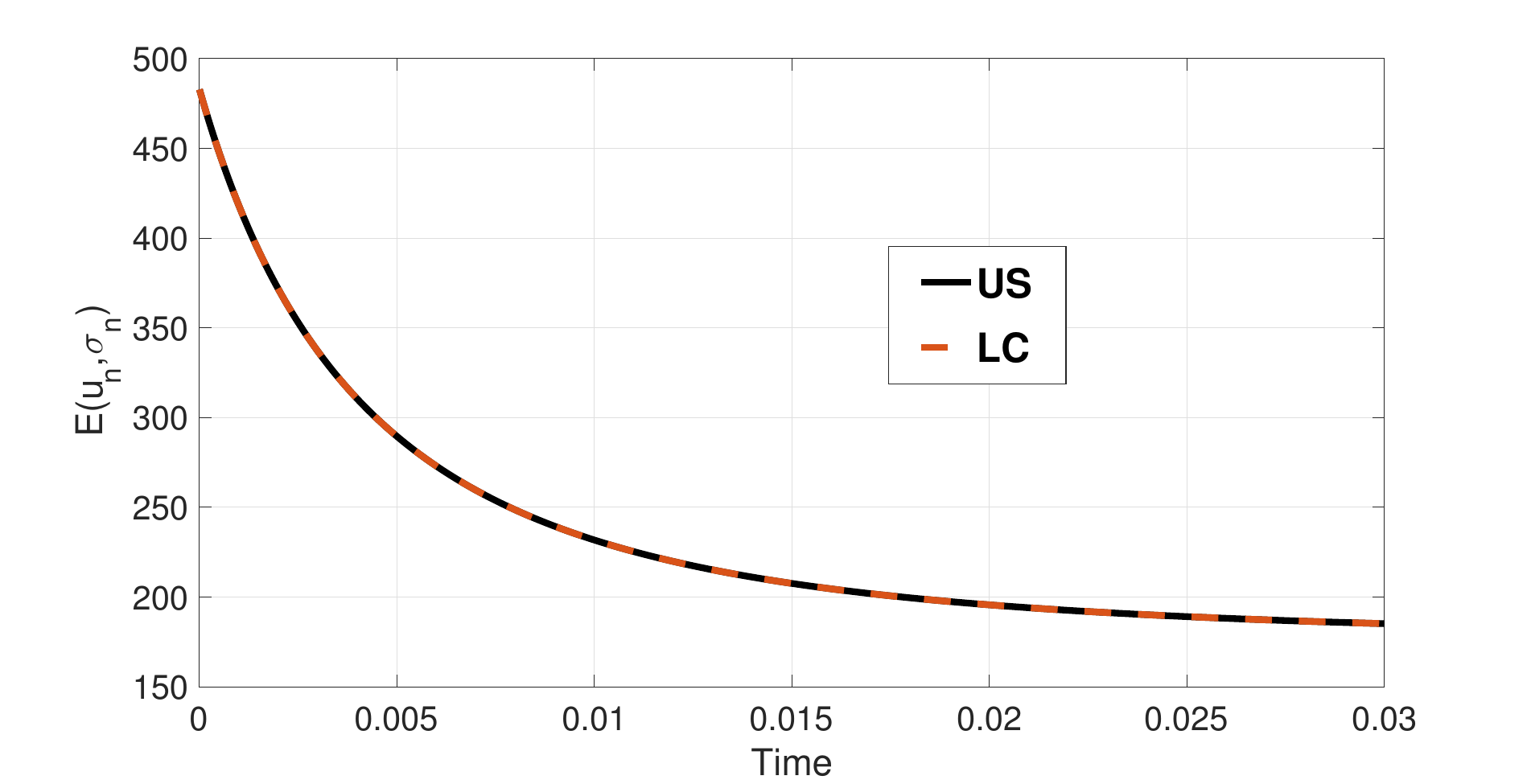}} \hspace{0.2 cm} 
	\subfigure[$RE(u_n,{\boldsymbol\sigma}_n)$ in the schemes \textbf{US} and \textbf{LC}]{\includegraphics[width=77mm]{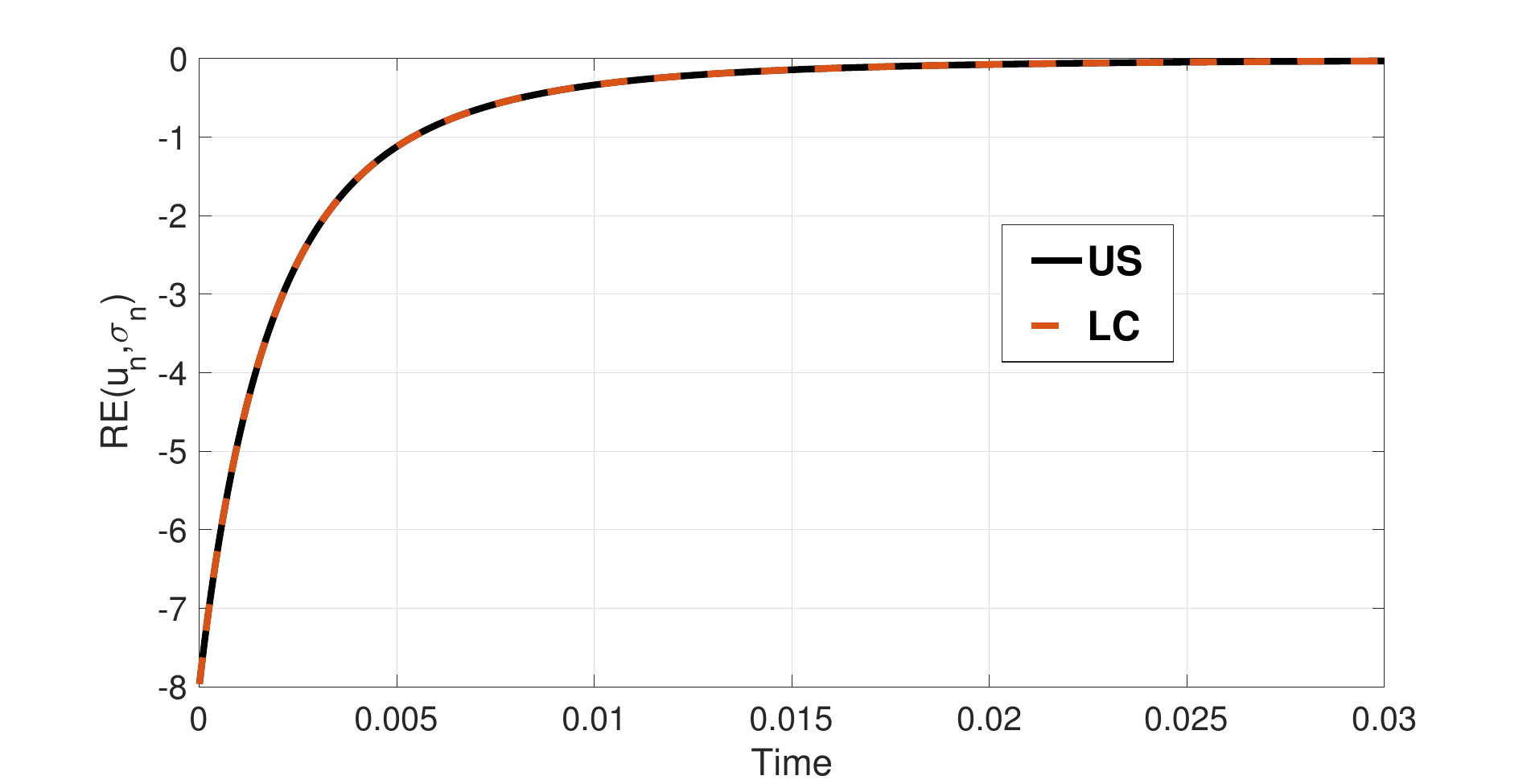}}
	\subfigure[$\mathcal{E}(u_n,v_n)$ in the schemes \textbf{US}, \textbf{LC} and \textbf{UV}]{\includegraphics[width=77mm]{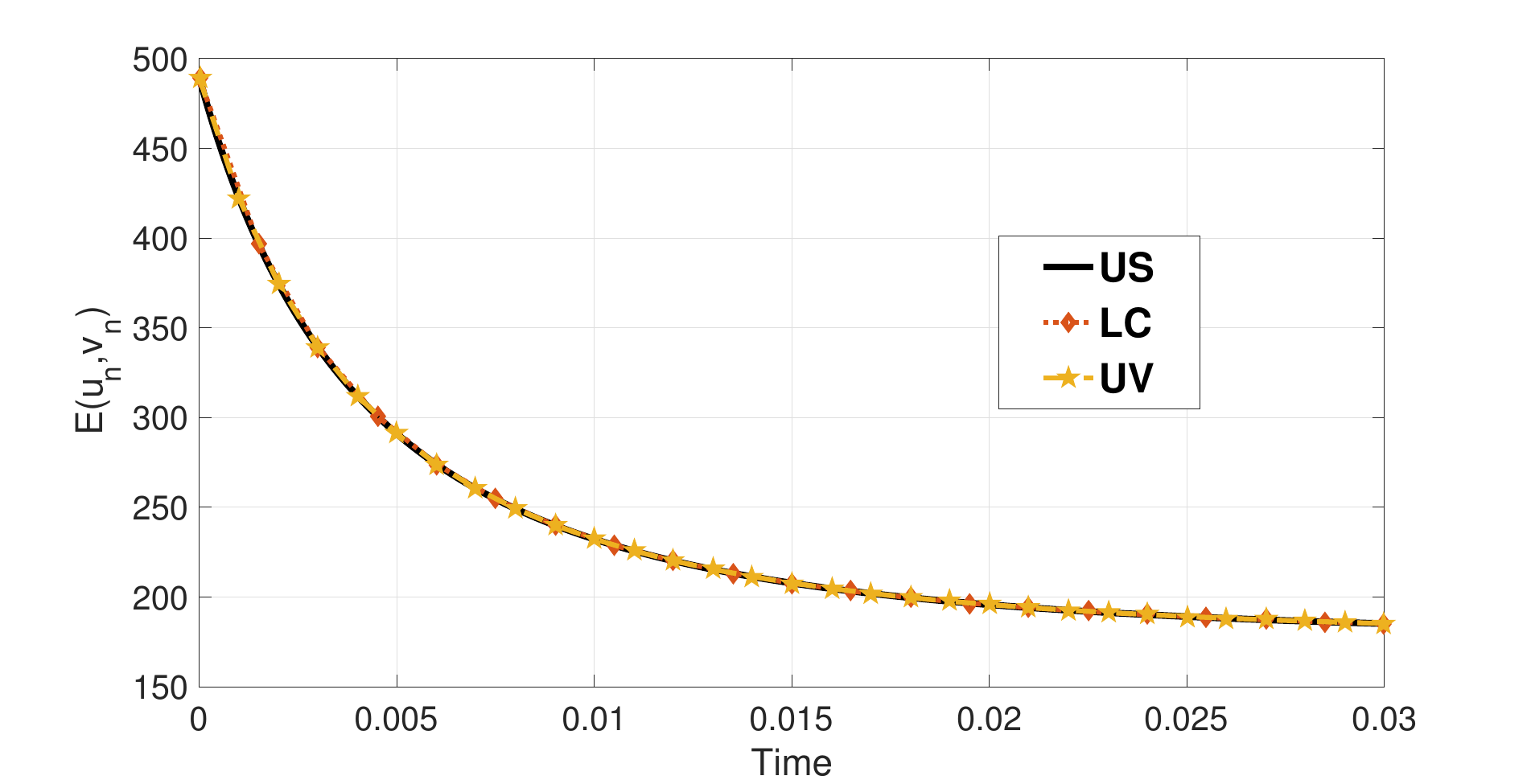}}\hspace{0.2 cm} 
	\subfigure[$RE(u_n,v_n)$ in the schemes \textbf{US}, \textbf{LC} and \textbf{UV}.]{\includegraphics[width=77mm]{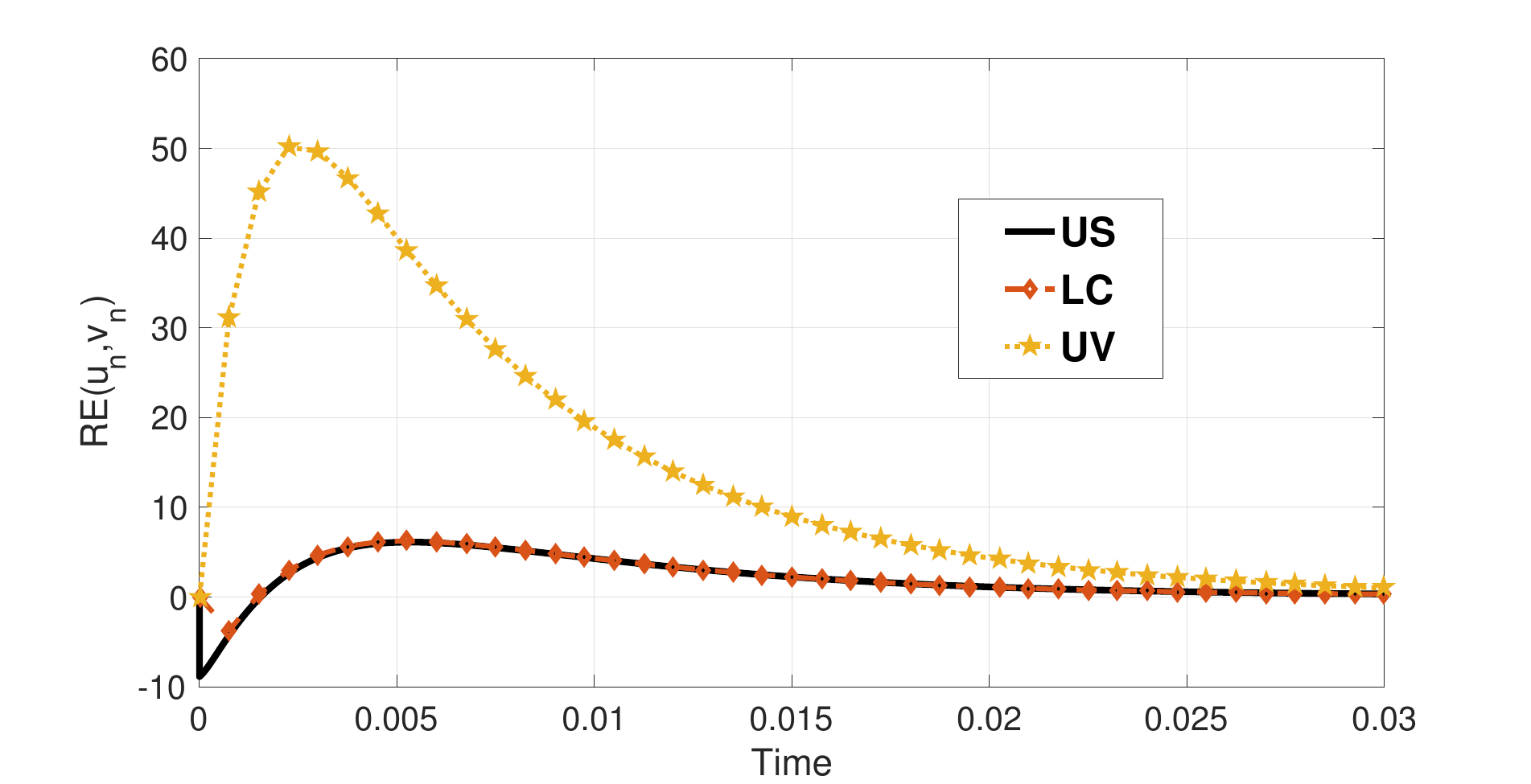}}
	\caption{Energy-stability of the schemes \textbf{US}, \textbf{LC} and \textbf{UV} for both energies $\mathcal{E}(u_n,{\boldsymbol\sigma}_n)$ and  $\mathcal{E}(u_n,v_n)$.} \label{fig:ENUS}
\end{figure}

\

\textbf{Appendix A}\\
%\appendix
In order to prove the solvability of (\ref{modelf02mod}), we will use the Leray-Schauder fixed point theorem. With this aim, we define the operator $R:L^4(\Omega)\times \L^4(\Omega)\rightarrow L^4(\Omega)\times \L^4(\Omega)$ by 
$R(\widetilde{u},\widetilde{\boldsymbol
\sigma})=(u,{\boldsymbol \sigma})$, such that $(u,{\boldsymbol \sigma})$ solves the following linear decoupled problems
\begin{equation}  \label{modelfexis01}
\begin{array}
[c]{lll}%
\displaystyle\frac{1}{k}(u,\overline{u}) + (\nabla u, \nabla\overline{u}) =\displaystyle\frac{1}{k}(u_{n-1},\overline{u}) -(\widetilde{u}_+\widetilde{\boldsymbol \sigma},\nabla \overline{u}), \ \ \forall \overline{u}\in {H}^{1}(\Omega),\vspace{0,2 cm}\\
\displaystyle\frac{1}{k}({\boldsymbol \sigma},\overline{\boldsymbol \sigma}) + \langle B
{\boldsymbol \sigma}, \overline{\boldsymbol \sigma}\rangle =\displaystyle\frac{1}{k}({\boldsymbol \sigma}_{n-1},\overline{\boldsymbol \sigma})
-(\widetilde{u}^2, \nabla\cdot \overline{\boldsymbol \sigma}), \ \ \forall \overline{\boldsymbol \sigma}\in  \H^{1}_{\sigma}(\Omega).
\end{array}
\end{equation}
\begin{enumerate}
\item{$R$ is well defined}. Let $(\widetilde{u},\widetilde{\boldsymbol \sigma})\in L^4(\Omega)\times \mathbf{L}^4(\Omega)$ and consider the following bilinear forms $\widetilde{a}:{H}^{1}(\Omega)\times
{H}^{1}(\Omega)\rightarrow \mathbb{R}$, $\widetilde{b}:\H^{1}_{\sigma}(\Omega)\times
\H^{1}_{\sigma}(\Omega)\rightarrow \mathbb{R}$, and the linear forms $l_1:{H}^{1}(\Omega)\rightarrow
\mathbb{R}$ and $l_2:\H^{1}_{\sigma}(\Omega)\rightarrow \mathbb{R}$ given by
$$\widetilde{a}(u,\overline{u})=\displaystyle\frac{1}{k} (u,\overline{u})+(\nabla u, \nabla\overline{u}), \quad \widetilde{b}({\boldsymbol\sigma},\overline{{\boldsymbol\sigma}})=\displaystyle\frac{1}{k} ({\boldsymbol\sigma},\overline{{\boldsymbol\sigma}})+\langle B {\boldsymbol\sigma},  \overline{\boldsymbol\sigma}\rangle,$$
$$
l_1(\overline{u})= \frac{1}{k}\int_\Omega
u_{n-1} \overline{u}- \int_\Omega \widetilde{u}_+\widetilde{\boldsymbol
\sigma}\cdot\nabla \bar{u} , \quad
l_2(\overline{{\boldsymbol\sigma}})= \frac{1}{k}\int_\Omega
{\boldsymbol\sigma}_{n-1} \overline{{\boldsymbol\sigma}} - \int_\Omega
\widetilde{u}^2\nabla \cdot \overline{{\boldsymbol\sigma}},
$$
for all $u,\overline{u}\in {H}^{1}(\Omega)$ and ${\boldsymbol\sigma},\overline{{\boldsymbol\sigma}}\in
\H^{1}_{\sigma}(\Omega)$. Then, using the H\"older inequality and Sobolev embeddings, we can verify that $\widetilde{a}$ and $\widetilde{b}$ are continuous and coercive on ${H}^{1}(\Omega)$ and $\H^{1}_{\sigma}(\Omega)$ respectively, and  $l_1\in  {H}^{1}(\Omega)'$ and $l_2\in \H^{1}_{\sigma}(\Omega)'$. Thus, from Lax-Milgram theorem, there exists a unique 
$(u,{\boldsymbol \sigma})\in H^1(\Omega)\times \H^1_{\sigma}(\Omega)\hookrightarrow L^4(\Omega)\times \L^4(\Omega)$
solution of (\ref{modelfexis01}).
\item{Now, let us prove that all possible fixed points of $\lambda R$ (with $\lambda \in (0,1]$) are bounded.}
In fact, observe that if $(u,{\boldsymbol \sigma})$ is a fixed point of $\lambda R$, then $(u,{\boldsymbol \sigma})$ satisfies the coupled problem
\begin{equation}\label{eq010}
\left\{
\begin{array}
[c]{lll}%
\displaystyle  \widetilde{a}( u, \overline{u}) =\displaystyle\frac{\lambda}{k}(u_{n-1},\overline{u}) 
-\lambda(u_+{\boldsymbol \sigma},\nabla \overline{u}),\ \ \forall \overline{u}\in {H}^{1}(\Omega), \vspace{0,2 cm}\\
\displaystyle \widetilde{b}({\boldsymbol \sigma},\overline{\boldsymbol \sigma}) =
\displaystyle\frac{\lambda}{k}({\boldsymbol \sigma}_{n-1},\overline{\boldsymbol \sigma})
-\lambda (u^2, \nabla\cdot \overline{\boldsymbol \sigma}), \ \ \forall \overline{\boldsymbol \sigma}\in  \H^{1}_{\sigma}(\Omega),
\end{array}
\right. 
\end{equation}
(because $\lambda R(u,{\boldsymbol \sigma}) = (u,{\boldsymbol \sigma})$ implies $R(u,{\boldsymbol \sigma}) = (\frac{1}{\lambda} u, \frac{1}{\lambda}{\boldsymbol \sigma})$). Proceeding as in Part A of the proof of Theorem \ref{USus},  it can be proved that if $(u,{\boldsymbol\sigma})$ is a solution of (\ref{eq010}), then $u\geq 0$, which implies that $u=u_+$. Then,  testing by $\overline{u}= u$ and $\overline{\boldsymbol\sigma}= \frac{1}{2}{\boldsymbol \sigma}$ in  (\ref{eq010})$_1$ and (\ref{eq010})$_2$, and taking into account that $\lambda \in (0,1]$,  one obtains
\begin{equation*}\label{acptofij01}
\displaystyle
\frac{1}{4}\Vert (u,{\boldsymbol\sigma})\Vert_{0}^2 +\displaystyle\frac{k}{2}\Vert (\nabla u, {\boldsymbol
\sigma})\Vert_{L^2\times H^1}^{2} \leq \displaystyle
C\lambda^2\Vert (u_{n-1},{\boldsymbol
\sigma}_{n-1})\Vert_{0}^{2} \le C(u_{n-1}, {\boldsymbol\sigma}_{n-1}).
\end{equation*}
Thus, we deduce that $\Vert (u,{\boldsymbol \sigma})\Vert_{L^4}\leq C\Vert (u,{\boldsymbol \sigma})\Vert_{1} \leq C(u_{n-1}, {\boldsymbol\sigma}_{n-1})$.
% where the constant $C$ depends on data $(\Omega, u_{n-1}, {\boldsymbol\sigma}_{n-1})$.
\item{We  prove that $R$ is continuous.} Let $\{(\widetilde{u}^l,\widetilde{\boldsymbol \sigma}^l)\}_{l\in\mathbb{N}}\subset {L}^4(\Omega)\times \L^4(\Omega)$ be a sequence such that 
\begin{equation}\label{c001}
(\widetilde{u}^l,\widetilde{\boldsymbol \sigma}^l)\rightarrow (\widetilde{u},\widetilde{\boldsymbol \sigma}) \ \mbox{ in }  {L}^4(\Omega)\times \mathbf{L}^4(\Omega),
\quad \hbox{as $l\to +\infty$}.
\end{equation}
Therefore, $\{(\widetilde{u}^l,\widetilde{\boldsymbol \sigma}^l)\}_{l\in\mathbb{N}}$ is bounded in ${L}^4(\Omega)\times \mathbf{L}^4(\Omega)$, and from item 1 we deduce that $\{(u^l,\boldsymbol \sigma^l)=R(\widetilde{u}^l,\widetilde{\boldsymbol \sigma}^l)\}_{l\in\mathbb{N}}$ is bounded in ${H}^1(\Omega)\times \H^1(\Omega)$. Then, there exists a subsequence $\{R(\widetilde{u}^{l^r},\widetilde{\boldsymbol \sigma}^{l^r})\}_{r\in\mathbb{N}}$  such that 
\begin{equation}\label{c002}
R(\widetilde{u}^{l^r},\widetilde{\boldsymbol \sigma}^{l^r})\rightarrow (u',{\boldsymbol\sigma}') \ \ \mbox{ weakly in } H^1(\Omega)\times \H^1(\Omega) \ \mbox{ and strongly  in } L^4(\Omega)\times \L^4(\Omega).
\end{equation}
 Thus, from (\ref{c001})-(\ref{c002}), a standard pass to the limit as $r\to +\infty$ in (\ref{modelfexis01}), allows to  deduce that $R(\widetilde{u},\widetilde{\boldsymbol \sigma})=({u}',{\boldsymbol \sigma}')$. Therefore,  any convergent subsequence of  $\{R(\widetilde{u}^l,\widetilde{\boldsymbol \sigma}^l)\}_{l\in\mathbb{N}}$ converges to $R(\widetilde{u},\widetilde{\boldsymbol \sigma})$ strongly in $L^4(\Omega)\times \L^4(\Omega)$.
 From uniqueness of $R(\widetilde{u},\widetilde{\boldsymbol \sigma})$,  one concludes that $R(\widetilde{u}^l,\widetilde{\boldsymbol \sigma}^l)\rightarrow R(\widetilde{u},\widetilde{\boldsymbol \sigma})$ in $L^4(\Omega)\times \L^4(\Omega)$. Thus, $R$ is continuous.
\item{$R$ is compact.}\label{3} In fact, if $\{(\widetilde{u}^l,\widetilde{\boldsymbol \sigma}^l)\}_{l\in\mathbb{N}}$ is a bounded sequence in ${L}^4(\Omega)\times \L^4(\Omega)$ and we denote $(u^l,\boldsymbol \sigma^l)=R(\widetilde{u}^l,\widetilde{\boldsymbol \sigma}^l)$, then we can deduce
\begin{equation*}
\label{fps02}
\displaystyle
\frac{1}{2k}\Vert (u^l, {\boldsymbol\sigma^l})\Vert_{0}^{2}+\frac{1}{2}\Vert (\nabla u^l, {\boldsymbol \sigma}^l)\Vert_{L^2\times H^1}^{2}\leq \displaystyle
\frac{1}{2k}\Vert (u_{n-1}, {\boldsymbol
\sigma}_{n-1})\Vert_{0}^{2} + \frac{1}{2}\Vert \widetilde{u}^l\Vert_{L^4}^2 \Vert \widetilde{\boldsymbol \sigma}^l \Vert_{L^4}^2 + \frac{1}{2}\Vert \widetilde{u}^l\Vert_{L^4}^4 \leq C,
\end{equation*}
where $C$ is independent of $l\in\mathbb{N}$. Therefore, we conclude that $\{R(\widetilde{u}^l,\widetilde{\boldsymbol \sigma}^l)\}_{l\in\mathbb{N}}$ is bounded in ${H}^1(\Omega)\times \H^1(\Omega)$ which is compactly embedded in ${L}^4(\Omega)\times \L^4(\Omega)$, and thus $R$ is compact.
\end{enumerate}
 The hypotheses of the Leray-Schauder fixed point theorem are then satisfied, and the existence of a fixed point for the map $R$ is proved. 
This fixed point
$R(u,{\boldsymbol \sigma})=(u,{\boldsymbol \sigma})$ is a solution of 
(\ref{modelf02mod}). 

\section*{Acknowledgements}
The authors have been partially supported by MINECO grant MTM2015-69875-P
(Ministerio de Econom\'{\i}a y Competitividad, Spain) with the participation of FEDER.
The third author have also been supported by Vicerrector\'ia de Investigaci\'on y Extensi\'on of Universidad
Industrial de Santander.

\end{document}